\documentclass[11pt]{amsart}

\usepackage[english]{babel}     
\usepackage[utf8]{inputenc}    
\usepackage[sort]{cite}           
\usepackage{fancyhdr}          
\usepackage[a4paper, left=1.2in]{geometry} 
\usepackage{xspace}            
\usepackage{dsfont}
\usepackage{amsxtra}
\usepackage{amsmath}
\usepackage{amscd}
\usepackage{amssymb}
\usepackage{amsfonts}
\usepackage[all]{xy}
\usepackage{mathrsfs}
\usepackage{amsthm} 
\usepackage{color}
\usepackage{upgreek}
\usepackage{enumerate}
\usepackage{latexsym}
\usepackage{todonotes}
\usepackage{hyperref}
\usepackage{stmaryrd}
\usepackage{nicefrac}
\usepackage{euscript}
\usepackage{mathabx}
\usepackage{enumitem}
\usepackage[foot]{amsaddr}

\newcommand{\arxiv}[1]{\href{http://arxiv.org/abs/#1}{\tt
    arXiv:\nolinkurl{#1}}}

\newtheorem{theorem}{\bf{Theorem}}[subsection]
\newtheorem{lemma}[theorem]{Lemma}
\newtheorem{corollary}[theorem]{Corollary}
\newtheorem{proposition}[theorem]{Proposition}
\newtheorem{claim}[theorem]{Claim}
\newtheorem{conjecture}[theorem]{Conjecture}
\newtheorem{definition}[theorem]{Definition}
\newtheorem*{theorem-A}{Theorem A}
\newtheorem*{theorem-B}{Theorem B}
\newtheorem*{theorem-C}{Theorem C}
\newtheorem*{theorem-D}{Theorem D}
\newtheorem*{conjecture-A}{Conjecture A}
\newtheorem*{conjecture-B}{Conjecture B}
\newtheorem*{conjecture-C}{Conjecture C}
\newtheorem*{conjecture-D}{Conjecture D}

\theoremstyle{remark} 

\newtheorem{remark}[theorem]{Remark}
\newtheorem{example}[theorem]{Example}

\def\A{\mathrm{A}}
\def\B{\mathrm{B}}
\def\C{\mathrm{C}}
\def\D{\mathrm{D}}
\def\E{\mathrm{E}}
\def\F{\mathrm{F}}
\def\G{\mathrm{G}}

\def\K{\mathrm{K}}
\def\L{\mathrm{L}}
\def\M{\mathrm{M}}

\def\P{\mathrm{P}}

\def\SS{\mathrm{S}}

\def\U{\mathrm{U}}

\def\Z{\mathrm{Z}}

\def\bbA{\mathbb{A}}

\def\bbC{\mathbb{C}}

\def\bbE{\mathbb{E}}

\def\bbG{\mathbb{G}}
\def\bbH{\mathbb{H}}

\def\bbN{\mathbb{N}}

\def\bbP{\mathbb{P}}
\def\bbQ{\mathbb{Q}}

\def\bbZ{\mathbb{Z}}

\def\frakg{\mathfrak{G}}

\def\frakL{\mathfrak{L}}

\def\calA{\mathcal{A}}
\def\calB{\mathcal{B}}
\def\calC{\mathcal{C}}
\def\calD{\mathcal{D}}
\def\calE{\mathcal{E}}
\def\calF{\mathcal{F}}
\def\calG{\mathcal{G}}
\def\calH{\mathcal{H}}
\def\calI{\mathcal{I}}

\def\calL{\mathcal{L}}
\def\calM{\mathcal{M}}
\def\calN{\mathcal{N}}
\def\calO{\mathcal{O}}
\def\calP{\mathcal{P}}
\def\calQ{\mathcal{Q}}
\def\calR{\mathcal{R}}
\def\calS{\mathcal{S}}
\def\calT{\mathcal{T}}
\def\calU{\mathcal{U}}

\def\calX{\mathcal{X}}
\def\calY{\mathcal{Y}}
\def\calZ{\mathcal{Z}}

\def\frakb{\mathfrak{b}}

\def\frakg{\mathfrak{g}}

\def\frakl{\mathfrak{l}}

\def\frakn{\mathfrak{n}}

\def\frakp{\mathfrak{p}}

\def\bfa{\mathbf{a}}

\def\bfl{\mathbf{l}}

\def\bfr{\mathbf{r}}

\def\bfA{\mathbf{A}}
\def\bfB{\mathbf{B}}
\def\bfC{\mathbf{C}}

\def\bfF{\mathbf{F}}
\def\bfG{\mathbf{G}}

\def\bfL{\mathbf{L}}

\def\bfQ{\mathbf{Q}}
\def\bfR{\mathbf{R}}
\def\bfS{\mathbf{S}}
\def\bfT{\mathbf{T}}
\def\bfU{\mathbf{U}}

\def\k{{\operatorname{k}\nolimits}}

\def\For{\operatorname{For}\nolimits}
\def\ind{\operatorname{ind}\nolimits}
\def\Isom{\operatorname{Isom}\nolimits}
\def\mod{\operatorname{mod}\nolimits}

\def\proj{{\operatorname{proj}\nolimits}}
\def\tilt{{\operatorname{tilt}\nolimits}}
\def\fl{{\operatorname{fl}\nolimits}}

\def\KP{\operatorname{KP}\nolimits}

\def\dgmod{\operatorname{dgmod}\nolimits}

\def\Coh{{\calC oh}}
\def\PCoh{{\calP\calC oh}}

\def\QCoh{\operatorname{\calQ\calC oh}\nolimits}

\def\Fr{{\operatorname{Fr}\nolimits}}
\def\nrk{{\operatorname{\rho}\nolimits}}

\def\D{{\operatorname{D}\nolimits}}
\def\b{{\operatorname{b}\nolimits}}

\def\e{{\operatorname{e}\nolimits}}

\def\mix{{\mu}}
\def\dg{{\operatorname{dg}\nolimits}}
\def\bbd{{\operatorname {m}}}

\def\perf{{\operatorname{perf}\nolimits}}
\def\re{{\operatorname{re}\nolimits}}

\def\op{{{\operatorname{op}\nolimits}}}
\def\Ad{{{\operatorname{Ad}\nolimits}}}
\def\-{{\operatorname{-}\!}}
\def\hgt{{\operatorname{ht}\nolimits}}

\def\Hom{\operatorname{Hom}\nolimits}
\def\bbHom{\operatorname{\bbH om}\nolimits}
\def\HHom{{\underline{{\Hom}}}}

\def\Par{\operatorname{C}\nolimits}

\def\Pure{\operatorname{C_\bbd}\nolimits}
\def\RHom{\operatorname{RHom}\nolimits}
\def\bbEnd{\operatorname{\bbE nd}\nolimits}
\def\End{\operatorname{End}\nolimits}
\def\coker{\operatorname{coker}\nolimits}

\def\Res{\operatorname{Res}\nolimits}
\def\Ind{\operatorname{Ind}\nolimits}
\def\Ext{\operatorname{Ext}\nolimits}
\def\bbExt{\operatorname{\bbE xt}\nolimits}

\def\Inv{\operatorname{Inv}\nolimits}

\def\Add{\operatorname{Add}\nolimits}
\def\Spec{\operatorname{Spec}\nolimits}

\def\Gr{{\operatorname{Gr}\nolimits}}
\def\Fl{{\operatorname{Fl}\nolimits}}
\def\bfGr{{\operatorname{\bfG\bfr}\nolimits}}

\def\Fr{\mathrm{Fr}}

\def\res{\mathrm{res}}

\numberwithin{itemcounter}{subsection}
\numberwithin{equation}{section}
\title{Coherent categorification of quantum loop algebras : the $SL(2)$ case}

\author{P. Shan$^1$} 
\address{\scriptsize{$^1$~Yau Mathematical Sciences Center, Tsinghua University, 100084, Beijing, China.}}
\author{M. Varagnolo$^2$} 
\address{\scriptsize{$^2$~D\'epartement de Math\'ematiques, Universit\'e Cergy-Pontoise, 95302 Cergy-Pontoise, France,
UMR8088 (CNRS), ANR-18-CE40-0024.}}
\author{E. Vasserot$^3$} 
\address{\scriptsize{$^3$~Universit\'e de Paris, 75013 Paris, France, UMR7586 (CNRS), ANR-18-CE40-0024,
Institut Universitaire de France (IUF).
}}

\begin{document}
\maketitle

\begin{abstract}
We construct an equivalence of graded Abelian categories from a category of representations of the 
quiver-Hecke algebra of type $A_1^{(1)}$ to the category of equivariant perverse coherent sheaves on the nilpotent 
cone of type $A$. We prove that this equivalence is weakly monoidal. This gives a representation-theoretic categorification
of the preprojective K-theoretic Hall algebra considered by Schiffmann-Vasserot.
Using this categorification, we compare the
monoidal categorification of the quantum open unipotent cells of type $A_1^{(1)}$ given by Kang-Kashiwara-Kim-Oh-Park in terms of quiver-Hecke
algebras with the one given by Cautis-Williams in terms of equivariant perverse coherent 
sheaves on the affine Grassmannians. \end{abstract}

\tableofcontents

\section{Introduction}

\subsection{Main results of the paper}
Let $\bfQ=(I,\Omega)$ be a quiver of Kac-Moody type.
For each dimension vector $\beta$, let $X_\beta$ be the variety of all $\beta$-dimensional 
representations of the path algebra of $\bfQ$ and $\calX_\beta$ be the corresponding moduli stack.
Lusztig introduced a graded additive subcategory $\C(\calX_\beta)$ of the bounded constructible derived category 
$\D^\b(\calX_\beta)$ whose split Grothendieck group is isomorphic 
to the $\beta$-weight subspace of the quantum unipotent enveloping algebra $\bfU_q(\frakn)$ of type $\bfQ$.
Then, Khovanov-Lauda and Rouquier showed that $\bfU_q(\frakn)$ is the 
Grothendieck group of the monoidal category of all projective graded
modules over the quiver-Hecke algebra $\bfR$ of $\bfQ$.
According to Rouquier and Varagnolo-Vasserot, 
this isomorphism lifts to an equivalence of additive graded monoidal categories
between $\bigoplus_\beta\C(\calX_\beta)$, equipped with Lusztig's geometric induction 
bifunctor and the grading given by the cohomological shift,
and the category $\calC^\proj$
of all finitely generated graded projective $\bfR$-modules, equipped with the algebraic induction bifunctor $\circ$.
Further, the quantum unipotent coordinate algebra $\bfA_q(\frakn)$ of $\bfQ$ is isomorphic to the Grothendieck group
of the monoidal category of all finite dimensional graded modules over $\bfR$, in such a way that
the irreducible self-dual graded $\bfR$-modules are identified with the dual canonical basis elements in $\bfA_q(\frakn)$. 


Now, let us assume that the quiver $\bfQ$ is of type $A_1^{(1)}$. 
The graded category $\calC$ of all finitely generated graded $\bfR$-modules is affine properly stratified.
The simple self-dual modules are labelled by Kostant partitions of a dimension vector $\beta=n\alpha_1+m\alpha_0$ 
of $\bfQ$, for some non-negative integers $n$, $m$.
There is a monoidal structure on $\calC$ given by the algebraic induction bifunctor $\circ$.
Let $\calD$ be the full graded subcategory of $\calC$ consisting of all objects whose composition factors are
graded shifts of the simple self-dual modules labelled by the 
Kostant partitions supported on the set of all real positive roots of $\bfQ$ of the form
$\alpha_0+n\delta$ with $n\in\bbN$. 
It is a graded monoidal subcategory of $\calC$ which is
polynomial highest weight.

\smallskip

For any integer $r\geqslant 0$, the Lie algebra $\frakg\frakl_r$ carries the adjoint action of $GL_r$ and a $\bbG_m$-action by dilation. Write $GL_r^c=GL_r\times \bbG_m$ and consider the quotient stack $[\frakg\frakl_r/GL_r^c]$.
Let $\D^\b\Coh([\frakg\frakl_r/GL_r^c])$ be the bounded derived category of all
$GL_r^c$-equivariant finitely generated graded $\SS(\frakg\frakl_r^*\langle-2\rangle)$-modules, 
where $\langle\bullet\rangle$ is the grading shift functor.
Let $\D^\b\Coh([\frakg/G^c])_{\Lambda^+}$ be the graded triangulated subcategory of the direct sum
$$\D^\b\Coh([\frakg/G^c])=\bigoplus_{r\in\bbN}\D^\b\Coh([\frakg\frakl_r/GL_r^c])$$
generated by all objects of the form $(V\otimes\calO_{\frakg\frakl_r})\langle a\rangle$ where 
$V$ is a polynomial representation of $GL_r$ and $a$ is an arbitrary integer.
There is a graded triangulated
monoidal structure on $\D^\b\Coh([\frakg/G^c])_{\Lambda^+}$ given by the convolution bifunctor
$\circ$.

\smallskip

In this paper, we study relations between (variants of) the category $\calD$ and $G^c$-equivariant coherent sheaves on $\frakg$. In particular,  we propose the following conjecture, see Conjecture \ref{conj:1} and Remark \ref{rem:rem-conj} below.

\smallskip

\begin{conjecture-A} There is an equivalence of triangulated graded monoidal categories
$$\E:(\D^\b(\calD)\,,\,\circ)\to(\D^\b\Coh([\frakg/G^c])_{\Lambda^+}\,,\,\circ^\op).$$
\end{conjecture-A}

\smallskip

Our main result is a proof of a slightly modified version of this conjecture.
To explain this, let $\D^\b\Coh([\calN/G^c])_{\Lambda^+}$ denote the triangulated subcategory of 
$\D^\b\Coh([\frakg/G^c])_{\Lambda^+}$ consisting of all complexes of coherent sheaves supported
on the nilpotent cone of $\frakg\frakl_r$ for some $r\geqslant 0$. 
This triangulated category is equipped with the perverse t-structure
whose heart is denoted by $\PCoh([\calN/G^c])_{\Lambda^+}$.
The convolution  yields a graded Abelian monoidal structure on $\PCoh([\calN/G^c])_{\Lambda^+}$.
Our main theorem, Theorem \ref{thm:2}, is the following analogue of Conjecture A.

\smallskip

\begin{theorem-B} There is a graded Abelian and Artinian monoidal subcategory $\calD^\sharp$ of $\calD$ 
containing all simple objects, with an equivalence of graded Abelian categories
$$\E^\sharp:\calD^\sharp\to\PCoh([\calN/G^c])_{\Lambda^+}.$$
\end{theorem-B}

\smallskip

Moreover, both categories $\calD^\sharp$ and $\PCoh([\frakg/G^c])_{\Lambda^+}$ are
graded stratified, and we prove that the equivalence $\E^\sharp$ respects these structures.
In particular it takes proper standard modules to proper standard ones.
Note that the proper standard modules are monomials in the simple ones, this statement can be viewed as a weak form of 
the monoidality of the functor $\E^\sharp$, see Remark \ref{rem:rem-conj}. 
To keep this paper in a reasonable length, we do not prove here that $\E^\sharp$ is a monoidal equivalence
$$(\calD^\sharp\,,\,\circ)\to(\PCoh([\calN/G^c])_{\Lambda^+}\,,\,\circ^\op).$$
This is stated as Conjecture \ref{conj:monoidal}. We will prove it in a sequel paper \cite{RSVV}.

\smallskip

The proof of Theorem B consists of constructing
a chain of graded triangulated equivalences 
\begin{align*}
\xymatrix{
\D^\perf(\calD^{\,\sharp})\ar[r]^-{\A^\sharp}&
\D^\b_\mix(\Gr^+_{\Lambda^+},S)\ar[r]^-{\B^\sharp}&\D^\b_\mix(\Gr^-_{\Lambda^+},S)\ar[r]^-{\C^\sharp}&\D^\perf\Coh([\calN/G^c])_{\Lambda^+},
}
\end{align*}
and then checking the t-exactness. Here, the two categories in the middle are some mixed categories on think/thin affine Grassmannians, the functor $\A^\sharp$
is given by the composition of some localization of quiver-Hecke algebras and the tilting equivalence
between the module category of the Kronecker quiver and coherent sheaves on $\bbP^1$.
The functor $\B^\sharp$ is a Radon transform. 
The functor $\C^\sharp$ is the derived geometric Satake equivalence. 
Both $\A^\sharp$ and $\B^\sharp$ use mixed geometry.

\smallskip

One of our motivations comes from the recent work of Cautis-Williams \cite{CW18}.
To explain the link, 
let $\bfU_q(\frakn)_{\Lambda^+}$ be the subalgebra of the quantum unipotent enveloping algebra 
$\bfU_q(\frakn)$ generated by the root vectors whose weights are of the form 
$\alpha_0+n\delta$ for some integer $n\geqslant 0$.
The theorem of Khovanov-Lauda and Rouquier implies that the split Grothendieck group
$\K_0(\calD^\proj)$ is isomorphic to $\bfU_q(\frakn)_{\Lambda^+}$.
There is a perfect pairing between $\K_0(\calD^\proj)$ and $\G_0(\calD^\sharp)$.
Hence, there is a ring 
isomorphism between $\G_0(\calD^\sharp)$ and the quantum unipotent coordinate algebra
$\bfA_q(\frakn)_{\Lambda^+}$ of $\bfU_q(\frakn)_{\Lambda^+}$.
Now, given a positive integer $N$, let  $\bfA_q(\frakn^w)$ be the quantum open unipotent cell of type $\bfQ$ 
associated with the element $w=(s_0s_1)^N$ in the Weyl group of $\bfQ$.
This quantum open unipotent cell is a localization of the quantum unipotent coordinate subalgebra
$\bfA_q(\frakn(w))$ of $\bfA_q(\frakn)_{\Lambda^+}$.
It admits a quantum cluster algebra structure.
Cautis-Williams proved that the category of $GL_N(\calO)\ltimes\bbG_m$-equivariant perverse coherent sheaves 
$\PCoh([\Gr/G_N^c(\calO)])$ on the affine Grassmanian of $GL_N$,
with the monoidal structure given by the convolution product, is a monoidal categorification of $\bfA_q(\frakn^w)$.

\smallskip

In \cite{KKKO18}, \cite{KKOP19} a localization $\widetilde\calC_w^\fl$ of a graded monoidal Serre subcategory
of $\calC$ is introduced. It is proved there that $\widetilde\calC_w^\fl$ 
is also a monoidal categorification $\bfA_q(\frakn^w)$. It is natural to compare it with the monoidal categorification
of Cautis-Williams.
To do that we introduce a localization $\widetilde\calD^\sharp_w$ of a Serre subcategory 
$\calD^\sharp_w$ of $\calD^\sharp$ by mimicking
the construction in \cite{KKOP19}.
The equivalence $\E^\sharp$ yields a faithful graded exact functor 
$$\widetilde\Psi_w:\widetilde\calD^\sharp_w\to\PCoh([\Gr/G_N^c(\calO)])$$
which induces an isomorphism of Grothendieck groups
$\G_0(\widetilde\calD^\sharp_w)=\G_0(\PCoh([\Gr/G_N^c(\calO)])).$
Our  Conjecture \ref{conj:D} is the following.

\begin{conjecture-C} The functor $\widetilde\Psi_w$ is a graded monoidal equivalence of categories.
\end{conjecture-C}

\smallskip

\subsection{Background and perspectives}
Let $\Pi_\bfQ$ be the preprojective algebra of the quiver $\bfQ$.
A geometric construction of affine quantum groups is given by the K-theoretic Hall algebra of the category of 
$\Pi_\bfQ$-modules considered by Schiffmann-Vasserot. 
There, the category of
constructible sheaves on the moduli stack of representations of $\bfQ$ is replaced by the category
of coherent sheaves on the 
derived moduli stack of representations of the preprojective algebra.
Our goal is to compare this category of coherent sheaves with a 
module category of the quiver-Hecke algebra of affine type $\bfQ^{(1)}$ when $\bfQ$ is of finite type.

\smallskip

The stack $\calX_\beta$ is the quotient of the variety $X_\beta$ by a linear group $G_\beta$.
The group $G_\beta^c=G_\beta\times\bbG_m$ acts on $T^*X_\beta$ so that  $\bbG_m$ has weight 1.
The $G_\beta$-action  is Hamiltonian.
The moment map $\mu_\beta:T^*X_\beta\to \frakg_\beta$ is $G_\beta^c$-equivariant,
with $\bbG_m$ of weight $2$ on $\frakg_\beta$. 
The cotangent dg-stack of $\calX_\beta$ is the derived fiber product 
$$T^*\calX_\beta^c=[T^*X_\beta\times^R_{\frakg_\beta}\{0\}\,/\,G_\beta^c].$$
The truncation of this dg-stack is the moduli stack $\Lambda_\beta^c=[\mu_\beta^{-1}(0)\,/\,G_\beta^c]$
of $\beta$-dimensional representations of  $\Pi_\bfQ$.
Let $\text{dg}\!\QCoh(T^*\calX_\beta)$ be the category of all $G_\beta^c$-equivariant sheaves of dg-modules over 
$T^*X_\beta\times^R_{\frakg_\beta}\{0\}$ whose cohomology is a quasi-coherent sheaf over $\Lambda_\beta^c$. 
Let $\D(\text{dg}\!\QCoh(T^*\calX_\beta^c))$ be its derived category. 
Since the dg-stack $T^*\calX_\beta^c$ is affine, 
this graded triangulated category is the derived category of $G_\beta^c$-equivariant modules
over a $G_\beta^c$-equivariant dg-algebra which can be described as follows.

\smallskip

Recall that a $G_\beta^c$-equivariant dg-algebra is a $G_\beta$-equivariant  
$\bbZ^2$-graded  algebra with a differential of bi-degree
$(1,0)$ satisfying the Leibniz rule. For each bi-degree $(i,j)$, we call $i$ the cohomological degree and
$j$ the internal degree. The corresponding grading shift functors are denoted by 
$[\bullet]$ and $\langle\bullet\rangle$ respectively.
The internal degree is the weight of the $\bbG_m$-action.
For any $\bbZ^2$-graded vector space $V$, let
$\SS(V)$ 
be the graded-symmetric 
algebra of $V$, 
i.e., the quotient of the tensor algebra by the relations $x\otimes y-(-1)^{|x|\,|y|}y\otimes x$.
Here $|x|$ is the cohomological degree an homogeneous element $x$.
The graded triangulated category $\D(\text{dg}\!\QCoh(T^*\calX_\beta^c))$ is equivalent to 
the derived category of all $G_\beta^c$-equivariant dg-modules 
over a $G_\beta^c$-equivariant graded-commutative non-positively graded dg-algebra $\bfC_\beta$ 
whose underlying $G_\beta$-equivariant $\bbZ^2$-graded algebra is
$\SS(T^*X_\beta\langle 1\rangle\oplus\frakg_\beta^*[1]\langle 2\rangle).$ 
To describe the differential on $\bfC_\beta$, we consider the $\bbZ$-graded Lie superalgebra 
$$\frakL_\beta=T^*X_\beta[-1]\langle -1\rangle\oplus \frakg_\beta[-2]\langle -2\rangle,$$ 
whose bracket is the extension by 0 of the map
\begin{align}\label{diff}\SS^2(T^*X_\beta)\to\frakg_\beta\quad,\quad
a\otimes b\mapsto\sum_{h\in\Omega}([a_h\,,\,b_{h^\op}]+[b_h\,,\,a_{h^\op}]).\end{align}
Then $\bfC_\beta$ is the Chevalley-Eilenberg complex which computes the extension group
$\Ext^\bullet_{\frakL_\beta}(\k\,,\,\k)$. As a 
$\bbZ^2$-graded algebra we have 
$$\bfC_\beta=\SS(\frakL^*_\beta[-1]) =\SS(T^*X_\beta\langle 1\rangle\oplus \frakg_\beta^*[1]\langle 2\rangle).$$ 
The differential is the unique derivation which vanishes on the subspace
$T^*X_\beta$ and is given on $\frakg_\beta^*$ by the map
$\frakg_\beta\to \SS^2(T^*X_\beta)$ dual to \eqref{diff}.
Let  $\D(\text{dg}\Coh(T^*\calX_\beta^c))$ 
denote the derived category  $\D(\dgmod(\bfC_\beta\rtimes G^c_\beta))$
of all $G_\beta^c$-equivariant  
dg-modules of $\bfC_\beta$ whose cohomology is finitely generated over the graded algebra $H^\bullet(\bfC_\beta)$.
Note that $H^\bullet(\bfC_\beta)$ is a finitely generated nilpotent extension of $H^0(\bfC_\beta)$, and that the latter
is isomorphic to the function ring over the variety $\mu_\beta^{-1}(0)$.

\smallskip

Now, fix a triple of dimension vectors $\alpha$, $\beta$, $\gamma$ such that $\beta=\alpha+\gamma$.
Fix a parabolic subgroup $P$ of $G_\beta$ with a Levi subgroup $L$ isomorphic to $G_\alpha\times G_\gamma$. 
A $\bbC$-point of $T^*X_\beta$ is the same as a representation of the double quiver
$\bar\bfQ=(I\,,\,\bar\Omega)$ such that $\bar\Omega=\Omega\sqcup\Omega^\op$.
Let $T^*X_P$ be the subspace of all representations in $T^*X_\beta$ which preserve a fixed flag of vector spaces of type
$(\alpha,\gamma)$ whose stabilizer in $G$ is $P$. 
One defines as above a $P^c$-equivariant graded-commutative dg-algebra structure 
on the symmetric graded-commutative algebra 
$$\bfC_P=\SS(T^*X_P\langle 1\rangle\oplus\frakp^*[1]\langle2\rangle).$$ 
The obvious projection and inclusion
$$T^* X_L\oplus\frakl^*\to T^*X_P\oplus\frakp^*\quad,\quad 
T^*X_G\oplus\frakg^*\to T^* X_P\oplus\frakp^*$$ 
yield $P^c$-equivariant dg-algebra homomorphisms
$q:\bfC_L\to\bfC_P$ and $p:\bfC_G\to\bfC_P.$ Composing the extension of scalars relative to $q$,
the restriction of scalars relative to $p$, the restriction from $L^c$ to $P^c$ and the induction from $P^c$ to $G^c$,
we get a triangulated functor
$$R_{L\subset P}^{G}:\D(\text{dg}\Coh(T^*\calX_\alpha^c))\times
\D(\dg\Coh(T^*\calX_\gamma^c))\to\D(\text{dg}\Coh(T^*\calX_\beta^c))$$
which yields a triangulated graded monoidal structure $\circ$ on the category
$$\D(\dg\Coh(T^*\calX^c))=\bigoplus_\beta\D(\dg\Coh(T^*\calX_\beta^c)).$$

\smallskip

Let $\G_0(\dg\Coh(T^*\calX^c))$ be the Grothendieck group of the graded triangulated category 
$\D(\dg\Coh(T^*\calX^c))$.
It coincides with the Grothendieck group of the graded Abelian category of coherent sheaves on the stack
$\bigsqcup_\beta\Lambda_\beta.$
We equip it with the multiplication given by the monoidal structure $\circ$.
This ring is  the K-theoretic Hall algebra of the category of $\Pi_\bfQ$-modules.
It was considered by Schiffmann-Vasserot in \cite{SV13}, \cite{SV12} for a one vertex quiver $\bfQ$.
Then it was generalized to several other settings, see \cite{SV} and the references there.
Let $\bfU_q(L\frakn)$ be the quantum enveloping algebra of the loop algebra of $\frakn$.
The following theorem holds. The proof will be given elsewhere.

\smallskip

\begin{theorem-D}
Assume that the quiver $\bfQ$ is of finite type.
Then, there is a $\bbZ[q,q^{-1}]$-algebra isomorphism $\G_0(\dg\Coh(T^*\calX^c))=\bfU_q(L\frakn)$.
\end{theorem-D}

\smallskip

The conjecture A is a lift of the isomorphism in Theorem D to a monoidal graded triangulated equivalence between
$\D(\text{dg}\Coh(T^*\calX^c))$ and the derived category of a module 
category of the quiver-Hecke algebra of affine type $\bfQ^{(1)}$
for $\bfQ$ of type $A_1$. Let us explain this in more details.
The general case will be considered elsewhere.

\smallskip

From now on, let us assume that the quiver $\bfQ=A_1$. The dimension vector $\beta$ is simply
a non-negative integer $r$.
The group $G_\beta$ is $GL_r$, the graded Lie superalgebra $\frakL_\beta$ is 
$\frakg\frakl_r^*[-2]\langle-2\rangle$ with the zero bracket,
and the dg-algebra $\bfC_\beta$ is the exterior algebra 
$\SS(\frakg\frakl_r[1]\langle 2\rangle)$ with the zero differential. 
In particular $\D(\dg\Coh(T^*\calX_\beta^c))$ is the triangulated category 
$\D(\mod(\SS(\frakg\frakl_r[1]\langle 2\rangle\rtimes GL_r^c))$ of all 
$GL_r^c$-equivariant finitely generated dg-modules over $\SS(\frakg\frakl_r[1]\langle2\rangle)$.
The derived category $\D^\b\Coh([\frakg\frakl_r/GL_r^c])$ is 
$\D(\mod(\SS(\frakg\frakl_r^*\langle -2\rangle)\rtimes GL_r^c))$, which is the same as the triangulated category 
$\D(\mod(\SS(\frakg\frakl_r^*[-2]\langle -2\rangle)\rtimes GL_r^c))$.
So the Koszul duality yields an equivalence 
$$\D(\text{dg}\Coh(T^*\calX_\beta^c))\to\D^\b\Coh([\frakg\frakl_r/GL_r^c]).$$
Thus, we must compare the triangulated category $\D^\b\Coh([\frakg\frakl_r/GL_r^c])$
with a module category of the quiver-Hecke algebra $\bfR$ of affine type $\bfQ^{(1)}$.
Let $\bfU_q(\frakn^{(1)})_{\Lambda^+}$ be the current subalgebra of $\bfU_q(L\frakn)$.
Theorem D  implies that the Grothendieck group of the triangulated category
$\D^\b\Coh([\frakg/G^c])_{\Lambda^+}$ is $\bfU_q(\frakn^{(1)})_{\Lambda^+}$.
On the other hand, the split Grothendieck group
$\K_0(\calD^\proj)$ is isomorphic to $\bfU_q(\frakn^{(1)})_{\Lambda^+}$.
Thus, the Grothendieck group of the triangulated category $\K^\b(\calD^\proj)$ is also 
 isomorphic to $\bfU_q(\frakn^{(1)})_{\Lambda^+}$. Since the category $\calD$ has finite homological dimension, we 
 deduce that the Grothendieck group of  the graded triangulated $\D^\b(\calD)$
 is also isomorphic to $\bfU_q(\frakn^{(1)})_{\Lambda^+}$.
 Thus it is natural to compare the graded triangulated
 categories $\D^\b(\calD)$ and $\D^\b\Coh([\frakg/G^c])_{\Lambda^+}$.
 This is precisely the Conjecture A above.

\smallskip

\subsection{Relation to previous works}
In \cite{B16}, Bezrukavnikov proved that the realization of the affine Hecke algebra of a reductive group, 
both 
as the Grothendieck group
of a monoidal category of equivariant coherent sheaves on the Steinberg variety 
\emph{\`a la Ginzburg-Kazhdan-Lusztig}
and as the Grothendieck group of a monoidal category of equivariant perverse sheaves on the affine 
flag manifold of the Langlands dual group \emph{\`a la Kazhdan-Lusztig}, 
lifts to a graded monoidal triangulated equivalence between corresponding 
derived categories.
Here, we prove that the realization of a positive piece of the 
quantum enveloping algebra of affine type $A_1^{(1)}$,
both 
as the preprojective K-theoretic Hall algebra of the category of $\Pi_\bfQ$-modules of type $A_1$ 
\emph{\`a la Schiffmann-Vasserot}
and as the Grothendieck group of a category of
semisimple complexes on a quiver moduli stack of affine type $A_1^{(1)}$ \emph{\`a la Lusztig}
(or rather its algebraic equivalent description via quiver-Hecke algebras of type $\bfQ^{(1)}$)
 lifts to a graded monoidal triangulated equivalence between the corresponding 
derived categories.

\smallskip

According to \cite{L09}, the Langlands correspondence for $\bbP^1$ 
can be viewed as the composition of a chain of equivalences
\begin{align*}
\xymatrix{
\D^\b(\calB un)\ar[r]^-{\F}&
\D^\b(\calG r^+)\ar[r]^-{\B}&\D^\b(\calG r^-)\ar[r]^-{\C}&\D^b(\mod(\SS(\frakg^*[-2]\langle-2\rangle)\rtimes G^c))
\ar[d]^-{\D}\\
&&\D^\b(Coh(Loc\,_G))&\ar@{=}[l]\D^b(\mod(\SS(\frakg[1]\langle2\rangle)\rtimes G^c))
}
\end{align*}
where $\D$ is the Koszul duality mentioned above and  $\F$ is the tilting equivalence in \S\ref{sec:tilting}.
Thus, our equivalence can be viewed as a mixed non-equivariant version of the Langlands correspondence for $\bbP^1$
and for the linear group.

\smallskip

\subsection{Convention}
We'll say that a square of functors
is commutative if it commutes up to a natural isomorphism. 
For any algebraic group $G$ or ring $\bfR$, we denote by $Z(G)$, $Z(\bfR)$ the centers of $G$ and $\bfR$.
For any commutative ring $\k$, let $H^\bullet_G$ be the cohomology of the classifying 
space of $G$ with $\k$-coefficients.
Unless specified otherwise, all modules are left modules.
We refer to Appendix \ref{sec:ASMG} for all conventions related to mixed geometry.

\smallskip

\subsection{Acknowledgements} Initial stages of this work were partly based on discussions with R. Rouquier.
 We would like to thank him for all these discussions.
 We are also grateful to S. Riche and O. Schiffmann for answering various questions.

\medskip

\section{Quiver-Hecke algebras}

Let $\k$ be a field which is algebraically closed of characteristic zero.
Let $F$ be any field. 

\subsection{Graded categories}

A \emph{graded $\k$-linear category} $\calC$ is a $\k$-linear 
additive category with a compatible system
of self-equivalences $\langle a\rangle$, with $a\in\bbZ$, called grading shift functors.
A \emph{graded functor} of graded $\k$-linear additive categories is a $\k$-linear functor $E$ 
with a compatible system
of natural equivalences $E\,\langle a\rangle\to \langle a\rangle\, E$.
All the additive categories we'll consider here are $\k$-linear.
To simplify, we'll omit the word \emph{$k$-linear} everywhere.
A functor of additive categories
is assumed to be $k$-linear.
If the category $\calC$ is triangulated, as well as the graded shift functors, then we'll say that $\calC$ is a
\emph{graded triangulated category}.
A monoidal structure on an additive category $\calC$ is the datum of a bifunctor $\otimes:\calC\times\calC\to\calC$,
an associativity constraint $\bfa$ and an object $\bf1$, called the unit, with an isomorphism
$\varepsilon : \bf1\otimes\bf1\to \bf1$ satisfying the pentagon axiom and such that
the functors $\bf1\otimes\bullet$ and $\bullet\otimes\bf1$ are fully faithful,
see, e.g., \cite[App.~A]{KKK18}.
If all functors involved in this definition are graded, we'll say that 
$(\calC,\otimes,\bfa,\bf1,\varepsilon)$ is a \emph{graded monoidal category}.
If the category $\calC$ is Abelian and the bifunctor $\otimes$ is exact, we says that the monoidal category is exact.
If the category $\calC$ is triangulated and the bifunctor $\otimes$ is triangulated, 
we says that the monoidal category is triangulated.
If $\otimes$, $\bf1$, $\bfa$ or $\varepsilon$ are obvious from the context, 
we may omit them and we write simply $(\calC,\otimes)$ or $\calC$.
A monoidal functor of monoidal categories is a functor $E$ with a natural isomorphism of bifunctors
$E(\bullet\otimes\bullet)\to E(\bullet)\otimes E(\bullet)$ and an isomorphism $E(\bf1)\to\bf1$ satisfying some well-known
axioms.

\smallskip

For any objects $A,B\in\calC$, the graded-homomorphisms space is the graded vector space
\begin{align*}
\bbHom_{\calC}(A,B)=\bigoplus_{a\in\bbZ}\Hom_{\calC}(A,B\langle a\rangle)\langle -a\rangle.
\end{align*}
A projective object $P$ of a graded $\k$-linear Abelian category $\calC$ is called a 
\emph{projective graded-generator}
if any object of $\calC$ is a quotient of a finite sum of $P\langle a\rangle$'s with $a\in\bbZ$. 
We'll also say that $\calC$ is \emph{graded-generated} by $P$. Then, the functor 
$M\mapsto\bbHom_\calC(P\,,\,M)$
is an equivalence from $\calC$ to the graded category of all finitely generated graded modules over
the graded $\k$-algebra $\bfR^\op$ opposite to $\bfR=\bbEnd_\calC(P)$, 
see, e.g., \cite[prop.~E.4]{AJS}.

\smallskip

Let $\calC^\fl,$ $\calC^\proj$ be the full graded subcategories of $\calC$ consisting of all 
finite length objects and all projective objects.
Let $\G_0(\calC^\fl)$ and $\K_0(\calC^\proj)$ be the Grothendieck group and the split Grothendieck group 
of the Abelian category $\calC^\fl$ and the additive category $\calC^\proj$.
For each object $M$, let $\lgroup M\rgroup$ be its class in the Grothendieck group.
The ring $\calA=\bbZ[q,q^{-1}]$ acts on $\G_0(\calC^\fl)$, $\K_0(\calC^\proj)$ with
$q=\langle 1\rangle.$

\smallskip

\begin{example}\hfill
\begin{itemize}[leftmargin=8mm]
\item[$\mathrm{(a)}$] 
Given a graded Noetherian $\k$-algebra $\bfR$, let $\calC=\bfR\-\mod$ be the category of 
all finitely generated graded $\bfR$-modules.
Both are equipped with the grading shift functors $\langle a\rangle$ such that
$(M\langle 1\rangle)_a=M_{a+1}$ for each integer $a$. We'll abbreviate
$\D^\b(\bfR)=\D^\b(\bfR\-\mod)$ and $\K^\b(\bfR)=\K^\b(\bfR\-\mod).$

\item[$\mathrm{(b)}$] 
Let $\K^\b(\calC)$ be the bounded homotopy category of a graded additive category $\calC$ with
grading shift functor $\langle \bullet\rangle$.
The category $\K^\b(\calC)$ is graded triangulated with the grading shift functor
given by $\langle \bullet\rangle$.
Let $[\bullet]$ be the cohomological shift functor.
Recall that
\begin{align*}
\bbHom_{\K^\b(\calC)}(A,B)=\bigoplus_{a\in\bbZ}\Hom_{\K^\b(\calC)}(A,B\langle a\rangle)\langle -a\rangle.
\end{align*}
We write
\begin{align*}
\Hom^\bullet_{\K^\b(\calC)}(A,B)=\bigoplus_{a\in\bbZ}\Hom^a_{\K^\b(\calC)}(A,B)[-a]
\quad,\quad
\Hom^a_{\K^\b(\calC)}(A,B)=\Hom_{\K^\b(\calC)}(A,B[a]).
\end{align*}

\item[$\mathrm{(c)}$] 
Let $\D^\b(\calC)$ be the bounded derived category of a graded Abelian category $\calC$
with grading shift functor $\langle \bullet\rangle$.
The category $\D^\b(\calC)$ is graded triangulated with the grading shift functor
given by $\langle \bullet\rangle$. 
We define
$\Hom^\bullet_{\D^\b(\calC)}(A,B)$ and
$\Hom^a_{\D^\b(\calC)}(A,B)$ as in (b). Let $\calC=\bfR\-\mod$ be as in (a).
A complex of graded $\bfR$-modules is \emph{perfect} if it is quasi-isomorphic
to a bounded complex of finitely generated projective graded $\bfR$-modules.
The \emph{category of perfect complexes} is the full graded triangulated subcategory of
$\D^\b(\calC)$ given by
$\D^\perf(\calC)=\K^\b(\calC^\proj)$. 
\end{itemize}

\end{example}

\smallskip

\subsection{Polynomial highest weight and affine properly stratified categories}\label{sec:HW}

Let $\calC$ be a graded $\k$-linear Abelian category with a finite set $\{L(\pi)\,;\,\pi\in\KP\}$ of simple objects which is complete and irredundant up to isomorphisms
and grading shifts. 
We'll assume that $\calC=\bfR\-\mod$, where $\bfR$ is a 
\emph{Schurian Noetherian Laurentian} graded algebra as in \cite[\S 2.1]{K15c}.
Hence, each simple object $L(\pi)$ admits a projective cover $P(\pi)\to L(\pi)$ with kernel $M(\pi)$, 
and for any object $M\in\calC$ the composition multiplicity of $L(\pi)$ in $M$ is the formal series
in $\bbZ((q))$ given by
$[M\,:\,L(\pi)]=\dim\bbHom_\calC(P(\pi)\,,\,M).$

\smallskip

We assume that the set $\KP$ is equipped with a partial \emph{preorder} $\leqslant$, and with a map
$\rho : \KP\to\Pi$ to a partial ordered set $\Pi$ such that we have
\begin{align}\label{order}
\pi\leqslant\pi'\iff\rho(\pi)\leqslant\rho(\pi').
\end{align}
For each $\pi\in \KP$, we define
the \emph{standard} and \emph{proper standard} objects $\Delta(\pi)$ and $\bar\Delta(\pi)$ such that
$\Delta(\pi)$ is the largest quotient of $P(\pi)$ such that all its composition factors $L(\sigma)$ satisfy 
$\sigma\leqslant\pi$, and $\bar\Delta(\pi)$ is the largest quotient of $P(\pi)$ which has $L(\pi)$ with multiplicity 1 
and such that all its other composition factors $L(\sigma)$ satisfy $\sigma<\pi$.
Let $K(\pi)$ be the kernel of the surjection $P(\pi)\to\Delta(\pi)$.

\smallskip

We say that an object $M\in\calC$ has a \emph{$\Delta$-filtration} if it has a separated filtration
whose subquotients are isomorphic to $\Delta(\pi)\langle a\rangle$ for some $\pi\in\KP$, $a\in\bbZ$.
Let $\calC^\Delta$ be the full graded additive subcategory of $\calC$ consisting of all 
$\Delta$-filtered objects.
For all $\xi\in\Pi$ we write
$\Delta(\xi)=\bigoplus_{\rho(\pi)=\xi}\Delta(\pi)$.

\smallskip

\begin{definition}[\cite{K15b}] The category $\calC$ is
\hfill
\begin{itemize}[leftmargin=8mm]
\item[$\mathrm{(a)}$]
 \emph{affine properly stratified} if,  for each $\pi\in\KP$,  $\xi\in\Pi$,
\begin{itemize}[leftmargin=8mm]
\item[$\mathrm{(1)}$] 
$K(\pi)$ has a $\Delta$-filtration with subquotients of the form $\Delta(\sigma)\langle a\rangle$ for $\sigma>\pi$ and $a\in\bbZ$,
\item[$\mathrm{(2)}$]  $\Delta(\xi)$ is finitely generated and flat as
a module over the algebra $\bbEnd_\calC(\Delta(\xi))^\op$,
\item[$\mathrm{(3)}$]  $\bbEnd_\calC(\Delta(\xi))^\op$ is a finitely generated commutative graded $\k$-algebra,
\end{itemize}

\item[$\mathrm{(b)}$]
 \emph{polynomial highest weight} if the map $\rho$ is bijective and,  for each $\pi\in\KP$,
\begin{itemize}[leftmargin=8mm]
\item[$\mathrm{(1)}$] 
$K(\pi)$ has a $\Delta$-filtration with subquotients of the form 
$\Delta(\sigma)\langle a\rangle$ for $\sigma>\pi$ and $a\in\bbZ$,
\item[$\mathrm{(2)}$]  $\Delta(\pi)$ is finitely generated and free as
a module over the algebra $\bbEnd_\calC(\Delta(\pi))^\op$,
\item[$\mathrm{(3)}$] $\bbEnd_\calC(\Delta(\pi))^\op$ 
is a graded polynomial $\k$-algebra.
\end{itemize}
\end{itemize}
\end{definition}

\smallskip

Since $\calC$ is an affine properly stratified category with a finite number
of isomorphism classes of simple objects, it
is graded equivalent to $R\-\mod$ for some \emph{affine properly stratified} algebra $R$ by \cite[cor.~6.8]{K15b}.
Let $I(\pi)$ denote the injective hull of $L(\pi)$ in the category of all graded $R$-modules
(not necessarily finitely generated).
We define $\nabla(\pi)$ to be the largest submodule of $I(\pi)$ such that all its composition factors $L(\sigma)$ satisfy 
$\sigma\leqslant\pi$, and $\bar\nabla(\pi)$ to be the largest submodule of $I(\pi)$ which has $L(\pi)$ with multiplicity 1 
and such that
all its other composition factors $L(\sigma)$ satisfy $\sigma<\pi$.
The objects $\nabla(\pi)$ and $\bar\nabla(\pi)$ are called the \emph{costandard} 
and the \emph{proper costandard} objects. Note that $\nabla(\pi)$ does not necessarily belong to $\calC$ in general.

\smallskip

For each subset $\Gamma\subset\KP$, let  $\calC_\Gamma\subset\calC$ be the full subcategory
of all objects whose composition factors are isomorphic to graded shifts of simple objects $L(\sigma)$ with 
$\sigma\in\Gamma$. Let 
$$(f_\Gamma)_*:\calC_\Gamma\to\calC\quad,\quad (f_\Gamma)^*:\calC\to\calC_\Gamma$$
be the obvious full embedding and its left adjoint.
The functor $(f_\Gamma)^*$ sends projectives to projectives because  $(f_\Gamma)_*$ is exact.
For each object $M$ we have
$(f_\Gamma)^*(M)=M\,/\,o_\Gamma(M)$, where $o_\Gamma(M)$ is the minimal subobject $N\subset M$ such that
$M/N\in\calC_\Gamma$.
We have the following exact functors between derived categories
$$(f_\Gamma)_*:\D^\b(\calC_\Gamma)\to\D^\b(\calC)
\quad,\quad
L(f_\Gamma)^*:\D^-(\calC)\to\D^-(\calC_\Gamma).$$
We say that $\Gamma$ is an \emph{order ideal} (resp. \emph{order coideal})
if for each $\sigma\in\Gamma$, $\pi\in\KP$ we have
$$\sigma\geqslant\pi\Longrightarrow\pi\in\Gamma\quad(\text{resp.}\ \sigma\leqslant\pi\Longrightarrow\pi\in\Gamma).$$
If $\Gamma$ is a finite order ideal, then the following hold \cite[lem.~7.18\,,\,prop.~7.20]{K15b}
\hfill
\begin{itemize}[leftmargin=8mm]
\item[$\mathrm{(a)}$] $(f_\Gamma)_*:\D^\b(\calC_\Gamma)\to\D^\b(\calC)$ is fully faithful,
\item[$\mathrm{(b)}$] $L(f_\Gamma)^*M=(f_\Gamma)^*M$ if $M\in\calC^\Delta$,
\item[$\mathrm{(c)}$] $(f_\Gamma)^*P(\pi)$ is a projective cover of
$L(\pi)$ in $\calC_\Gamma$ if $\pi\in\Gamma$, 
\item[$\mathrm{(d)}$]
$\Delta(\pi)=(f_\Gamma)_*(f_\Gamma)^*P(\pi)$ if $\Gamma=\{\leqslant\!\pi\}$.

\end{itemize}

\smallskip

\subsection{Kostant partitions}\label{sec:roots}
Let $\bfQ$ be a finite quiver associated with a Kac-Moody algebra of affine type.
Let $I$ be the set of vertices and $\Omega$ the set of arrows.
Let $\Phi$ be the root system of $\bfQ$.
We identify $I$ with a fixed set of simple roots $\{\alpha_i\,;\,i\in I\}$ in the obvious way.
Let $\{\Lambda_i\,;\,i\in I\}$ be the set of fundamental weights.
Let $\Phi_+\subset\Phi$ the subset of positive roots
and  $Q_+=\bigoplus_{i\in I}\bbN\alpha_i$.
Let $\hgt(\beta)$ be the height of an element $\beta$ of $Q_+$.
Fix $0\in I$ such that $\{\alpha_i\,;\,i\in I\,,\,i\neq 0\}$
is the set of simple roots of a root system $\Delta$ of finite type.
Let $\delta\in\Phi_+$ be the null root.
The set of real positive roots is $\Phi_+^\re=\Phi_{+-}\sqcup\Phi_{++}$, where
$$\Phi_{+-}=\{\beta+n\delta\,;\,\beta\in\Delta_+\,,\,n\in\bbN\}
\quad,\quad
\Phi_{++}=\{-\beta+n\delta\,;\,\beta\in\Delta_+\,,\,n\in\bbZ_{>0}\}.$$

\smallskip

We fix a total convex preorder on the set $\Phi_+$ such that for each $i\in I\setminus\{0\}$ we have
\begin{itemize}[leftmargin=8mm]
\item[$\mathrm{(a)}$] 
$\alpha_i>\alpha_i+\delta>\alpha_i+2\delta>\cdots>\bbZ_{>0}\,\delta>\cdots>-\alpha_i+2\delta>-\alpha_i+\delta,$
\item[$\mathrm{(b)}$] 
$m\delta>n\delta$ for all $m,n>0,$
\item[$\mathrm{(c)}$] each root in 
$\Phi_+^\re\,\setminus\,\{\alpha_i+n\delta\,,\,-\alpha_i+(n+1)\delta\,;\,n\in\bbN\}$ 
is either $>\alpha_i$ or $<-\alpha_i+\delta$.
\end{itemize}

\smallskip

For any subset $X\subset Q_+$ let $X^\beta$ be the set of tuples 
of elements of $X$ with sum $\beta$.
A \emph{Kostant partition} of $\beta$ is a decreasing tuple 
$\pi=(\pi_1\geqslant\pi_2\geqslant\cdots\geqslant\pi_r)$ in $\Phi_+^\beta$.
We may also write
$\pi=\big((\beta_l)^{p_l},\dots,(\beta_1)^{p_1},(\beta_0)^{p_0}\big)$, meaning that
$l$ is a positive integer and
$\beta_l>\dots>\beta_1>\beta_0$ are positive roots counted with multiplicities $p_l,\dots,p_1,p_0$ respectively.
A Kostant partition $\pi$ is a \emph{root partition} if the multiplicity of each decomposable affine positive root 
$n\delta$, with $n>1$, is zero.
Let $\Pi_\beta\subset\K\P\!_\beta$ be the sets of all root partitions and Kostant partitions of $\beta$.
Both sets are finite.
There is a unique map $\rho:\K\P\!_\beta\to\Pi_\beta$ such that the real positive roots have the same 
multiplicities in
$\pi$ and $\rho(\pi)$. We can view a Kostant partition $\pi$ as a pair $(\rho(\pi)\,,\,\mu)$ 
where $\mu$ is a partition of the multiplicity
of $\delta$ in $\rho(\pi)$, then the map $\rho$ is just the projection on the first factor.
The set $\Pi_\beta$ of root partitions is equipped with the \emph{bilexicographic} partial order, such that
\begin{align}\label{bilex}
\pi\leqslant\pi'\Longleftrightarrow \pi\leqslant_l\pi'\ \text{and}\  \pi\geqslant_r\pi',
\end{align} where
$\leqslant_l$ and $\leqslant_r$ are the left and right lexicographic orders.
We equip $\K\P\!_\beta$ with the partial preorder defined as in 
\eqref{order} using the map $\rho:\K\P\!_\beta\to\Pi_\beta$.
Let $S_p$ be the symmetric group. For any Kostant partition
$\pi=((\beta_l)^{p_l},\dots,(\beta_1)^{p_1},(\beta_0)^{p_0})$
we define the tuple
$\bar\pi=(p_l\beta_l,\dots,p_1\beta_1,p_0\beta_0)$
in $Q_+^\beta$.
Set
\begin{align}\label{norm}
S_\pi=S_{p_l}\times\cdots\times S_{p_1}\times S_{p_0}
\quad,\quad
n_\pi=\sum_{k=1}^lkp_k
\quad,\quad
r_\pi=\sum_{k=0}^lp_k.
\end{align}

\smallskip

\subsection{Quiver-Hecke algebras and their module categories}\label{sec:KLR}
Let $\beta$ be any element in $Q_+$. Let $\bfR_\beta$ be the quiver-Hecke algebra over $\k$
associated with the category of $\beta$-dimensional representations of the quiver $\bfQ$. 
We choose the parameters of $\bfR_\beta$ to be given by the family of polynomials 
$Q_{ij}(u,v)=(-1)^{h_{ij}}(u-v)$ where
$h_{ij}=\sharp\{i\to j\in\Omega\}.$
Let $b=\hgt(\beta)$ be the height of  $\beta$. 
The quiver-Hecke algebra $\bfR_\beta$ is the associative, unital, graded $\k$-algebra generated by a complete set
of orthogonal idempotents $\{\e_\nu\,;\,\nu\in I^\beta\}$ and some
elements $x_i$, $r_k$ with $i\in[1,b]$, $k\in[1,b-1]$, satisfying the following relations
\smallskip
\begin{itemize}[leftmargin=8mm]
\item[$\mathrm{(a)}$] $x_ix_j=x_jx_i\quad,\quad x_i\,\e_\nu=\e_\nu \,x_i\quad,\quad
r_k\,\e_\nu=\e_{s_k\nu}\,r_k\quad,\quad r_k^2\,\e_\nu=Q_{\nu_k,\nu_{k+1}}(x_k,x_{k+1})\,\e_\nu,$
\item[$\mathrm{(b)}$] $(r_kx_l-x_{s_k(l)}r_k)\,\e_\nu=
\begin{cases}
-\e_\nu&\text{if} \ l=k$, $\nu_k=\nu_{k+1},\\\e_\nu&\text{if} \ l=k+1$, $\nu_k=\nu_{k+1},\\0&\text{otherwise},
\end{cases}$\\
\item[$\mathrm{(c)}$] $(r_{k+1}r_kr_{k+1}-r_kr_{k+1}r_k)\,\e_\nu=
\begin{cases}
{Q_{\nu_k,\nu_{k+1}}(x_k,x_{k+1})-Q_{\nu_{k},\nu_{k+1}}(x_{k+2},x_{k+1})\over x_k-x_{k+2}}\,\e_\nu&\text{if} \ \nu_k=\nu_{k+2},\\0&\text{otherwise}.
\end{cases}$
\end{itemize}
There is an antiautomorphism  of $\bfR_\beta$ which fixes all the generators.

\smallskip

Let $\calC_\beta=\bfR_\beta\-\mod$ be the graded Abelian category
of all finitely generated graded $\bfR_\beta$-modules.
For any tuple $\pi=(\pi_1,\dots, \pi_r)$ in $Q_+^\beta$
we set
$$\bfR_\pi=\bfR_{\pi_1}\otimes\dots\otimes \bfR_{\pi_{r-1}}\otimes \bfR_{\pi_r}
\quad,\quad
\calC_\pi=\bfR_\pi\-\mod.$$
There is an obvious inclusion $\bfR_\pi\subset\bfR_{\beta}$.
It yields an adjoint pair of exact induction/restriction functors 
\begin{align}\label{indres1}\Ind_\pi:\calC_\pi\to\calC_\beta
\quad,\quad
\Res_\pi:\calC_\beta\to \calC_\pi.
\end{align}
If $r=2$ we abbreviate $\circ=\Ind_\pi$.
By \cite[cor.~6.24]{K15c}, the category $\calC_\beta$ is affine properly stratified.
The sets of proper standard objects $\{\bar\Delta(\pi)\,;\,\pi\in\K\P\!_\beta\}$ and standard objects
$\{\Delta(\pi)\,;\,\pi\in\K\P\!_\beta\}$
are defined in \cite[(6.5), (6.6)]{K15c}. 
The map $\rho:\K\P\!_\beta\to\Pi_\beta$ is as in \S\ref{sec:roots}.
Let $L(\pi)$ be the top of the module $\bar\Delta(\pi)$ and $P(\pi)$ its projective cover.
For a future use, we fix an idempotent $\e(\pi)$ in $\bfR_\beta$ such that
\begin{align}\label{idempotent}P(\pi)=\bfR_\beta\, \e(\pi).\end{align}
Assume temporarily that $\beta\in\Phi^\re_+$ is real positive root.
Then, by \cite[\S8,9]{Mc18}, the graded module $L(\beta)$ is the unique self-dual irreducible module in $\calC_\beta$ 
which is \emph{cuspidal}, which means that for each $\alpha,\gamma\in Q_+$ such that $\beta=\alpha+\gamma$
we have $\Res_{\alpha,\gamma}(L(\beta))=0$ unless $\alpha\prec\beta\prec\gamma$.
The partial order is such that
$\alpha\prec\beta$ if and only if $\alpha$ is a sum of positive roots $<\beta$ and
$\beta\prec\gamma$ if and only if $\gamma$ is a sum of positive roots $>\beta$.
For each graded module $M$ we write
\begin{align*}M\langle p\rangle\,!=
\bigoplus_{s=1}^tM\langle k_s\rangle\ \text{where}\ 
\prod_{k=0}^{p-1}(1+q^2+\cdots+q^{2k})=q^{k_1}+q^{k_2}+\cdots+q^{k_t}.
\end{align*}
For any integer $p>0$ and $M\in\calC_\beta$,
we write $M^{\circ p}=\Ind_{(\beta^p)}\big(M^{\otimes p}\big)$.
Then, we have
$$L(\beta^p)=L(\beta)^{\circ p}\langle p(p-1)/2\rangle.$$
Further, there is an indecomposable summand $\Delta(\beta)^{(p)}$ of  $\Delta(\beta)^{\circ p}$ such that
\begin{align}\label{Delta0}
\Delta(\beta)^{\circ p}=\Delta(\beta)^{(p)}\langle p\rangle\,!.
\end{align}
By \cite[(6.6)]{K15c}, \cite[\S10,\S 24]{Mc18}, the standard and proper standard modules in $\calC_\beta$ are
\begin{align}\label{Delta}
\begin{split}
\Delta(\pi)=
\Delta(\beta_l)^{(p_l)}\circ\dots\circ\Delta(\beta_0)^{(p_0)}
\quad,\quad
\bar\Delta(\pi)=
L(\beta_l^{\,p_l})\circ\dots\circ 
L(\beta_0^{\, p_0}).
\end{split}
\end{align}

\smallskip

Now, let $\beta$ be any element in $Q_+$.
Let
$\Gamma\!_\beta\subset\K\P\!_\beta$ be the set of all Kostant partitions supported
on the set $\Phi_{++}$ and let $\Gamma\subset\K\P\!_\beta$ be any order ideal. We abbreviate
$$\calD_\Gamma=(\calC_\beta)_\Gamma\quad,\quad\calD_\beta=\calD_{\Gamma\!_\beta}.$$
Let $(f_\Gamma)^*$ be the left adjoint functor of the obvious exact full embedding
$(f_\Gamma)_*:\calD_\Gamma\to\calC_\beta.$
For any nested order ideal $\Gamma\!_1\subset\Gamma$ contained in $\Gamma\!_\beta$ we define similarly the functor
$(h_{\Gamma\!_1,\Gamma})_*:\calD_{\Gamma\!_1}\to\calD_\Gamma$
and its left adjoint $(h_{\Gamma\!_1,\Gamma})^*$.
We abbreviate $h_\Gamma=h_{\Gamma,\Gamma\!_\beta}$.
Let $Q(\pi)$ be the projective cover in $\calD_\beta$ of the simple graded module $L(\pi)$.
We have 
$Q(\pi)=(f_{\Gamma\!_\beta})^*P(\pi).$

\smallskip

\begin{proposition}\label{prop:Cinfty}
Assume that $\Gamma\!_1\subset\Gamma\subset\Gamma\!_\beta$ are order ideals.
\hfill
\begin{itemize}[leftmargin=8mm]
\item[$\mathrm{(a)}$] $\calD_\Gamma$ is polynomial highest weight with
finite global dimension.
\item[$\mathrm{(b)}$]
$(f_\Gamma)_*:\D^\b(\calD_\Gamma)\to\D^\perf(\calC_\beta)$ and 
$(h_{\Gamma\!_1,\Gamma})_*:\D^\b(\calD_{\Gamma\!_1})\to\D^\b(\calD_\Gamma)$ 
are fully faithful  functors.
\end{itemize}
\end{proposition}

\begin{proof}
Since the category $\calC_\beta$ is affine properly stratified relatively to the map $\rho$, by \cite[prop.~5.16]{K15b}
the category $\calD_\Gamma$ is affine properly stratified relatively to restriction
$\rho|_{\Gamma}$ of $\rho$ to the subset $\Gamma$, with the sets of
proper standard and standard modules given by
$\{\bar\Delta(\pi)\,;\,\pi\in\Gamma\}$ and
$\{\Delta(\pi)\,;\,\pi\in\Gamma\}.$
The map $\rho|_\Gamma$ is the identity of $\Gamma$.
By \cite[thm.~6.14]{K15c}, to prove that the category $\calD_\Gamma$ 
is polynomial highest weight, it is enough to check that
$\bbEnd_{\calC_\beta}(\Delta(\pi))$ is a graded polynomial ring for each $\pi\in\Gamma$.
Let $\alpha\in \Phi^\re_+$ be any positive real root.
There is a central element of degree 2
\begin{align}\label{CENT3}
z_\alpha=x_1+\cdots+x_{\hgt(\alpha)}\in Z(\bfR_\alpha)
\end{align} 
such that the standard module
$\Delta(\alpha)$ is a free module over $\k[z_\alpha]$ and the action yields an isomorphism of graded rings, 
see \cite[\S 6]{K15c},  \cite[\S 15]{Mc18}
\begin{align}\label{isom0}\k[z_\alpha]=\bbEnd_{\calC_\alpha}(\Delta(\alpha))
\end{align}
By \cite[thm.~24.1]{Mc18}, the functoriality of induction and \eqref{isom0} yield a graded ring isomorphism
\begin{align}\label{ISOM0}
\bbEnd_{\calC_\beta}(\Delta(\alpha)^{(p)})=\k[z_1,z_2,\dots,z_p]^{S_p}.
\end{align}
Thus, by \eqref{Delta}, a short computation yields
\begin{align*}
\bbEnd_{\calC_\beta}(\Delta(\pi))
&\simeq
\bbHom_{\calC_{\bar\pi}}\big(\Delta(\beta_l)^{(p_l)}\otimes\dots\otimes\Delta(\beta_1)^{(p_1)}\otimes\Delta(\beta_0)^{(p_0)}\,,\,
\Res_{\bar\pi}(\Delta(\pi))\big)\\
&\simeq
\bbEnd_{\calC_{\bar\pi}}\big(\Delta(\beta_l)^{(p_l)}\otimes\dots\otimes\Delta(\beta_1)^{(p_1)}\otimes\Delta(\beta_0)^{(p_0)}\big)\\
&\simeq
\k[z_1,z_2,\dots,z_{r_\pi}]^{S_\pi}.
\end{align*}
The first isomorphism is the adjunction, the second  is \cite[lem.~8.6]{Mc18},
the third  is \eqref{ISOM0}.
The global dimension of $\calD_\Gamma$ is finite by \cite[cor.~5.25]{K15b}, because the set
$\K\P\!_\beta$ is finite.
This proves part (a). 

\smallskip

Since $\calD_\Gamma$ has a finite global dimension and any object of $\calD_\Gamma^\proj$
has a finite $\Delta$-filtration, to prove that $(f_\Gamma)_*$ maps into $\D^\perf(\calC_\beta)$ it is enough to 
observe that any standard module in $\calD_\Gamma$ has a finite projective dimension in $\calC_\beta$.
By \cite[lem.~5.17]{K15b}, this projective dimension is less than $\sharp\K\P\!_\beta$, which is finite.
Thus, part (b) follows from  \cite[prop.~7.20]{K15b}.
\end{proof}

\smallskip

\subsection{Quiver-Hecke algebras and the moduli stack of representations of quivers}
\label{sec:CK}

Let  $\ell$ be a prime number and $\k=\overline\bbQ_\ell$. 
Let $F$ be an algebraically closed field of characteristic prime to $\ell$.
See \S 2 below for a reminder on Artin stacks and constructible sheaves on stacks.

\smallskip

Let $\beta$ be any element in $Q_+$. 
Let $F\bfQ$ be the path algebra of the quiver $\bfQ$ over $F$.
Let $X_\beta$ be the variety of all $F\bfQ$-modules
in an $I$-graded $\beta$-dimensional $F$-vector space $V$.
It is isomorphic to the affine space $\bbA^{d_\beta}_{F}$ for some $d_\beta\in\bbN$.
Let $G_\beta$ be the group of all $I$-graded $F$-linear automorphisms of $V$.
The algebraic $F$-group $G_\beta$ acts on $X_\beta$ in the obvious way.
The moduli stack  of $\beta$-dimensional representations of $F\bfQ$ is the quotient stack 
$\calX_{\beta}=[X_{\beta}\,/\,G_\beta]$ over $F$. 
It is a smooth stack of finite type.

\smallskip

For each tuple $\nu=(\nu_1,\dots,\nu_r)$ in $Q_+^\beta$ we fix a flag of $I$-graded vector $F$-spaces
$V=V_0\supset\cdots\supset V_{r-1}\supset V_r=0$ in $V$
such that $\dim(V_{k-1}/V_k)=\nu_k$ for each $k=1,2,\dots,r$.
Let $X_{V_\bullet}\subset X_\beta$ be the set
of all representations of $F\bfQ$ in $V$ which preserve the flag $V_\bullet$.
The stabilizer of the flag $V_\bullet$ in $G_\beta$ is a parabolic subgroup $P_{V_\bullet}$ 
whose action preserves the subset $X_{V_\bullet}$ of $X_\beta$.
Let $\widetilde\calX_\nu$ be the quotient stack.
Set
$\calX_\nu=\calX_{\nu_1}\times\dots\times \calX_{\nu_{r-1}}\times \calX_{\nu_r}.$
The stacks $\calX_\nu$, $\widetilde\calX_\nu$ fit into an induction diagram
\begin{align}\label{ind-diag}
\xymatrix{\calX_\nu&\ar[l]_-{q_\nu}\widetilde\calX_\nu\ar[r]^-{p_\nu}&\calX_{\beta}}.
\end{align}
The map $p_\nu$ is representable and projective.
The map $q_\nu$ is a vector bundle stack whose fiber over the object $(M_1,\dots,M_{r-1},M_r)$
is isomorphic to the quotient stack, relatively to the trivial action,
\begin{align}\label{fiber}
\big[\bigoplus_{k<h}\Ext^1_{F\bfQ}(M_k,M_h)\,\big/\,\bigoplus_{k<h}\Hom_{F\bfQ}(M_k,M_h)\big].
\end{align}
See, e.g., \cite[prop.~6.2]{B12} for a proof of this fact.
We have the adjoint pair of functors  \cite[\S 9]{L}
\begin{align}\label{indres2}
\ind_\nu:\D^\b(\calX_\nu) \to \D^\b(\calX_\beta)
\quad,\quad
\res_\nu:\D^\b(\calX_\beta) \to \D^\b(\calX_\nu)
\end{align}
which is given by
\begin{align}\label{ind2}\ind_\nu(\calE)=(p_\nu)_!(q_\nu)^*(\calE)[\dim q_\nu]
\quad,\quad
\forall\calE\in\D^\b(\calX_\nu).\end{align}
Since the induction is a shift of a proper direct image and a smooth inverse image, 
it commutes with the Verdier duality functor $D$

\smallskip

Let $\calL_\beta$ be the sum of the  complexes 
$\ind_\nu(\k_{\calX_\nu})$ where $\nu$ runs over the set $I^\beta$. 
It is self-dual, semisimple, and  
it decomposes in the following way 
$$\calL_\beta=\bigoplus_{\pi\in\K\P\!_\beta}\calL(\pi)
\quad,\quad
\calL(\pi)=V(\pi)\otimes IC(\pi).$$
The $IC(\pi)$'s are self-dual irreducible perverse sheaves which are pairwise non isomorphic. 
The multiplicities $V(\pi)$'s are nonzero complexes of vector spaces with 0 differential. 
We define the graded additive category $\Par(\calX_\beta)$
as the strictly full additive subcategory of $\D^\b(\calX_\beta)$ generated by 
the direct summands of $\calL_\beta$ and all their cohomological shifts.
The grading shift functors $\langle\bullet\rangle$ on $\Par(\calX_\beta)$ are  the cohomological shift
functors $[\bullet]$.
The Verdier duality yields an antiautoequivalence $D$ of $\Par(\calX_\beta)$.
Set
$$\Par(\calX_\nu)=\Par(\calX_{\nu_1})\times \dots\times \Par(\calX_{\nu_{r-1}})\times \Par(\calX_{\nu_r}).$$
The functor $\ind_\nu$ maps the subcategory 
$\Par(\calX_\nu)$ of $\D^\b(\calX_\nu)$ to $\Par(\calX_\beta)$.
It equips the graded addtitive category 
$\Par(\calX)=\bigoplus_{\beta\in Q_+}\Par(\calX_\beta)$ with the structure of a graded monoidal category
$(\Par(\calX)\,,\,\circledast)$.

\smallskip

By \cite{VV11}, there is a graded $\k$-algebra isomorphism
\begin{align}\label{VV}\bfR_\beta=\End^\bullet_{\D^\b(\calX_\beta)}(\calL_\beta)^\op.\end{align}
Under the Yoneda composition, for each complex $\calE$ we view 
$\Hom^\bullet_{\D^\b(\calX_\beta)}(\calL_\beta,\calE)$
as a graded $\bfR_\beta$-module.
This yields a graded functor of graded additive categories
\begin{align}\label{PHI0}\Phi_\beta^\bullet:\D^\b(\calX_\beta)\to\calC_\beta\quad,\quad 
\calE\mapsto\Hom^\bullet_{\D^\b(\calX_\beta)}(\calL_\beta,\calE)\end{align}
such that for each $\nu\in I^\beta$ we have
$$\Phi_\beta^\bullet(\ind_\nu(\k_{\calX_\nu}))=\bfR_\beta\,\e_{\nu^\op}
\quad,\quad\nu^\op=(\nu_r,\nu_{r-1},\dots,\nu_1).$$
The labelling of the simple summands $IC(\pi)$ of $\calL_\beta$ is chosen such that we have
\begin{align}\label{PPP} \Phi_\beta^\bullet(IC(\pi))=P(\pi)\quad,\quad\forall\pi\in\KP_\beta.\end{align}
By Remark \ref{rem:mixte2}, the functor $\Phi_\beta^\bullet$ 
gives an equivalence of graded additive categories
\begin{align}\label{PHI}
\Par(\calX_\beta)\to\calC_\beta^\proj\quad,\quad
IC(\pi)\mapsto P(\pi)\quad,\quad\forall\pi\in\K\P\!_\beta.
\end{align}
A functor of additive categories yields a triangulated functor of the corresponding homotopy categories.
Thus \eqref{PHI} yields an equivalence of graded triangulated categories
$$\Phi^\bullet:\K^\b(\Par(\calX_\beta))\to\D^\perf(\calC_\beta)$$
such that
$\Phi_\beta^\bullet(\ind_\nu(\k_{\calX_\nu}))
=\bfR_{\nu_r}\circ\bfR_{\nu_{r-1}}\circ\cdots\circ\bfR_{\nu_1}.$
More generally, the following is easy to prove.

\smallskip

\begin{proposition}\label{prop:ind}
The graded functor $\Phi_\beta^\bullet$ extends to a graded functor of
graded monoidal categories
$(\Par(\calX_\beta),\circledast)\to(\calC_\beta^\proj,\circ^\op)$, i.e.,
there is an isomorphism of functors
$\Phi_\beta^\bullet\,\ind_\nu=
\Phi_{\nu_r}^\bullet\circ\Phi_{\nu_{r-1}}^\bullet\circ\cdots\circ\Phi_{\nu_1}^\bullet.$
\qed
\end{proposition}


\smallskip

\subsection{Relation with quantum groups} \label{sec:CB}

Let $\bfU_q(\frakn)$ be the Lusztig $\calA$-form of
the positive half of the quantized enveloping algebra of type $\bfQ$.
Let $\bfU_q(\frakn)_\beta$ be the $\beta$-weight subspace of $\bfU_q(\frakn)$ and
$\bfB_\beta$ its canonical basis.
Let $\bfA_q(\frakn)_\beta$ be the dual $\calA$-module of $\bfU_q(\frakn)_\beta$ and
$\bfB_\beta^*$ its dual canonical basis.
We view $\bfA_q(\frakn)_\beta$ as an $\calA$-submodule of 
$\bfU_q(\frakn)_\beta\otimes_{\calA}\bbQ(q)$ via the Lusztig pairing  on $\bfU_q(\frakn)_\beta$.
Set $\bfA_q(\frakn)=\bigoplus_{\beta\in Q_+}\bfA_q(\frakn)_\beta.$
By \cite{VV11}, for each $\beta\in Q_+$ there is an $\calA$-linear isomorphism 
$g:\K_0(\calC^\proj_\beta)\to \bfU_q(\frakn)_\beta$
which takes
$\{\lgroup P(\pi)\rgroup\,;\,\pi\in\K\P\!_\beta\}$ to the canonical basis $\bfB_\beta.$ 
The transpose yields an isomorphism 
$\G_0(\calC_\beta^\fl)\to\bfA_q(\frakn)_\beta$
which takes
$\{\lgroup L(\pi)\rgroup\,;\,\pi\in\K\P\!_\beta\}$ to the dual canonical basis $\bfB^*_\beta$.
Since any module in $\calC_\beta$ has a (maybe infinite) composition series, this map extends to
a $\bbZ((q^{-1}))$-linear map 
\begin{align*}
\widehat g:\G_0(\calC_\beta)\to\bfA_q(\frakn)_\beta\otimes_\calA\bbZ((q^{-1})).\end{align*}
If $\beta$ is a real positive root let $E(\beta)$ be the root vector of weight $\beta$.
The dual root vector is $E(\beta)^*=(1-q^{-2})E(\beta)$.
By \cite[thm.~18.2]{Mc18}, we have
\begin{align*}
\widehat g\lgroup \Delta(\beta)\rgroup=E(\beta)
\quad,\quad 
\widehat g\lgroup L(\beta)\rgroup=E(\beta)^*
.\end{align*}
Next, there is a canonical $\bbZ[q,q^{-1}]$-algebra isomorphism 
$f:\K_0(\Par(\calX))\to \bfU_q(\frakn)$
such that 
$\{f\lgroup IC(\pi)\rgroup\,;\,\pi\in\K\P\!_\beta\}=\bfB_\beta$. 
Consider the following diagram
\begin{align*}
\begin{split}
\xymatrix{
\K_0(\Par(\calX_\beta))\ar[r]^-{\Phi_\beta^\bullet}\ar[rd]_f&
\K_0(\calC^\proj_\beta)\ar[d]^-g\ar@{^{(}->}[r]&\ar[d]^-{\widehat{g}}\G_0(\calC_\beta)\\
&\bfU_q(\frakn)_\beta\ar@{^{(}->}[r]&\bfA_q(\frakn)_\beta\otimes_\calA\bbZ((q^{-1}))
}
\end{split}
\end{align*}
The right square  commutes by definition.
The left triangle commutes by Proposition \ref{prop:ind}, because
the quantum divided powers $E(\alpha_i)^{(p)}$ with $i=0,1$ and $p>0$ generate $\bfU_q(\frakn)$ and
we have
$$f\lgroup\k_{X_{p\alpha_i}}\rgroup=E(\alpha_i)^{(p)}
\quad,\quad
g\lgroup P(\alpha_i^p)\rgroup=E(\alpha_i)^{(p)}
\quad,\quad
\Phi^\bullet_{p\alpha_i}\lgroup\k_{X_{p\alpha_i}}\rgroup=\lgroup P(\alpha_i^p)\rgroup
.$$

\medskip

\section{The Kronecker quiver $\bfQ$}

From now on, let $\bfQ$ be the Kronecker quiver. 
We have $I=\{0,1\}$ and $\Omega=\{x,y\}$ with both arrows oriented from 1 to 0.
We have $\delta=\alpha_0+\alpha_1$ and
$$\Phi_{+-}=\{\gamma_n\,;\,n\in\bbN\}\quad,\quad
\Phi_{++}=\{\beta_n\,;\,n\in\bbN\}\quad,\quad
\gamma_n=\alpha_1+n\delta\quad,\quad\beta_n=\alpha_0+n\delta.$$
We fix the total convex preorder on the set $\Phi_+$ such that we have
$$\gamma_0>\gamma_1>\gamma_2>\cdots>\bbZ_{>0}\,\delta>\cdots>\beta_2>\beta_1>\beta_0
\quad,\quad
m\delta>n\delta
\quad,\quad
\forall m,n>0.$$
Fix an element $\beta$ in the set
$Q_{++}=\{n\alpha_1+m\alpha_0\in Q_+\,;\,n,m,m-n\geqslant 0\}.$ We write 
\begin{align}\label{beta}
\beta=n\alpha_1+m\alpha_0=r\alpha_0+n\delta\quad,\quad r=m-n. 
\end{align}
We have
$d_\beta=2nm$ and $G_\beta=GL_{n,F}\times GL_{m,F}.$
The set $\KP\!_\beta$ is equipped with the partial order $\leqslant$ defined in \eqref{order}, \eqref{bilex}.
Let $\Gamma\!_\beta$ be the order ideal consisting of all Kostant partitions of $\beta$ supported on the set $\Phi_{++}$.
An element of $\Gamma\!_\beta$ is a decreasing sequence $\pi$ in $\Phi_{++}$ of the form
$\pi=((\beta_l)^{p_l},\dots,(\beta_1)^{p_1},(\beta_0)^{p_0})$ such that
$r_\pi=r$ and $n_\pi=n$.

\smallskip

\subsection{The degeneration order for $\bfQ$}
The indecomposable $F\bfQ$-modules are partitioned into  the preprojective, preinjective 
and regular modules.
For each $k\in\bbN$, there is a unique indecomposable preprojective $\calP_k$ with dimension vector $\beta_k$ 
and an unique indecomposable preinjective $\calI_k$ with dimension vector $\gamma_k$.
The regular indecomposables have dimension 
$(k+1)\delta$, and any such representation, denoted by $\calR_{k,z}$, is labelled by a point $z\in\bbP^1(F)$.
It is well-known that for any $k\leqslant h$, we have
\begin{align}\label{preprojective}
\Hom_{F\bfQ}(\calP_{k}\,,\,\calP_h)\simeq\k^{h-k+1}\quad,\quad
\Hom_{F\bfQ}(\calP_{h+1}\,,\,\calP_k)=0\quad,\quad\Ext^1_{F\bfQ}(\calP_k\,,\,\calP_h)=0.
\end{align}
Any representation $M\in X_{\beta}(F)$ is isomorphic to a representation of the form
\begin{align}\label{M}
M=\bigoplus_{k\geqslant 0}(\calI_k)^{\oplus\, i_{M,k}}\oplus
\bigoplus_{k\geqslant 0}\bigoplus_{z\in\bbP^1(F)}(\calR_{k,z})^{\oplus\, r_{M,k,z}}
\oplus\bigoplus_{k\geqslant 0}(\calP_k)^{\oplus\, p_{M,k}}\quad,\quad
p_{M,k}\,,\,r_{M,k,z}\,,\,i_{M,k}\in\bbN.
\end{align}
Set $r_{M,k}=\sum_{z\in\bbP^1(F)}r_{M,k,z}$ for each $k$. We define the partitions
\begin{align*}
p_M&=(p_{M,\geqslant 0}\,,\,p_{M,\geqslant 1}\,,\,\dots)\quad,\quad
i_M=(i_{M,\geqslant 0}\,,\,i_{M,\geqslant 1}\,,\,\dots),\\
r_M&=(r_{M,\geqslant 0}\,,\,r_{M,\geqslant 1}\,,\,\dots)\quad,\quad
r_{M,z}=(r_{M,\geqslant 0,z}\,,\,r_{M,\geqslant 1,z}\,,\,\dots).
\end{align*}
The \emph{rank} of $M$ is the integer $\nrk_M=n-i_{M,\geqslant 0}=m-p_{M,\geqslant 0},$
and its \emph{type} is the Kostant partition 
$$\pi=\big((\gamma_0)^{i_{M,0}},(\gamma_1)^{i_{M,1}},\dots,\delta^{\,r_{M,0}},(2\delta)^{r_{M,1}},\dots,
(\beta_1)^{p_{M,1}},(\beta_0)^{p_{M,0}}\bigr).$$
Note that
$\pi\in\Gamma\!_\beta$ if and only if $i_M=r_M=\emptyset.$
For each partitions $\lambda=(\lambda_a)$, $\mu=(\mu_a)$ and integer $s$ put
$$\lambda+s=(\lambda_1+s,\lambda_2+s,\dots)\quad,\quad
n_\lambda=\sum_{a\geqslant 1}\lambda_a\quad,\quad
\lambda\trianglelefteq\mu\Longleftrightarrow
\sum_{a=1}^b\lambda_a\leqslant\sum_{a=1}^b\mu_a\,,\,\forall b.$$
Let $Y_M\subset X_\beta$ be the $G_\beta$-orbit of $M$.
Necessary and sufficient conditions for the inclusion of orbits closures are given by Pokrzywa's theorem, 
see, e.g., \cite[thm.~3.1]{EEK99} for a more recent formulation.

\begin{proposition}\label{prop:degen}
For each  $M,N\in X_{\beta}(F)$, we have
\begin{align*}
Y_N\subset\bar Y_M\iff \begin{cases}
p_N+\nrk_N\trianglelefteq p_M+\nrk_M,\\
r_{N,z}+i_{N,\geqslant 0}\trianglerighteq r_{M,z}+i_{M,\geqslant 0},\quad\forall z\in\bbP^1(F),\\
i_N+\nrk_N\trianglelefteq i_M+\nrk_M.
\end{cases}
\end{align*}
\qed
\end{proposition}

\smallskip

To each Kostant partition $\pi=((\beta_l)^{p_l},\dots,(\beta_1)^{p_1},(\beta_0)^{p_0})$
in $\Gamma\!_\beta$ we associate the preprojective $F\bfQ$-module of type $\pi$ given by
$\calP_\pi=(\calP_l)^{\oplus\, p_l}\oplus\cdots\oplus(\calP_1)^{\oplus\, p_1}
\oplus(\calP_0)^{\oplus\, p_0}$.
For any subset $\Gamma\subset\Gamma\!_\beta$ let
$$Y_\Gamma=\bigsqcup_{\pi\in\Gamma}Y_{\calP_\pi}\quad,\quad
\calY_\Gamma=[Y_\Gamma\,/\,G_\beta].$$
We'll need the following locally closed inclusions 
$$i_\Gamma:Y_\Gamma\to Y_\beta
\quad,\quad 
j_\beta:Y_\beta\to X_\beta
\quad,\quad 
j_\Gamma=j_\beta\circ i_\Gamma:Y_\Gamma\to X_\beta
.$$
Write $\calM_\Gamma=(j_\Gamma)^*\calL_\beta$ and $\calM_\beta=\calM_{\Gamma_\beta}$.
We abbreviate  $i_\pi=i_{\{\pi\}}$, $j_\pi=j_{\{\pi\}}$ and
$$Y_\beta=Y_{\Gamma\!_\beta}\quad,\quad \calY_\beta=\calY_{\Gamma\!_\beta}\quad,\quad
Y_\pi=Y_{\{\pi\}}\quad,\quad
\calY_\pi=\calY_{\{\pi\}}.$$

\smallskip

\begin{proposition} \label{prop:inftyX} \hfill
\begin{itemize}[leftmargin=8mm]
\item[$\mathrm{(a)}$] $i_\Gamma$, $j_\Gamma$ are open for each
order ideal $\Gamma\subset\Gamma\!_\beta$.
\item[$\mathrm{(b)}$] 
$Y_\sigma\subset \bar Y_\pi\Longrightarrow\sigma\geqslant\pi$ for each $\pi,\sigma\in\Gamma\!_\beta$.
\end{itemize}
\end{proposition}

\begin{proof}
To prove (b), let $\pi,\sigma\in\Gamma\!_\beta$ and $M,N\in X_\beta(F)$ be of type $\pi,\sigma$.
Then, we have $\nrk_N=\nrk_M$.
By Proposition \ref{prop:degen}, we have
$Y_N\subset \overline Y_M$ if and only if $p_N\trianglelefteq p_M.$
Fix an integer $l$ such that $p_{M,>l}=p_{N,>l}=0$ and consider the partitions
$$\lambda_M=\big(l^{p_{M,l}},\dots,1^{p_{M,1}},0^{p_{M,0}}\big)
\quad,\quad
\lambda_N=\big(l^{p_{N,l}},\dots,1^{p_{N,1}},0^{p_{N,0}}\big).$$
The partitions $1+\lambda_M$, $1+\lambda_N$ are transpose of the partitions $p_M$, $p_N$.
Note that we have $\sum_kp_{M,k}=\sum_kp_{N,k}$.
We deduce that
\begin{align*}
Y_N\subset \overline Y_M
\iff \lambda_N\trianglerighteq \lambda_M
\iff \lambda_N-\lambda_M\text{\ is\ a\ sum\ of\ positive\ roots}
\ \,\Longrightarrow\,\,\sigma\geqslant\pi.
\end{align*}
By (b), to prove (a) it is enough to check that the inclusion $Y_\beta\subset X_\beta$ is open.
To do so, note that a representation $M\in X_\beta(F)$ is preprojective if and only if
$\Hom_{F\bfQ}(R,M)=0$ for any regular representation $R$ of dimension $\delta$.
Note also that $\Hom_{F\bfQ}(R,M)$ is the fiber at $(R,M)$ of the projection
$$\{(\phi,R,M)\,;\,\phi\in\Hom_{F\bfQ}(R,M)\,,\,R\in X_\delta\,,\,M\in X_\beta \}\to 
X_\delta\times X_\beta,$$
and the left hand side is the set of closed points of a closed $F$-subset of
$\bbA^{m+n}_F\times X_\delta\times X_\beta$.
Chevalley's semi-continuity theorem implies that
the set of all $(R,M)$ such that $\Hom_{F\bfQ}(R,M)=0$ is open. 
Since the moduli space of the regular representations of dimension $\delta$ is isomorphic to $\bbP^1_F$, hence
is projective, we deduce that the set of all $M\in X_\beta(F)$ such that $\Hom_{F\bfQ}(R,M)=0$ for any regular 
representation $R$ of dimension $\delta$ is also open.
\end{proof}

\smallskip

\begin{corollary}
$(j_\beta)^*\calL(\pi)=0$ unless $\pi\in\Gamma\!_\beta$, hence
$\calM_\beta=\bigoplus_{\pi\in\Gamma\!_\beta}(j_\beta)^*\calL(\pi)$.
\end{corollary}

\begin{proof}
For each root partition $\sigma\in\Pi_\beta,$
let $X_\sigma=\bigcup_MY_M$ be the locally closed subset of 
$X_\beta$ consisting of the union of the orbits of all representations $M$ of type $\sigma$
such that the regular part of $M$ is semisimple.
By \cite{L92}, each perverse sheaves in the set $\{IC(\pi)\,;\,\pi\in\KP\!_\beta\}$
is the intermediate extension of a local system on a dense open subset of the stratum
$X_\sigma$ for some root partition $\sigma$.
Thus the corollary follows from Proposition \ref{prop:inftyX}(a).
\end{proof}

\smallskip

\subsection{The category $\calD_\beta$ and the moduli stack of representations of $\bfQ$}

Let $\Gamma$ be any order ideal of $\Gamma\!_\beta$.
The additive category $\Par(\calY_\Gamma)$ is generated by 
the summands of the complex $\calM_\Gamma$ and their cohomological shifts.
Passing to the homotopy categories, 
we get the graded triangulated category
$\D^\b_\mix(\calY_\Gamma)=\K^\b(\Par(\calY_\Gamma)).$ 
The maps $i_\Gamma$ and $j_\Gamma$ are open inclusions.
Thus, we have the restriction functor 
$(i_\Gamma)^*:\Par(\calY_\beta)\to\Par(\calY_\Gamma)$
which yields the graded triangulated functor
$(i_\Gamma)^*:\D^\b_\mix(\calY_\beta)\to\D^\b_\mix(\calY_\Gamma).$
For each element $\pi\in\Gamma$, let $IC(\pi)_\mix$ be the complex
$(j_\beta)^*IC(\pi)$ in $\Par(\calY_\beta)$ viewed as an object of  $\D^\b_\mix(\calY_\beta)$.
In \S\ref{sec:KLR} we have introduced graded subcategories $\calD_\beta$,
$\calD_\Gamma$ of the graded module category $\calC_\beta$
and the graded triangulated functor 
$L(h_\Gamma)^*:\D^\b(\calD_\beta)\to\D^\b(\calD_\Gamma).$

\smallskip

\begin{proposition}\label{prop:equivalence1}
Assume that $\Gamma\subset\Gamma\!_\beta$ is an order ideal.
\hfill
\begin{itemize}[leftmargin=8mm]
\item[$\mathrm{(a)}$] 
There is an equivalence of graded triangulated categories 
$\A''_\Gamma:\D^\b(\calD_\Gamma)\to\D^\b_\mix(\calY_\Gamma)$
such that
$L(h_\Gamma)^*Q(\pi)\mapsto (i_\Gamma)^*IC(\pi)_\mix$ for all $\pi\in\Gamma.$
\item[$\mathrm{(b)}$]
There is a right adjoint triangulated functor $(i_\Gamma)_*$ to $(i_\Gamma)^*$ yielding a commutative
diagram  
\begin{align*}\begin{split}
\xymatrix{
\D^\b(\calD_\beta)\ar[d]_-{\A''_\beta}
\ar@/^/[r]^-{L(h_\Gamma)^*}&\ar@/^/[l]^-{(h_\Gamma)_*}
\D^\b(\calD_\Gamma)\ar[d]^-{\A''_\Gamma}
\\
\D^\b_\mix(\calY_\beta)
\ar@/^/[r]^-{(i_\Gamma)^*}&\ar@/^/[l]^-{(i_\Gamma)_*}
\D^\b_\mix(\calY_\Gamma).
}\end{split}\end{align*}
\end{itemize}
\end{proposition}

\smallskip

The proof of the proposition consists of checking that the graded $\k$-algebra
\begin{align}\label{inftyVV}
\bfS_\Gamma=\End^\bullet_{\D^\b(\calY_\Gamma)}(\calM_\Gamma)
\end{align}
is the graded-endomorphism
ring of a graded-projective generator of the category $\calD_\Gamma$.
We'll abbreviate $\bfS_\beta=\bfS_{\Gamma\!_\beta}.$
To prove this, we introduce the following complexes in $\D^\b(\calX_\beta)$
\begin{align*}
\nabla(\k_\pi)=(j_\pi)_*\k_\pi[d_\pi]\quad&,\quad
IC(\k_\pi)=(j_\pi)_{!*}\k_\pi[d_\pi],\\
\Delta(\k_\pi)=(j_\pi)_!\k_\pi[d_\pi]\quad&,\quad
Q(\k_\pi)=(j_\beta)_*(j_\beta)^*IC(\k_\pi),
\end{align*}
where $d_\pi=\dim Y_\pi$ and $\k_\pi=\k_{Y_\pi}$ for each $\pi\in\Gamma\!_\beta$.
We consider the composed functor 
\begin{align*}\Psi_\Gamma^\bullet=\Phi_\beta^\bullet\, (j_\Gamma)_*:\D^\b(\calY_\Gamma)\to\calC_\beta\quad,\quad 
\calE\mapsto\Hom^\bullet_{\D^\b(\calX_\beta)}(\calL_\beta,(j_\Gamma)_*\calE)=
\Hom^\bullet_{\D^\b(\calY_\Gamma)}(\calM_\Gamma,\calE).\end{align*}

\smallskip

\begin{lemma}\label{lem:A}
For each $\pi\in\Gamma\!_\beta$
we have the following isomorphisms of graded $\bfR_\beta$-modules\hfill
\hfill
\begin{itemize}[leftmargin=8mm]
\item[$\mathrm{(a)}$] $\Phi_\beta^\bullet(IC(\k_\pi))\simeq P(\pi)$,
\item[$\mathrm{(b)}$] $\Phi_\beta^\bullet(\nabla(\k_\pi))\simeq \Delta(\pi)$,
\item[$\mathrm{(c)}$] $\Phi_\beta^\bullet(Q(\k_\pi))\simeq Q(\pi)$.
\end{itemize}
\end{lemma}

\begin{lemma}\label{lem:AA}
We have an equivalence of graded Abelian categories
$\calD_\Gamma\to\bfS_\Gamma\-\mod$.
The functor
$\Psi_\Gamma^\bullet$ 
gives an equivalence of graded additive categories $\Par(\calY_\Gamma)\to\calD^\proj_\Gamma$
such that $(j_\Gamma)^*IC(\pi)\mapsto(f_\Gamma)^*P(\pi)$ for each $\pi\in\Gamma$.
\end{lemma}

\begin{proof}[Proof of Proposition  $\ref{prop:equivalence1}$]
Part (a) of the proposition follows from Proposition \ref{prop:Cinfty} and Lemma \ref{lem:AA}, 
taking the homotopy categories.
Applying Lemma \ref{lem:AA} both to $\Gamma$ and $\Gamma\!_\beta$,
we get two functors $A''_\Gamma$ and $A''_\beta$ such that
$A''_\Gamma\, L(h_\Gamma)^*=(i_\Gamma)^*\, A''_\beta$.
We define $(i_\Gamma)_*$ as the unique functor such that the square in part (b) of the proposition commutes.
\end{proof}

\smallskip

\begin{proof}[Proof of Lemma $\ref{lem:A}$]
By \cite[thm.~6.16]{L92}, the complexes $IC(\k_\pi)$ with $\pi\in\Gamma\!_\beta$ belong to the canonical basis
for each $\beta\in Q_+$, i.e., there is a map $\Gamma\!_\beta\to\K\P\!_\beta$, $\pi\mapsto\pi^\flat$ such that
$IC(\k_\pi)= IC(\pi^\flat).$ 

\medskip

\noindent{\bf Step 1 :} \emph{We prove $\mathrm{(a), (b)}$ in the case $r=1$.}

\medskip

Let $r=1$. Hence $\beta$ is the positive root $\beta_n$.
Further, we have $\Gamma\!_\beta=\{(\beta)\}$ and
$\pi\geqslant (\beta)$ for all $\pi\in\K\P\!_\beta$.
We abbreviate $\k_\beta=\k_{X_\beta}$. 
The stratum $Y_{(\beta)}$ is the open dense $G_\beta$-orbit in $X_\beta$, and it coincides with the open subset 
$Y_\beta$.
We deduce that $\k_{(\beta)}=\k_{Y_\beta}$ and 
\begin{align}\label{smooth}
\Delta(\k_{(\beta)})=(j_\beta)_!\k_{(\beta)}[d_\beta]\quad,\quad
\nabla(\k_{(\beta)})=(j_\beta)_*\k_{(\beta)}[d_\beta]\quad,\quad IC(\k_{(\beta)})=\k_\beta[d_\beta].
\end{align}
To prove (a) for $r=1$ we must check that $(\beta)^\flat=(\beta)$.

\smallskip

\begin{claim} 
If $\alpha,\gamma\neq 0$ and the condition $\alpha\prec\beta\prec\gamma$ does not hold, 
then the perverse sheaf $IC(\k_{(\beta)})[a]$ is not a direct summand of 
$IC(\pi)\circledast IC(\sigma)$ whenever $\pi\in\KP\!_\gamma$, $\sigma\in\KP\!_\alpha$ and $a\in\bbZ$. 
\qed
\end{claim}

\smallskip

Indeed, for each $F\bfQ$-modules $M\in X_\alpha(F)$, $N\in X_\gamma(F)$ with an exact sequence 
$$0\to M\to\calP_n\to N\to 0,$$ 
the module $M$ is a sum of $\calP_m$'s for some integers $m<n$ 
because $\Hom_{F\bfQ}(Q,\calP_n)=0$ whenever $Q$ is preinjective or regular, and the module $N$
is a sum of regular and preinjective modules, because 
$\Hom_{F\bfQ}(\calP_n\,,\,\calP_k)=0$ for all $k<n$ by \eqref{preprojective}.
Hence, we have $\alpha\prec\beta\prec\gamma$.

\smallskip

By Proposition \ref{prop:ind},
for all $\sigma$, $\pi$ as above we have
$\bbHom_{\calC_\beta}(P(\sigma)\circ P(\pi)\,,\,L((\beta)^\flat))=0,$
and, by adjunction, that
$$\bbHom_{\calC_\alpha\times\calC_\gamma}(P(\sigma)\otimes P(\pi)\,,\,\Res_{\alpha,\gamma}L((\beta)^\flat)=0.$$
Thus we have $\Res_{\alpha,\gamma}L((\beta)^\flat)=0,$
hence the module $L((\beta)^\flat)$ is cuspidal, so it is  $L(\beta)$.

\smallskip

Now, we prove (b) for $r=1$.
The obvious morphism of complexes
$IC(\k_{(\beta)})\to\nabla(\k_{(\beta)})$ yields an $\bfR_\beta$-module homomorphism
\begin{align}\label{rho}
\Phi_\beta^\bullet(IC(\k_{(\beta)}))\to\Phi_\beta^\bullet(\nabla(\k_{(\beta)})).
\end{align} 
We have
\begin{align}\label{id1}
\begin{split}
\Phi_\beta^\bullet(IC(\k_{(\beta)}))&=\Phi_\beta^\bullet(\k_\beta[d_\beta]),\\
&=H^\bullet_{G_\beta}(X_\beta\,,\,\calL_\beta)[-d_\beta],\\
&=\big(V(\beta)\otimes H^\bullet_{G_\beta}(X_\beta\,,\,\k)\big)\oplus\bigoplus_{\pi\neq(\beta)}
\big(V(\pi)\otimes H^\bullet_{G_\beta}(X_\beta\,,\,IC(\pi)\big)[-d_\beta],
\end{split}
\intertext{and}
\label{id2}
\begin{split}
\Phi_\beta^\bullet(\nabla(\k_{(\beta)}))&=H^\bullet_{G_\beta}(Y_{\beta}\,,\,\calM_\beta)[-d_\beta]
=V(\beta)\otimes H^\bullet_{G_\beta}(Y_{\beta}\,,\,\k).
\end{split}
\end{align}
Under the isomorphisms \eqref{id1} and \eqref{id2} the map \eqref{rho} is
identified with the restriction
$H^\bullet_{G_\beta}(X_\beta\,,\,\calL_\beta)\to 
H^\bullet_{G_\beta}(Y_{\beta}\,,\,\calM_\beta).$
The restriction to $Y_\beta$ gives also a map
\begin{align}\label{surj1}H^\bullet_{G_\beta}(X_\beta\,,\,\k)\to H^\bullet_{G_\beta}(Y_{\beta}\,,\,\k).\end{align}
Let $D_\beta$ be the diagonal copy of $\bbG_{m}$ in $G_\beta$.
It is the stabilizer of any point of $Y_{(\beta)}$.
Hence, we have
$H^\bullet_{G_\beta}(X_\beta\,,\,\k)=H^\bullet_{G_\beta}$ and
$H^\bullet_{G_\beta}(Y_{\beta}\,,\,\k)=H^\bullet_{D_\beta}$.
We deduce that the map \eqref{surj1} is surjective, hence \eqref{rho} is also surjective.
So, to prove that $\Phi_\beta^\bullet(\nabla(\k_{(\beta)}))=\Delta(\beta)$
it is enough to check that
$$\Ext^1_{\calC_\beta}(\Phi_\beta^\bullet(\nabla(\k_{(\beta)})),L(\pi))=0\quad,\quad\forall \pi\leqslant(\beta),$$
or, equivalently, that 
$\Ext^1_{\calC_\beta}(\Phi_\beta^\bullet(\nabla(\k_{(\beta)})),L(\beta))=0.$
To do that, we apply the functor $\Hom_{\calC_\beta}(\bullet,L(\beta))$ to the 
short exact sequence in $\calC_\beta$ given by
$$\xymatrix{
0\ar[r]& J(\beta)\ar[r]&P(\beta)\ar[r]^-{\eqref{rho}}& \Phi_\beta^\bullet(\nabla(\k_{(\beta)}))\ar[r]& 0}.$$
We must check that
\begin{align}\label{eq1}\Hom_{\calC_\beta}(J(\beta),L(\beta))=0.\end{align}
Let $H^\bullet_+$ be the kernel of the restriction map $H^\bullet_{G_\beta}\to H^\bullet_{D_\beta}$.
We have
$$J(\beta)=\big(V(\beta)\otimes H^\bullet_+\big)\oplus\bigoplus_{\pi\neq(\beta)}
\big(V(\pi)\otimes H^\bullet_{G_\beta}(X_\beta\,,\,IC(\pi)\big)[-d_\beta].$$
Since $H^0_+=0$, we do have \eqref{eq1} for degree reasons.

\medskip

\noindent{\bf Step 2 :} \emph{We prove $\mathrm{(b)}$ for any $r$.}

\medskip

Fix a tuple 
$\pi=((\beta_0)^{p_0},\dots,(\beta_{l-1})^{p_{l-1}},(\beta_l)^{p_l})$ in $Q_+^\beta$.
Taking the parts of $\pi$ in the reverse order we get a Kostant partition $\pi^\op$ in $\Gamma\!_\beta$.
Consider the open subsets $\calX_{(\pi)}\subset\calX_\pi$ and $\widetilde\calX_{(\pi)}\subset\widetilde\calX_\pi$ given by
$$\calX_{(\pi)}=(\calY_{(\beta_0)})^{p_0}\times\dots\times(\calY_{(\beta_l)})^{p_l}
\quad,\quad\widetilde\calX_{(\pi)}=q_\pi^{-1}(\calX_{(\pi)}).$$
Comparing \eqref{fiber} with the third equality in \eqref{preprojective}, 
we deduce that the restriction of the map $q_\pi$ to
$\widetilde\calX_{(\pi)}$
is a gerbe with group the unipotent radical of $P_\pi$.
The second equality  in \eqref{preprojective} implies that the restriction of the map $p_\pi$ to $\widetilde\calX_{(\pi)}$
 is a closed embedding with image $\calY_\pi.$
In other words, any extension of
$\calP_l, \dots\calP_l,\calP_{l-1},\dots,\calP_0$ in that order
(meaning that $\calP_0$ is a subobject and $\calP_l$ a quotient object)
is necessarily trivial, hence isomorphic to $\calP_\pi$.
Further, 
the representation $\calP_\pi$ preserves a unique flag in $V$ which is conjugate to $V_\bullet$ 
under the action of $G_\beta$.
Thus, for each $k=0,\dots,l$, we get 
\begin{align}\label{IND5}
\begin{split}
\ind_{((\beta_k)^{p_k})}\big(\Delta(\k_{(\beta_k)})^{\otimes p_k}\big)=
\Delta(\k_{((\beta_k)^{p_k})})\langle p_k\rangle\,!,\\
\ind_{\bar\pi}\big(\Delta(\k_{((\beta_0)^{p_0})})\otimes\dots\otimes
\Delta(\k_{((\beta_{l-1})^{p_{l-1}})})\otimes\Delta(\k_{((\beta_l)^{p_l})})\big)
&=\Delta(\k_\pi).
\end{split}
\end{align}
From \eqref{Delta0}, \eqref{Delta}, \eqref{IND5}, Proposition \ref{prop:ind} and the isomorphism 
$\Phi_\beta^\bullet(\nabla(\k_{(\beta_k)}))\simeq\Delta(\beta_k)$
proved in Step 1  for each $k\in[0,l],$ we deduce that  
$\Phi_\beta^\bullet\big(\nabla(\k_\pi))=\Delta(\pi)$,
proving the part (b) of the lemma for any $r$.

\medskip

\noindent{\bf Step 3 :} \emph{We prove $\mathrm{(a)}$ for any $r$.}

\medskip

Step 2 implies that for each $\sigma\in\Gamma_\beta$ we have
\begin{align*}
\Delta(\sigma)&=\Phi_\beta^\bullet(\nabla(\k_\sigma))
=\bigoplus_{\pi\in\KP\!_\beta}V(\pi)\otimes H^\bullet_{G_\beta}(Y_\sigma\,,\,(j_\sigma)^*IC(\pi)).
\intertext{We must check that for all  $\pi\in\Gamma\!_\beta$
we have $\pi^\flat=\pi$. We have}
[\Delta(\sigma)\,:\,L(\pi^\flat)]&=\dim\,\e(\pi^\flat)\,\Delta(\sigma),\\
&=\dim\,H^\bullet_{G_\beta}(Y_\sigma\,,\,(j_\sigma)^*IC(\pi^\flat)),\\
&=\dim\,H^\bullet_{G_\beta}(Y_\sigma\,,\,(j_\sigma)^*IC(\k_\pi)),\\
&=\begin{cases}0&\text{if\ }Y_\sigma\not\subseteq\bar Y_\pi,\\
\neq 0&\text{if\ } \sigma=\pi.\end{cases}
\intertext{The first relation and Proposition \ref{prop:inftyX}(b) yield}
[\Delta(\pi^\flat)\,:\,L(\pi^\flat)]\neq0
&\Longrightarrow Y_{\pi^\flat}\subseteq\bar Y_\pi
\Longrightarrow \pi^\flat\geqslant\pi.
\end{align*}
The second relation yields $\pi\geqslant\pi^\flat$.
Hence, we have $\pi=\pi^\flat$ as wanted.

\medskip

\noindent{\bf Step 4 :} \emph{We construct
a surjective $\bfR_\beta$-module homomorphism 
$P(\pi)\to\Phi_\beta^\bullet(Q(\k_\pi))$.}

\medskip

We'll need mixed analogues of the complexes above. 
Let $X_{\beta,0}$ be the variety of all 
$\beta$-dimensional representations of the path algebra $F_0\bfQ$.
The linear algebraic $F_0$-group $G_{\beta,0}$ acts on $X_{\beta,0}$ and we can consider the quotient $F_0$-stack
$\calX_{\beta,0}=[X_{\beta,0}\,/\,G_{\beta,0}]$. 
Since the group $G_\beta$ is connected, by Lang's theorem we have
$\calX_\beta=\calX_{\beta,0}\otimes_{F_0}F.$ 
We consider the mixed complex 
$$IC(\pi)_\bbd=IC(\k_\pi)_\bbd=(j_\pi)_{!*}\k_\pi\langle d_\pi\rangle.$$
It is pure of weight 0.
We abbreviate
$\calE_\bbd=(j_\beta)_*(j_\beta)^*IC(\pi)_\bbd.$
The complex $\calE=\omega(\calE_\bbd)$ is $Q(\k_\pi)$.
The adjunction yields a canonical map $IC(\pi)_\bbd\to\calE_\bbd$, hence by functoriality a map
\begin{align}\label{map3}\Phi_\beta^\bullet(IC(\pi))\to\Phi_\beta^\bullet(\calE).\end{align}

\smallskip

\begin{claim}\label{claim:triangle3}
There is a distinguished triangle
$\xymatrix{\ar[r]& IC(\pi)_\bbd\ar[r]&\calE_\bbd\ar[r]&(\calE_\bbd)_{>0}.}$
\end{claim}

\begin{proof} 
By \cite[\S 5.1.14\,,\,prop.~ 1.4.12]{BBD}, the mixed complex
$\calE_\bbd$ is in $\D^\b_{\geqslant 0}(\calX_{\beta,0}) 
\cap {}^p\D^\b(\calX_{\beta,0})^{\geqslant 0}.$
By Proposition \ref{prop:triangle}, we have a distinguished triangle
\begin{align}\label{DT3}\xymatrix{\ar[r]&(\calE_\bbd)_0\ar[r]^-{f_\bbd}&\calE_\bbd\ar[r]&(\calE_\bbd)_{>0}.}
\end{align}
Write $\calE_{>0}=\omega((\calE_\bbd)_{>0})$, $\calE_0=\omega((\calE_\bbd)_0)$ and $f=\omega(f_\bbd)$.
We have short exact sequences
\begin{align*}
\xymatrix{0\ar[r]&{}^p\!H^a(\calE_0)\ar[r]^-{{}^p\!H^a(f)}&{}^p\!H^a\calE\ar[r]&{}^p\!H^a(\calE_{>0})\ar[r]& 0}
\quad,\quad\forall a\in\bbZ.
\end{align*}
By \cite[ex.~III.10.3]{KW} we have 
\begin{align}\label{zero}{}^p\!H^0(\calE_0)=IC(\pi).\end{align}
Since the mixed complex $(\calE_\bbd)_0$ is pure, we have
\begin{align*}\calE_0=IC(\pi)\oplus\bigoplus_{a>0}{}^p\!H^a(\calE_0)[-a].\end{align*}
Let $a>0$. If ${}^p\!H^{a}(\calE_0)\neq 0$ then the restriction of
$f$ to the summand ${}^p\!H^{a}(\calE_0)[-a]$ is nonzero.
Since $\calE=(j_\beta)_*(j_\beta)^*IC(\pi)$, this yields a nonzero map 
$$(j_\beta)^*{}^p\!H^{a}(\calE_0)[-a]\to (j_\beta)^*IC(\pi).$$
This is absurd by \eqref{zero} because, by definition of $\calE$, we have
\begin{align}
\sum_{a\in\bbZ}[{}^p\!H^a(\calE_0)\,:\,IC(\pi)]\leqslant
\sum_{a\in\bbZ}[{}^p\!H^a\calE_\bbd\,:\,IC(\pi)_\bbd]=1.\end{align}
We deduce that $(\calE_\bbd)_0=IC(\pi)_\bbd$, and  \eqref{DT3} yields Claim \ref{claim:triangle3}.
\end{proof}

\smallskip

\begin{claim}
The map \eqref{map3} is surjective.
\end{claim}

\begin{proof}
By \cite[\S 5.3]{L98}, the complex $\calL_\beta$ has a canonical mixed structure $\calL_{\beta,\bbd}$
which is pure of weight 0. 
Set
$\calM_{\beta,\bbd}=(j_\beta)^*\calL_{\beta,\bbd}.$
The mixed complex $\calM_{\beta,\bbd}$ on $\calY_\beta$ is pure of weight 0.
We consider the following mixed complex of vector spaces in $\D^+_\bbd(\Spec F_0)$
$$\Phi_{\beta,\bbd}^\bullet(\calF_\bbd)=\HHom^\bullet_{\D^\b(\calX_{\beta,0})}(\,\calL_{\beta,\bbd}\,,\,\calF_\bbd\,)
\quad,\quad
\forall\calF_\bbd\in\D^\b_\bbd(\calX_{\beta,0}).$$ 
By \cite[cor.~3.10]{S12}, we have 
\begin{align}\label{E+}\Phi_{\beta,\bbd}^\bullet((\calE_\bbd)_{>0})\in\D^+_{>0}(\Spec F_0).\end{align}
Next, we consider the mixed complex $\Phi_{\beta,\bbd}^\bullet(\calE_\bbd)$. We have
\begin{align*}
\Phi_{\beta,\bbd}^\bullet(\calE_\bbd)=
\HHom^\bullet_{\D^\b(\calY_{\beta,0})}(\,\calM_{\beta,\bbd}\,,\,(j_\beta)^*IC(\pi)_\bbd\,)
\quad,\quad
\omega\Phi_{\beta,\bbd}^\bullet(\calE_\bbd)=\Phi_\beta^\bullet(\calE).
\end{align*} 
We'll use the notation $H_\beta$, $\calG r^+_\beta$, $\Gr^+_\beta$ and $S^+_\beta$ introduced in \S\ref{sec:Gr} below,
to which we refer for more details. In particular, we have
$\calG r^+_\beta=[\Gr^+_\beta/H_\beta]$ as a stack, and
the stratification $T^+_\beta$ is even affine by Proposition \ref{prop:even1}(b).
The image of the mixed complexes $\calM_{\beta,\bbd}$ and $(j_\beta)^*IC(\pi)_\bbd$ by 
the equivalence of triangulated categories $\D^\b_\bbd(\calY_{\beta,0})\to\D^\b_\bbd(\calG r^+_{\beta,0})$
in Proposition \ref{prop:Y=X} are very pure of weight 0.
We deduce from Proposition \ref{prop:BY} that 
$\Phi_{\beta,\bbd}^\bullet(\calE_\bbd)$ 
is free of finite rank as an $H^\bullet_{H_\beta}$-module and that
\begin{align}\label{E}\Phi_{\beta,\bbd}^\bullet(\calE_\bbd)\in\D^+_0(\Spec F_0).\end{align}
Now, we apply the functor $\Phi_{\beta,\bbd}^\bullet$
to the triangle in Claim \ref{claim:triangle3}. We get a long exact sequence of mixed vector spaces
$$\xymatrix{
\ar[r]&\Phi_{\beta,\bbd}^a(IC(\pi)_\bbd)\ar[r]^-{f^a}&\Phi_{\beta,\bbd}^a(\calE_\bbd)\ar[r]
&\Phi_{\beta,\bbd}^a((\calE_\bbd)_{>0})\ar[r]&\Phi_{\beta,\bbd}^{a+1}(IC(\pi)_\bbd)\ar[r]^-{f^{a+1}}&.
}$$
From \eqref{E+}, \eqref{E}, we deduce that the map $f^a$ is onto.
Hence, taking the sum over all integers $a$, we get that the map 
$$\omega\Phi_{\beta,\bbd}^\bullet(IC(\pi)_\bbd)=\Phi_\beta^\bullet(IC(\pi))
\to\omega\Phi_{\beta,\bbd}^\bullet(\calE_\bbd)=\Phi_\beta^\bullet(\calE)$$
in \eqref{map3} is surjective.
\end{proof}

\medskip

\noindent{\bf Step 5 :} \emph{We prove $\mathrm{(c)}$ for any $r$.}

\medskip

We must prove that $\Phi_\beta^\bullet(Q(\k_\pi))=Q(\pi)$.
For any Kostant partition $\sigma\in\K\P\!_\beta$, we have
\begin{align}\label{un}
\begin{split}
[\Phi_\beta^\bullet(Q(\k_\pi))\,:\,L(\sigma)]
&=\dim\bbHom_{\bfR_\beta}(P(\sigma)\,,\,\Phi^\bullet_\beta(Q(\k_\pi))),\\
&=\dim \e(\sigma)\Phi^\bullet_\beta(Q(\k_\pi)),\\
&=\dim\Hom^\bullet_{\D^\b(\calX_\beta)}(DQ(\k_\pi)\,,\,IC(\sigma)),\\
&=\dim\Hom^\bullet_{\D^\b(\calY_\beta)}((j_\beta)^* IC(\pi)\,,\,(j_\beta)^* IC(\sigma)),\\
&=0\ \text{if\ } \sigma\notin\Gamma\!_\beta.
\end{split}
\end{align}
Note that $(j_\beta)^*=(j_\beta)^!$ because $j_\beta$ is an open immersion.
We deduce that 
$\Phi_\beta^\bullet(Q(\k_\pi))\in\calD_\beta$.
By Step 4 the graded $\bfR_\beta$-module $P(\pi)$ maps onto $\Phi_\beta^\bullet(Q(\k_\pi))$.
Since the graded $\bfR_\beta$-module 
$Q(\pi)$ is the largest quotient of $P(\pi)$ which lies in $\calD_\beta$, 
we get a surjective graded $\bfR_\beta$-module homomorphism
\begin{align}\label{ONTO}Q(\pi)\to\Phi_\beta^\bullet(Q(\k_\pi)).\end{align}
Since the category $\calC_\beta$ is affine properly stratified, the module $P(\pi)$ has a $\Delta$-filtration for all 
$\pi\in\KP\!_\beta$. By Propositions \ref{prop:even1}, \ref{prop:Y=X} and \ref{prop:BY}, 
the extension algebra $\bfS_\beta$ in \eqref{inftyVV}
satisfies the conditions in \cite[thm.~4.1]{K12}, hence the graded $\bfS_\beta$-module $\Phi_\beta^\bullet(Q(\k_\pi))$
has an increasing filtration whose layers are isomorphic to the modules $\Phi_\beta^\bullet(\nabla(\k_\sigma))$
with $\sigma\in\Gamma\!_\beta$. Note that $\Phi_\beta^\bullet(\nabla(\k_\sigma))=\Delta(\sigma)$ by Step 2 above.
Hence, we must check that the multiplicity of $\Delta(\sigma)$ in $P(\pi)$ and $\Phi_\beta^\bullet(Q(\k_\pi))$ are the 
same for all $\sigma\in\Gamma\!_\beta$. This follows from the relations
\begin{align*}
P(\pi)&=\bigoplus_{\tau\in\KP\!_\beta}V(\tau)\otimes\Hom^\bullet(IC(\pi),IC(\tau)),\\
\Phi_\beta^\bullet(Q(\k_\pi))&=\bigoplus_{\tau\in\Gamma\!_\beta}V(\tau)\otimes
\Hom^\bullet((j_\beta)_!(j_\beta)^*IC(\pi),IC(\tau)),\\
\Delta(\sigma)&=\bigoplus_{\tau\in\Gamma\!_\beta}V(\tau)\otimes
\Hom^\bullet((j_\sigma)_!(k_\sigma[d_\sigma]),IC(\tau)).
\end{align*}

\end{proof}

\smallskip

\begin{proof}[Proof of Lemma $\ref{lem:AA}$] 

First, we prove the lemma for $\Gamma=\Gamma\!_\beta$.
Consider the graded $\bfR_\beta$-modules
$$P_\beta=\bigoplus_{\pi\in\KP\!_\beta}V(\pi)\otimes P(\pi)
\quad,\quad
Q_\beta=\bigoplus_{\pi\in\Gamma\!_\beta}V(\pi)\otimes Q(\pi).$$
The isomorphism \eqref{PPP} and Steps 3 and 5 above yield the following  isomorphisms
$$P_\beta=\Phi^\bullet_\beta(\calL_\beta)=\bfR_\beta
\quad,\quad
Q_\beta=\Psi^\bullet_\beta(\calM_\beta)=\bfS_\beta.$$
We have a surjective graded $\bfR_\beta$-module homomorphism $P_\beta\to Q_\beta$.
So, the $\k$-algebra homomorphism
\begin{align}\label{ONTO2}(j_\beta)^*:\bfR_\beta\to \bfS_\beta\end{align}
given by the restriction functor $(j_\beta)^*$ is surjective.
The $\bfR_\beta$-action on $Q_\beta$ 
factorizes through \eqref{ONTO2}
to an $\bfS_\beta$-action.
Hence, we have $\bbEnd_{\bfR_\beta}(Q_\beta)=\bfS_\beta.$
The graded $\bfR_\beta$-module $Q_\beta$
is a projective graded-generator of the category $\calD_\beta$, because
$Q(\pi)$ is the projective cover of  $L(\pi)$  in $\calD_\beta$ for each $\pi\in\Gamma\!_\beta$.
We have $\bbEnd_{\calD_\beta}(Q_\beta)=\bfS_\beta$, hence the category
$\calD_\beta$ is equivalent to $(\bfS_\beta)^\op\-\,\mod.$
Since $\bfS_\beta=(\bfS_\beta)^\op$, we get 
\begin{align}\label{DS}\calD_\beta=\bfS_\beta\-\mod.\end{align}
By Remark \ref{rem:mixte2}(a), \eqref{inftyVV} and \eqref{DS},
the functor $\Psi_\beta^\bullet$ yields an equivalence of graded additive categories 
$$\Par(\calY_\beta)=\bfS_\beta\-\,\proj=\calD^\proj_\beta$$ 
which maps $(j_\beta)^*IC(\pi)$ to $Q(\pi)$.

\smallskip

Now, we prove the lemma for any order ideal $\Gamma\subset\Gamma\!_\beta$.
By functoriality, we have the $\k$-algebra homomorphisms 
$$\xymatrix{(j_\Gamma)^*=(i_\Gamma)^*(j_\beta)^*:
\bfR_\beta\ar[r]&\bfS_\beta\ar[r]&\bfS_\Gamma.}$$
Under the equivalence of triangulated categories $\D^\b(\calY_\beta)\to \D^\b(\calG r^+_\beta)$
given by Proposition \ref{prop:Y=X}, 
the complex $\calL_\beta$ on $\calY_\beta$ is identified with an $H_\beta$-equivariant complex on
$\Gr^+_\beta$ which is parity, because the stratification $S^+_\beta$
is even by Proposition \ref{prop:even1}(b).
Hence, by \cite[cor.~2.9]{JMW}, we have
$(i_\Gamma)^*\bfS_\beta=\bfS_\Gamma.$
Since $(j_\beta)^*\bfR_\beta=\bfS_\beta$, we deduce that $(j_\Gamma)^*$ is surjective.

\smallskip

For each $\pi\in\Gamma$, the projective module $(h_\Gamma)^*Q(\pi)$ in $\calD_\Gamma$ is the cover of $L(\pi)$.
Since $\calD_\Gamma$ is a subcategory of $\calD_\beta$, 
we can view  $(h_\Gamma)^*Q(\pi)$ as a graded $\bfS_\beta$-module.
Let $\e(\pi)$ denote the image by the algebra homomorphism
$(j_\Gamma)^*$ of the idempotent $\e(\pi)$ in \eqref{idempotent}.
We claim that the $\bfS_\beta$-module $(h_\Gamma)^*Q(\pi)$ is isomorphic to the pullback by the algebra 
homomorphism $(i_\Gamma)^*$
of the projective $\bfS_\Gamma$-module $\bfS_\Gamma\, \e(\pi)$. 
Since $\bigoplus_{\pi\in\Gamma}(h_\Gamma)^*Q(\pi)$ is a projective graded-generator of
$\calD_\Gamma$, this yields an equivalence of graded Abelian categories
$\calD_\Gamma=\bfS_\Gamma\-\mod$.
We deduce that there is an equivalence of graded additive categories
$\calD^\proj_\Gamma=\bfS_\Gamma\-\,\proj$ and
$\bfS_\Gamma\-\,\proj=\Par(\calY_\Gamma)$, proving the lemma.

\smallskip

Now, we prove the claim.
As graded $\bfR_\beta$-modules, we have
\begin{align*}
Q(\pi)=\Hom^\bullet_{\D^\b(\calY_\beta)}((j_\beta)^*IC(\pi)\,,\,\calM_\beta)=\bfS_\beta \e(\pi).
\end{align*}
Consider the $\bfS_\beta$-module 
$Q(\pi)_\Gamma=\Psi_\Gamma^\bullet(j_\Gamma)^*IC(\pi)$. We have
\begin{align*}
Q(\pi)_\Gamma&=\Hom^\bullet_{\D^\b(\calY_\Gamma)}((j_\Gamma)^*IC(\pi)\,,\,\calM_\Gamma)
=\bfS_\Gamma \e(\pi).
\end{align*}
Thus the functor $(i_\Gamma)^*$ yields a $\bfS_\beta$-module homomorphism $Q(\pi)\to Q(\pi)_\Gamma$.
It is  surjective because the algebra homomorphism
$(i_\Gamma)^*:\bfS_\beta\to\bfS_\Gamma$ is surjective.
We also have a surjective $\bfS_\beta$-module homomorphism $Q(\pi)\to (h_\Gamma)^*Q(\pi)$.
We claim that 
$(h_\Gamma)^*Q(\pi)= Q(\pi)_\Gamma$
as graded $\bfS_\beta$-module.
This is proved as in Step 5 above.
We first observe that both modules belong the subcategory $\calD_\Gamma$ of $\calD_\beta$
and have $\Delta$-filtrations, using \cite{K12}.
Further, the multiplicities of the standard modules of $\calD_\Gamma$ in $(h_\Gamma)^*Q(\pi)$, 
$Q(\pi)_\Gamma$ are the same.
\end{proof}

\medskip

\section{The affine Grassmannians}
\label{sec:Gr}

Recall that $F$ is the algebraic closure of the finite field $F_0$. 
Write $ O=F[[t]]$, $K=F((t))$ and $O^-=F[t^{-1}].$
Let $\beta\in Q_{++}$ be as in \eqref{beta}. We abbreviate $G_r=GL_{r,F}$.
Let $\calA lg_F$ be the category of commutative $F$-algebras.
Let $\calG_r(O)$, $\calG_r(O^-)$, $\calG_r(K)$ be the presheaves of groups defined,
for any $R\in\calA lg_F$, by
$$\calG_r(O)(R)=G_r(R[[t]])\quad,\quad \calG_r(O^-)(R)=G_r(R[t^{-1}])\quad,\quad 
\calG_r(K)(R)=G_r(R((t))).$$
Let $(T_r,B_r)$ be the standard Borel pair in $G_r$.
If there is no confusion, we may abbreviate $T=T_r$ and $B=B_r$.
We denote the sets of characters of $T$, 
of dominant characters and of dominant characters with non negative entries by
$\Lambda^+_r\subset\Lambda_r\subset\bbZ^r.$
We equip $\Lambda_r$ with the partial order such that $\lambda\geqslant\mu$ whenever
$\lambda-\mu$ is a sum of positive roots. 
An interval of $\Lambda_r$ is a subset of the form
$$[\lambda\,,\,\mu]=\{\geqslant\!\lambda\}\cap\{\leqslant\!\mu\}\quad,\quad\lambda,\mu\in\Lambda_r,$$
where
$\{\leqslant\!\lambda\}=\{\nu\in\Lambda_r\,;\,\mu\leqslant\lambda\}$ and
$\{\geqslant\!\mu\}=\{\nu\in\Lambda_r\,;\,\nu\geqslant\mu\}.$
Consider the order ideal 
$\Lambda_\beta=\{\leqslant\!n\omega_1\}$
where 
$\omega_i=(1,\dots 1, 0,\dots,0)$ has $i$ entries equal to 1 for each $i\in[1,r]$.
We have
\begin{align*}
\Lambda_\beta=\{\lambda\in\Lambda^+_r\,;\,n_\lambda=n\}
\end{align*}
where 
$n_\lambda=\lambda_1+\cdots+\lambda_r$ for each
$\lambda=(\lambda_1,\dots,\lambda_r).$
We'll identify the sets  $\Gamma\!_\beta$ and $\Lambda_\beta$ via the bijection
\begin{align}\label{bij} \Gamma\!_\beta\to \Lambda_\beta\quad,\quad
\pi=((\beta_l)^{p_l},\dots,(\beta_1)^{p_1},
(\beta_0)^{p_0})\mapsto\lambda_\pi=
(l^{p_l},\dots,1^{p_1},0^{p_0}).
\end{align}
Let $W$ be the Weyl group of $G_r$, 
let $\widehat W$ be its affine Weyl group, and $M$ 
the set of maximal length representatives in $\widehat W$ of the left cosets in 
$\widehat W\,/\,W$.
For each weight $\lambda\in\Lambda_r$, let $P_\lambda\subset G_r$ be the standard parabolic subgroup
whose Levi factor is the centralizer $G_\lambda$ of the character $\lambda$ in $G_r$.

\smallskip

\subsection{The affine Grassmannians} \label{sec:AffGr}

We'll call \emph{affine Grassmannian} $\Gr$ the reduced algebraic $F$-space underlying the $F$-space 
which is sheafification of the
presheaf 
$\calA lg_F\to\calS ets$ taking $R$ to $\calG_r(K)(R)\,/\,\calG_r(O)(R).$
It is represented by a formally smooth reduced ind-projective $F$-scheme with a left action of the $F$-group $G_r(K)$.
The Cartan decomposition 
$$G_r(K)=\bigsqcup_{\lambda\in \Lambda_r}G_r(O)\,t^\lambda\, G_r(O)$$
yields a partition into $G_r(O)$-orbits 
$\Gr=\bigsqcup_{\lambda\in \Lambda_r}\Gr_\lambda$.
We have $\Gr_\lambda\subset\overline{\Gr_\mu}$ if and only if $\lambda\leqslant\mu$.
Define an irreducible projective $G_r(O)$-equivariant $F$-scheme of finite type by setting
$\Gr_\beta=\overline{\Gr_{n\omega_1}}.$
The set of $F$-points $\Gr_\beta(F)$
is the set of all $O$-sublattices of codimension $n$ in the \emph{standard lattice}
$L_0=O^{\oplus r}$.

\smallskip

We'll call \emph{thick affine Grassmannian} $\bfGr$ the reduced algebraic $F$-space underlying the $F$-space 
which is the sheafification of the presheaf  
$\calA lg_F\to\calS ets$ taking $R$ to $\calG_r(K)(R)\,/\,\calG_r( O^-)(R).$
It is represented by 
a reduced separated $F$-scheme of infinite type with a left action of the $F$-group $G_r(K)$,
see \cite{K89}.
The decomposition
$$G_r(K)=\bigsqcup_{\lambda\in \Lambda_r}G_r(O)\,t^\lambda\, G_r(O^-)$$
yields a partition of $\bfGr$ into $G_r( O)$-orbits 
$\bfGr=\bigsqcup_{\lambda\in \Lambda_r}\bfGr_\lambda$
such that
$\bfGr_\lambda\subset\overline{ \bfGr_\mu}$ if and only if $\lambda\geqslant\mu$.
The open subset $\bfGr_\beta=\bigsqcup_{\lambda\in \Lambda_\beta}\bfGr_\lambda$
is an irreducible $G_r(O)$-equivariant quasi-compact $F$-scheme of infinite type.

\smallskip

Fix a principal congruence subgroup $K_\beta$ of $G_r( O)$ which is contained
into $\bigcap_{\lambda\in\Lambda_\beta} t^\lambda G_r( O)t^{-\lambda}$.
The group $K_\beta$ acts trivially on $\Gr_\beta$
and freely on $\bfGr_\beta$.
We define
$$\Gr^-_\beta=\Gr_\beta
\quad,\quad
\Gr^+_\beta=\bfGr_\beta/K_\beta.$$ 
The $F$-scheme $\Gr^+_\beta$ is smooth of finite type,
because the $K_\beta$-action on $\bfGr_\beta$ is locally free.  
It is separated, see, e.g., \cite[\S A.6]{VV09}, \cite[lem.~6.3]{K17}.
Since $K_\beta$ is a normal subgroup of $G_r( O)$, we may consider the action of the affine algebraic 
$F$-group $H_\beta=G_r( O)/K_\beta$ on  $\Gr^\pm_\beta$.
The $H_\beta$-orbits give a partition
$$\Gr^\pm_\beta=\bigsqcup_{\lambda\in \Lambda_\beta}\Gr^\pm_\lambda.$$
Let $S_\beta^\pm=\{\Gr^\pm_\lambda\,;\,\lambda\in \Lambda_\beta\}$
be the corresponding stratification of $\Gr^\pm_\beta$.

\smallskip

Let $I\subset G_r( O)$ be an Iwahori subgroup containing $K_\beta$.
Consider the subgroup $I_\beta=I/K_\beta$ of $H_\beta.$
Let 
$T_\beta^\pm=\{\Gr^\pm_w\,;\,w\in M_\beta\}$
be the stratification of $\Gr^\pm_\beta$ by the $I_\beta$-orbits.
Here $M_\beta$ is a subset of $M$.

\smallskip

For each subset $\Gamma\subset\Lambda_\beta$, we consider the quotient stack
$\calG r^\pm_\Gamma=[\Gr^\pm_\Gamma/H_\beta]$, where
$$\Gr^\pm_\Gamma=\bigsqcup_{\lambda\in \Gamma}\Gr^\pm_\lambda.$$
Let $i^\pm_\Gamma$ be the locally closed inclusion $\Gr^\pm_\Gamma\subset \Gr^\pm_\beta$. 
We abbreviate $\calG r^\pm_\beta=\calG r^\pm_{\Lambda_\beta}$
and $\calG r^\pm_\lambda=\calG r^\pm_{\{\lambda\}}$.
The complex $IC(\lambda)^\pm$ defined by
\begin{align}\label{IC}
IC(\lambda)^\pm&=(i_\lambda)_{!*}\k_\lambda[d^\pm_\lambda]
\quad,\quad
d^\pm_\lambda=\dim \Gr^\pm_\lambda
\end{align}
can be viewed either as an object of $\D^\b(\calG r^+_\beta)$
or as an object of
$\D^\b_\mix(\calG r^+_\beta)$.  
In the latter case we write $IC(\lambda)^+_\mix$.
Applying the forgetful functor $\For$, we may also view them as objects of the categories
$\D^\b(\Gr^+_\beta,S)$ or $\D^\b_\mix(\Gr^+_\beta,S)$.
Recall from Definition \ref{def:even} the notion of good and even stratifications.

\smallskip

\begin{proposition}\label{prop:even1}
\hfill
\begin{itemize}[leftmargin=8mm]
\item[$\mathrm{(a)}$] $T_\beta^\pm$ is a good stratification.
\item[$\mathrm{(b)}$] $S_\beta^\pm$ is an even stratification.
\item[$\mathrm{(c)}$] $H^\bullet(\Gr^\pm_\lambda\,,\,\k)=H^\bullet(G_r/P_\lambda\,,\,\k)$ for each 
$\lambda\in\Lambda_r$.
\end{itemize}
\end{proposition}

\begin{proof}
The strata of $S_\beta^\pm$ are connected, because they are $H_\beta$-orbits.
The simply connectedness and the statement in (c) follow from the fact that $\Gr^\pm_\lambda$ is an affine bundle over 
$G_r/P_\lambda$, see, e.g., \cite[\S 5]{La}. So (a) implies (b).
The strata of $T_\beta^\pm$ are affine.
The conditions \cite[\S 4.1(a)-(d)]{Y09} are proved in \cite[\S 5.2]{Y09}.
So to prove (a), it is enough to prove that the strata of $T_\beta^\pm$ satisfy the conditions (2), (3) in Definition \ref{def:even}.

\smallskip

For $T^-_\beta$ this is well-known, the proof consists of checking
the conditions in  \cite[lem.~4.4.1]{BGS96}.

For $T^+_\beta$, the condition (2) follows from \cite{KT02} and
the odd vanishing of the Kazhdan-Lusztig polynomials.
Let us concentrate on the condition (3) for $T^+_\beta$.
A well-known argument of Kazhdan-Lusztig implies that $H^a((i_v^+)^*IC(w)^+_\bbd)$ 
is pure of weight $a$. See \cite[lem.~3.5]{G91} or \cite[lem.~3.1.3]{BY13}
for an easier proof which generalizes to our setting.
The condition (3) requires in addition that the mixed vector space $H^a((i_v^+)^*IC(w)^+_\bbd)$ is semisimple.
By \cite[lem.~4.4.1]{BGS96}, it is enough to check that the mixed vector space
$\underline H^a(\Gr^+_{\beta,0}\,,\,IC(w)^+_\bbd)$ is a sum of copies of $\k(-a/2)$  
for each $w\in M_\beta$ and $a\in\bbN$. 

\smallskip

To prove this, we consider Kashiwara's thick affine flag manifold $\bfF\bfl$. 
Let $\G$, $\B^+$ and $\B^-$ be as in the proof of Lemma \ref{lem:CD} below.
The scheme $\bfF\bfl$ is separated
of infinite type represented  by the quotient $\G\, /\,\B^-$ with a decomposition into $\B^+$-orbits 
$\bfF\bfl=\bigsqcup_v\bfF\bfl_v$ labelled by the affine Weyl group $\widehat W$ of $\G$.
The $\B^+$-orbit $\bfF\bfl_v$ is an affine space of codimension equal to the length of $v$.
Let $W_\beta\subset\widehat W$ be an order ideal. Let $K_\beta$ be a congruence subgroup
which acts freely on $\bfF\bfl_\beta=\bigsqcup_{v\in W_\beta}\bfF\bfl_v$.
We have a smooth (separated)
scheme $\Fl_\beta=\bfF\bfl_\beta\,/\,K_\beta$ with an affine stratification
$\Fl_\beta=\bigsqcup_{v\in W_\beta}\Fl_v$ such that $\Fl_v=\bfF\bfl_v\,/\,K_\beta$.
Here $W_\beta$ is the inverse image of $M_\beta$ by the obvious projection $\widehat W\to M$.
The bundle $\bfF\bfl\to\bfGr$ yields a bundle
$q:\Fl_{\beta,0}\to\Gr^+_{\beta,0}$ over $F_0$ with smooth, projective fibers with affine stratifications. See Section \ref{ss:mixcmplx} for the convention.
For each element $v\in W_\beta$ let $IC(v)_\bbd\in\D^\b_\bbd(\Fl_{\beta,0})$ 
be the intermediate extension of the mixed complex 
$\k_{\Fl_v}\langle d_v\rangle$, where $d_v=\dim \Fl_v$.
The map $q$ restricts to an isomorphism $\Fl_{w,0}\to\Gr^+_{w,0}$.
Thus, we have
$q_*IC(w)=IC(w)^+\oplus\calF$ for some complex $\calF\in\D^\b(\Gr^+_\beta)$.
The following is well-known.

\begin{claim}\label{claim:fact} Let $\calE\in\D^\b_\bbd(\calZ_0)$ be a pure complex and $\calF\in\D^\b_\bbd(\calZ_0)$
be a subquotient of some perverse cohomology sheaf ${}^p H^b(\calE)$. 
For each $a\in\bbZ$, the mixed vector space $\underline H^a(\calZ_0,\calF)$
is a subquotient of $\underline H^a(\calZ_0,\calE)$.
\qed
\end{claim}

Thus the mixed vector space $\underline H^a(\Gr^+_{\beta,0}\,,\,IC(w)^+_\bbd)$
is a subquotient of $\underline H^a(\Fl_{\beta,0}\,,\,IC(w)_\bbd)$  for each $a\in\bbN$.
Hence, it is enough to check that $\underline H^a(\Fl_{\beta,0}\,,\,IC(v)_\bbd)$ is a sum of copies of $\k(-a/2)$
for each $v\in W_\beta$. 
If  $v$ is maximal in the poset $W_\beta$, then
the stratum $\Fl_v$ is closed in $\Fl_\beta$, hence we have 
$IC(v)_\bbd=\k_{\Fl_v}\langle d_v\rangle$, thus the claim is obvious. 
For an arbitrary element $v\in W_\beta$, we argue by decreasing induction on the length of $v$.

\smallskip

For each simple reflection $s_i\in\widehat W$, we consider the parabolic subgroup $\P_i^-=\B^-s_i\B^-\cup\B^-$.
Let  $\bfF\bfl^i$ be the partial thick affine flag manifold $\bfF\bfl^i=\G\, /\,\P_i^-$.
We define as above a smooth $F_0$-scheme $\Fl^i_{\beta,0}$ with 
a $\bbP^1_{F_0}$-bundle $p:\Fl_{\beta,0}\to\Fl_{\beta,0}^i$ such that  
$$\Fl_{v,0}\cup\Fl_{vs_i,0}=p^{-1}p(\Fl_{v,0})\quad,\quad p(\Fl_{vs_i,0})=p(\Fl_{v,0})
\quad,\quad\forall v\in W_\beta.$$
Now, assume that $vs_i<v$. The map $p$ restricts to an isomorphism $\Fl_v\to p(\Fl_v)$.
Consider the endofunctor
$\SS_i$ of $\D^\b_\bbd(\Fl_{\beta,0})$
given by
$$\SS_i(\calE)=p^*p_*\calE\langle 1\rangle\quad,\quad\forall\calE\in\D^\b_\bbd(\Fl_{\beta,0}).$$
There is a complex $\calF$ supported on the closure of the stratum $\Fl_v$ in $\Fl_\beta$ such that
$$\omega\,\SS_i(IC(v)_\bbd)=IC(vs_i)\oplus\calF.$$
Hence Claim \ref{claim:fact} implies that $\underline H^a(\Fl_{\beta,0}\,,\,IC(vs_i)_\bbd)$
is a  subquotient of  the mixed vector space
$\underline H^a(\Fl_{\beta,0}\,,\,\SS_i(IC(v)_\bbd))$.
Now, consider the following Cartesian diagram of $F_0$-schemes
$$
\xymatrix{\Z^i_{\beta,0}\ar[r]^-m\ar[d]_-{\pi_1}&\Fl_{\beta,0}\ar[d]^-{p}\\
\Fl_{\beta,0}\ar[r]^-{p}&\Fl^i_{\beta,0}.
}$$
All maps are $\bbP^1_{F_0}$-bundles. By proper base change, we have
$$\SS_i(IC(v)_\bbd)=m_*(\pi_1)^*(IC(v)_\bbd)\langle 1\rangle.$$
We deduce that
$$\underline H^a(\Fl_{\beta,0}\,,\,\SS_i(IC(v)_\bbd))=
\underline H^a(\Z^i_{\beta,0}\,,\,(\pi_1)^*(IC(v)_\bbd)\langle 1\rangle).$$
The $F_0$-scheme $\Z^i_{\beta,0}$ has an affine stratification given by the product of Bruhat cells on 
$\G/\B^-$ and $\P_i^-/\B^-$.
The morphism $\pi_1$ is stratified.
The mixed complex 
$(\pi_1)^*(IC(v)_\bbd)\langle 1\rangle$ in $\D^\b_\bbd(\Z^i_{\beta,0})$ is the intermediate extension
of the constant sheaf on the open dense stratum in $(\pi_1)^{-1}(\Fl_{v,0}).$
It satisfies the conditions (2) and (3) in Definition \ref{def:even}, because the mixed complex 
$IC(v)_\bbd$ in $\D^\b_\bbd(\Fl_{\beta,0})$ satisfy them by the induction hypothesis and
both conditions are preserved by a pullback by a smooth stratified morphism.
\end{proof}

\smallskip

\subsection{The thick affine Grassmannian and the quiver $\bfQ$}\label{sec:tilting}

Let $\calC oh_\beta$ be the $F$-stack classifying coherent sheaves on $\bbP^1_F$ of rank $r$ and degree $n$.
Let $\calB un_\beta$ be the open substack parametrizing locally free coherent sheaves.
We abbreviate
$\calC oh=\bigsqcup_{\beta}\calC oh_\beta$ and $\calB un=\bigsqcup_{\beta}\calB un_\beta$.
Both stacks are smooth locally quotient stacks. 
Consider the vector bundle $\calO(\lambda)$ on $\bbP^1_F$ given by
$\calO(\lambda)=\calO(\lambda_1)\oplus\cdots\oplus\calO(\lambda_r)$ with
$\lambda\in\Lambda_\beta.$
Let $\calB un_\beta^+$ be the full substack of $\calB un_\beta$ classifying all vector bundles isomorphic to
$\calO(\lambda)$ for some $\lambda\in\Lambda_\beta$.

\smallskip

\begin{proposition}\label{prop:Y=X} 
\hfill
\begin{itemize}[leftmargin=8mm]
\item[$\mathrm{(a)}$] 
There are $F$-stack isomorphisms
$\calG r^+_\beta\simeq \calB un^+_\beta\simeq \calY_\beta$ taking
$\calG r^+_{\lambda_\pi}$ to the isomorphism classes of $\calO(\lambda_\pi)$ and $\calP_\pi$ for each 
$\pi\in\Gamma\!_\beta$.
\item[$\mathrm{(b)}$] 
$\lambda_\sigma\geqslant\lambda_\pi\iff Y_\sigma\subset\bar Y_\pi\Longrightarrow\sigma\geqslant\pi$
for each $\pi,\sigma\in \Gamma\!_\beta$.
\end{itemize}
\end{proposition}

\begin{proof}
Let $\D^\b(\calC oh)$ denote the derived category of the Abelian category $\calC oh$.
The vector bundle $\calT=\calO\oplus\calO(1)$ over $\bbP^1_F$ is a tilting generator of 
$\D^\b(\calC oh)$, i.e., it is a generator of $\D^\b(\calC oh)$ as a triangulated
category such that $\End^{>0}_{\D^\b(\calC oh)}(\calT)=0$.
We have an $F$-algebra isomorphism $\End_{\D^\b(\calC oh)}(\calT)^\op=F\bfQ$ such that the elements
$x,y\in F\bfQ$ span the $F$-subspace $\Hom_{\D^\b(\calC oh)}(\calO,\calO(1))$. Thus the functor
$\RHom_{\D^\b(\calC oh)}(\calT,\bullet)$ yields an equivalence of triangulated categories
$$\D^\b(\calC oh)\to\D^\b(F\bfQ\-\mod)\quad,\quad
\calO(k)\mapsto\calP_k\quad,\quad \forall k\in\bbN.$$
This equivalence restricts to an equivalence of additive categories $\calB un_\beta^+\to\calY_\beta$
which maps $\calO(\lambda_\pi)$ to $\calP_\pi$ for each $\pi\in\Gamma\!_\beta$.
Now, part (a) follows from the following isomorphism of $F$-stacks
$\calG r^+\simeq \bigsqcup_n\calB un_{r\alpha_0+n\delta},$
see, e.g., \cite[thm.~2.3.7]{Z16} for details. 
Part (b) follows from the proof of Proposition \ref{prop:inftyX}(b).
\end{proof}

\smallskip

Let $\Gamma\subset\Gamma\!_\beta$ be an order ideal.
By Proposition \ref{prop:equivalence1}
we have adjoint functors
$(i_\Gamma)^*$, $(i_\Gamma)_*$ between the categories $\D^\b_\mix(\calY_\Gamma)$, $\D^\b_\mix(\calY_\beta)$.

\smallskip

\begin{proposition}\label{prop:equivalence8}
For each order ideal $\Gamma\subset\Gamma\!_\beta$
\hfill
\begin{itemize}[leftmargin=8mm]
\item[$\mathrm{(a)}$] 
there is an equivalence of graded triangulated categories 
$\A'_\Gamma:\D^\b_\mix(\calY_\Gamma)\to \D^\b_\mix(\calG r^+_\Gamma)$
such that
$(j_\Gamma)^*IC(\pi)_\mu\mapsto IC(\lambda_\pi)^+_\mix$ for all $\pi\in\Gamma,$
\item[$\mathrm{(b)}$]
there is a right adjoint triangulated functor $(i_\Gamma^+)_*$ to $(i_\Gamma^+)^*$ yielding a commutative
diagram  
$$\xymatrix{
\D^\b_\mix(\calY_\beta)\ar[d]_-{\A'_\beta}
\ar@/^/[r]^-{(i_\Gamma)^*}&\ar@/^/[l]^-{(i_\Gamma)_*}
\D^\b_\mix(\calY_\Gamma)\ar[d]_-{\A'_\Gamma}\\
\D^\b_\mix(\calG r_\beta^+)
\ar@/^/[r]^-{(i_\Gamma^+)^*}&\ar@/^/[l]^-{(i_\Gamma^+)_*}
\D^\b_\mix(\calG r_\Gamma^+).&
}$$
\end{itemize}
\end{proposition}

\begin{proof}
Proposition \ref{prop:Y=X} yields an equivalence of graded additive categories
$\Par(\calY_\Gamma)\simeq \Par(\calG r_\Gamma^+).$
Taking the homotopy categories, we get an equivalence of graded triangulated categories
$\D^\b_\mix(\calY_\Gamma)\simeq\D^\b_\mix(\calG r^+_\Gamma).$
The other claims in the proposition are straightforward.
For instance, since the map $i^+_\Gamma$ is an open embedding, the functor 
$(i_\Gamma^+)^*:\D^\b_\mix(\calG r_\beta^+)\to\D^\b_\mix(\calG r_\Gamma^+)$ is well-defined.
We define $(i_\Gamma^+)_*$ to be the unique functor such that the diagram above commutes.
\end{proof}

\smallskip

\subsection{The thick affine Grassmannian and the category $\calD_\beta$}
Let $\Gamma\subset\Gamma\!_\beta$ be any order ideal.
We define the triangulated functor
\begin{align}\label{equivalence0}
\A_\Gamma=\A'_\Gamma\,\A''_\Gamma:\D^\b(\calD_\Gamma)\to \D^\b_\mix(\calG r_\Gamma^+).
\end{align}
Conjugating the functor $(i_\Gamma^+)_*$ with the Verdier duality yields the functor
$(i_\Gamma^+)_!$ which is left adjoint to $(i_\Gamma^+)^*$. 
We define the following complexes in $\D^\b_\mix(\calG r^+_\Gamma)$
\begin{align}\label{EDN+}
\Delta(\lambda)^+_\mix=(i_{\leqslant\lambda}^+)_!(i_{\leqslant\lambda}^+)^*IC(\lambda)^+_\mu\quad,\quad
\nabla(\lambda)^+_\mix=(i_{\leqslant\lambda}^+)_*(i_{\leqslant\lambda}^+)^*IC(\lambda)^+_\mu.
\end{align}

\smallskip

\begin{proposition}\label{prop:equivalence11}
Let $\Gamma\subset\Gamma\!_\beta$ be any order ideal.
\hfill
\begin{itemize}[leftmargin=8mm]
\item[$\mathrm{(a)}$] 
$\A_\Gamma$ is an
equivalence of graded triangulated categories such that
$\A_\Gamma(h_\Gamma)^*Q(\pi)=IC(\lambda_\pi)^+_\mix$
and
$\A_\Gamma\Delta(\pi)=\nabla(\lambda_\pi)^+_\mix$
for all $\pi\in\Gamma.$
\item[$\mathrm{(b)}$] 
The following diagram of functors commutes
$$\xymatrix{
\D^\b(\calD_\beta)\ar[d]_-{\A_\beta}
\ar@/^/[r]^-{L(h_\Gamma)^*}&\ar@/^/[l]^-{(h_\Gamma)_*}
\D^\b(\calD_\Gamma)\ar[d]_-{\A_\Gamma}\\
\D^\b_\mix(\calG r_\beta^+)
\ar@/^/[r]^-{(i_\Gamma^+)^*}&\ar@/^/[l]^-{(i_\Gamma^+)_*}
\D^\b_\mix(\calG r_\Gamma^+).
}$$

\end{itemize}
\end{proposition}

\begin{proof}
The proposition follows from Propositions \ref{prop:equivalence1}, \ref{prop:equivalence8}, except the last claim in (a).
We may assume $\Gamma=\Gamma\!_\beta$.
 By part (b), for each $\pi\in\Gamma\!_\beta$ we have
\begin{align*}
\nabla(\lambda_\pi)^+_\mix
=(i_{\leqslant\lambda_\pi}^+)_*(i_{\leqslant\lambda_\pi}^+)^*IC(\lambda_\pi)^+_\mix
=\A_\beta(h_{\leqslant\pi})_*L(h_{\leqslant\pi})^*Q(\pi)
=\A_\beta(\Delta(\pi)).
\end{align*}
\end{proof}

\smallskip

\subsection{The non equivariant case}\label{sec:NE}
Fix an order ideal $\Gamma\subset\Gamma_\beta$. For the Radon transform studied in the next section,
we need to consider the mixed category
$\D^\b_\mix(\Gr^+_\Gamma,S)$ rather than the equivariant mixed category $\D^\b_\mix(\calG r^+_\Gamma)$.
Let us give some properties of $\D^\b_\mix(\Gr^+_\Gamma,S)$.
We write $V(\lambda_\pi)=V(\pi)$ for each $\pi\in\Gamma$.
Set
\begin{align*}
\calM_\Gamma^\pm=\bigoplus_{\lambda\in\Gamma}V(\lambda)\otimes IC(\lambda)^\pm.
\end{align*}
We view $\calM_\Gamma^\pm$ either as an object of $\D^\b_\mix(\calG r^+_\Gamma)$,
or as an object of $\D^\b_\mix(\Gr^+_\Gamma,S)$ via the forgetful functor
$\For:\D^\b(\calG r^\pm_\Gamma)\to\D^\b(\Gr^\pm_\Gamma,S)$.
We define
\begin{align}\label{sharp}
\bfS^\pm_\Gamma =\End^\bullet_{\D^\b(\calG r^\pm_\Gamma)}(\calM_\Gamma^\pm)
\quad,\quad
\bfS^{\pm,\,\sharp}_\Gamma= \End^\bullet_{\D^\b(\Gr^\pm_\Gamma)}(\calM_\Gamma^\pm).\end{align}

\smallskip

\begin{proposition}\label{prop:isom5}
Let $\Gamma\subset\Lambda_\beta$ be any order ideal.
\hfill
\begin{itemize}[leftmargin=8mm]
\item[$\mathrm{(a)}$] 
$\Par(\calG r^\pm_\Gamma)\simeq \bfS^\pm_\Gamma\-\,\proj$ and
$\Par(\Gr^\pm_\Gamma,S)\simeq \bfS^{\pm,\,\sharp}_{\Gamma}\-\,\proj$
as graded additive categories.
\item[$\mathrm{(b)}$] 
$\D^\b_\mix(\calG r^\pm_\Gamma)=\D^\perf(\bfS^\pm_\Gamma)$ and
$\D^\b_\mix(\Gr^\pm_\Gamma,S)=\D^\perf(\bfS^{\pm,\,\sharp}_\Gamma)$ as graded triangulated categories.
\item[$\mathrm{(c)}$] 
$\bfS^\pm_\Gamma$ is free of finite rank as an $H^\bullet_{G_r}$-module, and
$\k\otimes_{H^\bullet_{G_r}}\bfS^\pm_\Gamma\simeq \bfS^{\pm,\,\sharp}_\Gamma$ as graded $\k$-algebras.
\end{itemize}
\end{proposition}

\begin{proof}
Parts (a), (b) follow from Remark \ref{rem:mixte2}(a).
To prove (c), note that \cite[prop.~2.6]{JMW} yields
\begin{align}\label{isom33}
\Hom^\bullet_{\D^\b(\Gr^\pm_\Gamma,S)}(\For\calE\,,\,\For\calF)=
\k\otimes_{H^\bullet_{G_r}}\Hom^\bullet_{\D^\b(\calG r^\pm_\Gamma)}(\calE\,,\,\calF)
\quad,\quad\forall\calE,\calF\in\Par(\calG r^\pm_\Gamma),\end{align}
because the stratification $S$ is even.
\end{proof}

\smallskip

By Proposition \ref{prop:isom5}(c), there is an obvious surjective algebra homomorphism
$\xi^\pm:\bfS^\pm_\Gamma\to\bfS^{\pm,\,\sharp}_\Gamma$.
The restriction of scalars relatively to  the morphism $\xi^\pm$
is exact and yields a functor of triangulated categories
$(\xi^\pm)_*: \D^\b_\mix(\Gr^\pm_\Gamma,S)\to\D^\b_\mix(\calG r^\pm_\Gamma).$
The left adjoint functor
$L(\xi^\pm)^*:\D^\b_\mix(\calG r^\pm_\Gamma)\to\D^\b_\mix(\Gr^\pm_\Gamma,S)$ 
is the functor $\For$. It is given by derived induction relatively to the morphism $\xi^\pm$.

\smallskip

For any $\alpha\in Q_+$, 
an element of the representation variety $X_\alpha$ is a pair of matrices $(x,y)$ 
corresponding to the actions of the generators
$x,y$ of the path algebra $F\bfQ$. 
In particular, setting $\alpha=\delta$ and $x=1$, $y=0$, we have the element $(1,0)\in X_\delta$.
Let $\calE_\Gamma$ be the quotient stack over 
$\calY_\Gamma$ given by
\begin{align*}
\calE_\Gamma=\Big[\Big\{(x,y,\varphi)\,;\,(x,y)\in 
Y_\Gamma\,,\,\varphi\in\Hom_{F\bfQ}((x,y),(1,0))\Big\}\,\Big/ G_\beta\Big].
\end{align*}
The morphism $\calE_\Gamma\to\calY_\Gamma$ such that $(x,y,\varphi)\mapsto (x,y)$ is a vector bundle of rank $r$, because
$$\Hom_{F\bfQ}((x,y)\,,\,(1,0))= \coker(y)^*$$
and $\coker(y)\simeq F^r$ since $y$ is injective.
Hence $\calE_\Gamma$ yields a stack homomorphism
$\calY_\Gamma\to [\bullet/G_r]$.
The pullback by this map is a graded $\k$-algebra homomorphism
\begin{align}\label{mor2}H^\bullet_{G_r}\to H^\bullet(\calY_\Gamma\,,\,\k).\end{align}
By \eqref{inftyVV} there is a graded algebra homomorphism
\begin{align}\label{Hy}H^\bullet(\calY_\Gamma\,,\,\k)\to\bfS_\Gamma\end{align}
such that the elements in the image graded-commute with the elements of $\bfS_\Gamma$.
Composing \eqref{Hy} with \eqref{mor2} yields a graded algebra homomorphism
\begin{align}\label{CENT1}H^\bullet_{G_r}\to 
Z(\bfS_\Gamma).\end{align}
We define
$\bfS_\Gamma^\sharp=\k\otimes_{H^\bullet_{G_r}}\!\!\bfS_\Gamma$ and
$\calD_\Gamma^{\sharp}=\bfS_\Gamma^\sharp\-\mod.$
Let $\xi:\bfS_\Gamma\to\bfS^\sharp_\Gamma$ be the specialization homomorphism.

\smallskip

By Lemma \ref{lem:AA}, we have
$\calD_\Gamma=\bfS_\Gamma\-\mod.$
The restriction of scalars relatively to the map $\xi$
yields a full embedding of graded Abelian categories 
\begin{align}\label{xi1}\xi_*:\calD_\Gamma^{\sharp}\to\calD_\Gamma.\end{align}
It is exact and gives a functor of triangulated categories
$\xi_*: \D^\perf(\calD^{\,\sharp}_\Gamma)\to\D^\b(\calD_\Gamma),$
which is not fully faithful in general.
Let $\xi^*$, $L\xi^*$ be the left adjoint functors, which are
given by induction and derived induction relatively to $\xi$.

\smallskip

\begin{lemma} \label{lem:sharp} Let $\Gamma\subset\Lambda_\beta$ be any order ideal.
\hfill
\begin{itemize}[leftmargin=8mm]
\item[$\mathrm{(a)}$] 
$\bfS_\Gamma^\sharp=\bfS_\Gamma^{+,\,\sharp}$
as graded $\k$-algebras.
\item[$\mathrm{(b)}$] $\Par(\Gr^+_\Gamma,S)\simeq\bfS_\Gamma^\sharp\-\,\proj$
as graded additive categories.
\end{itemize}
\end{lemma}

\begin{proof}
The module $Q_\Gamma=(h_\Gamma)^*Q_\beta$, is a projective graded-generator of $\calD_\Gamma$.
By Proposition \ref{prop:equivalence11} there is an equivalence of graded triangulated categories
$\A_\beta:\D^\b(\calD_\beta)\to\D^\b_\mix(\calG r^+_\beta)$
such that $\A_\beta Q_\Gamma=\calM_\Gamma^+$.
Hence, we have
\begin{equation}
\begin{aligned}\label{isom3}
\bfS_\Gamma
=\bbEnd_{\calD_\Gamma}(Q_\Gamma)
=\bbEnd_{\D^\b(\calD_\Gamma)}(Q_\Gamma)
=\bbEnd_{\D^\b_\mix(\calG r^+_\Gamma)}(\calM_\Gamma^+)
=\End^\bullet_{\D^\b(\calG r^+_\Gamma)}(\calM_\Gamma^+)
=\bfS^+_\Gamma.
\end{aligned}
\end{equation}
The vector bundle $\calE_\Gamma$ over $\calY_\Gamma$ is isomorphic to the
pull-back of the universal bundle over $\calB un_\beta^+\times\bbP^1_F$ by the embedding
$$\calY_\Gamma\subset\calY_\beta=\calB un_\beta^+\times\{0\}\subset\calB un_\beta^+\times\bbP^1_F.$$
Therefore, under the stack isomorphism
$\calY_\Gamma\simeq\calG r^+_\Gamma$ in Proposition \ref{prop:Y=X},
the $G_r$-torsor associated with $\calE_\Gamma$ is identified with the
$G_r$-torsor $[\Gr^+_\Gamma/U_\beta]$ over $\calG r^+_\Gamma$, where
$U_\beta$ is the unipotent radical of $H_\beta$ with Levi complement $G_r$.
We deduce that the obvious graded $\k$-algebra homomorphism
\begin{align}\label{CENT2}H^\bullet_{G_r}\to H^\bullet(\calG r^+_\Gamma\,,\,\k)\to
Z(\bfS_\Gamma^+),\end{align}
is identified with the central homomorphism
\eqref{CENT1} under the algebra isomorphism $\bfS_\Gamma=\bfS^+_\Gamma$.
This implies that
$\bfS_\Gamma^\sharp=\bfS_\Gamma^{+,\,\sharp}.$
Part (b) follows from (a) and Proposition \ref{prop:isom5}(a).
\end{proof}

\smallskip

For each Kostant partition $\pi\in\Gamma\!_\beta$ we define the module
$\Delta(\pi)^\sharp\in\calD_\beta^{\,\sharp}$ by setting
\begin{align}\label{Dsharp}\Delta(\pi)^\sharp=\xi^*\Delta(\pi).\end{align}

\smallskip

\begin{proposition}\label{prop:equivalence3}
Let $\Gamma\subset\Lambda_\beta$ be any order ideal.
\hfill
\begin{itemize}[leftmargin=8mm]
\item[$\mathrm{(a)}$] 
We have an equivalence of graded triangulated categories
$\A_\Gamma^\sharp:\D^\perf(\calD_\Gamma^\sharp)\to\D^\b_\mix(\Gr^+_\Gamma,S)$
such that the following diagram of functors commutes
$$\xymatrix{
\D^\perf(\calD^{\,\sharp}_\Gamma)\ar[d]_-{\A^\sharp_\Gamma}
\ar@/^/[r]^{\xi_*}&\ar@/^/[l]^-{L\xi^*}\ar[d]^-{\A_\Gamma}\D^\b(\calD_\Gamma)\\
\D^\b_\mix(\Gr^+_\Gamma,S)\ar@/^/[r]^{\xi_*^+}&\ar@/^/[l]^-{\For}\D^\b_\mix(\calG r^+_\Gamma).}$$

\item[$\mathrm{(b)}$] 
For each $\pi\in\Gamma$, we have
$\A_\Gamma^\sharp L(h_\Gamma)^*Q(\pi)^\sharp= (i^+_\Gamma)^*IC(\lambda_\pi)^+_\mix$
and
$\A_\Gamma^\sharp\Delta(\pi)^\sharp= \nabla(\lambda_\pi)^+_\mix.$
\end{itemize}
\end{proposition}

\begin{proof}
By Lemma \ref{lem:sharp}, we have
an equivalence of graded triangulated categories
$$\A^\sharp_\Gamma:\D^\perf(\calD_\Gamma^\sharp)=\K^\b(\bfS_\Gamma^\sharp\-\,\proj)
\to\K^\b(\Par(\Gr^+_\Gamma,S))=\D^\b_\mix(\Gr^+_\Gamma,S)$$
such that 
$\A^\sharp_\Gamma L(h_\Gamma)^*Q(\pi)^\sharp=(i_\Gamma^+)^*IC(\lambda_\pi)^+$
for all $\pi\in\Gamma.$
The isomorphism $\A_\Gamma^\sharp\Delta(\pi)^\sharp= \nabla(\lambda_\pi)^+_\mix$
is proved as in Proposition \ref{prop:equivalence11}.
\end{proof}

For a future use, we now describe explicitly the map 
$\xi:\bfS_\Gamma\to\bfS_\Gamma^\sharp$.
To simplify we set $\Gamma=\Gamma\!_\beta$.
Let $c_t(\calE_\beta)$ be the Chern polynomial of the rank $r$ vector bundle $\calE_\beta\to\calY_\beta$. 
Let $J_\beta\subset  Z(\bfS_\beta)$
be the ideal generated by the image by \eqref{Hy} of the non-constant coefficients of  $c_t(\calE_\beta)$.

\smallskip

Let $\Lambda$ be the ring of symmetric functions in two sets of variables
$x_1,x_2,\dots$ and $y_1,y_2,\dots$ and coefficients in $\k$.
Let $e_i(x), e_i(y)$ be
the $i$th elementary symmetric functions and
$E_t(x),$ $E_t(y)$ be the corresponding generating series. 
Since $X_\beta$ is an affine space, we have 
\begin{align}\label{ID0}H^\bullet(\calX_\beta\,,\,\k)=H^\bullet_{G_\beta}=\Lambda\,/\,(e_i(x),e_j(y)\,;\,i>m\,,\,j>n).
\end{align}
Let $I_\beta\subset H^\bullet_{G_\beta}$ be the ideal generated by the non-constant coefficients of
the formal series 
\begin{align}\label{Eb}E_\beta(t)=E_{-t}(x)\,/\,E_{-t}(y).\end{align}
The restriction $(j_\beta)^*:H^\bullet(\calX_\beta,\k)\to H^\bullet(\calY_\beta,\k)$ 
maps the coefficient of $t^i$ in $E_\beta(t)$ to 0 for each $i>r$, 
due to the exact sequence \eqref{EX1}.
Let $I'_\beta\subset  Z(\bfS_\beta)$ be the ideal generated by the image by $\eqref{Hy}$ of $(j_\beta)^*I_\beta$.

\begin{proposition}\label{prop:sharp}
We have $\bfS_\beta^\sharp=\bfS_\beta/J_\beta\bfS_\beta$ and $J_\beta=I'_\beta$.
\end{proposition}

\begin{proof}
The first claim is the definition of $\bfS_\beta^\sharp$. We now concentrate on the second one.
The tautological representations of the group $G_\beta$ in $F^m$ and $F^n$ yield two vector bundles
$\calU_0$, $\calU_1$ over $\calX_\beta$ of rank $m$, $n$ respectively, with an exact sequence of vector bundles
over $\calY_\beta$
\begin{align}\label{EX1}
\xymatrix{0\ar[r]&(j_\beta)^*\calU_1\ar[r]&(j_\beta)^*\calU_0\ar[r]&(\calE_\beta)^*\ar[r]& 0.}\end{align}
Since the $\k$-algebra homomorphism
$(j_\beta)^*:\bfR_\beta\to\bfS_\beta$ in \eqref{ONTO} is onto, it gives an algebra homomorphism
$Z(\bfR_\beta)\to Z(\bfS_\beta)$, which fits in the following commutative diagram
\begin{align}\label{MAP2}
\begin{split}
\xymatrix{
\ar[r]H^\bullet_{G_r}\ar@/^2pc/[rr]^{\eqref{CENT1}}&H^\bullet(\calY_\beta\,,\,\k)\ar[r]& Z(\bfS_\beta)\\
H^\bullet_{G_\beta}\ar@{=}[r]&H^\bullet(\calX_\beta\,,\,\k)\ar@{=}[r]\ar[u]^-{(j_\beta)^*}&
Z(\bfR_\beta)\ar[u]_-{(j_\beta)^*}
}
\end{split}
\end{align}
Let $c_t(\calU_0)$ and $c_t(\calU_1)$ be the Chern polynomials.
By \eqref{EX1} we have 
$$c_t(\calE_\beta)=(j_\beta)^*c_{-t}(\calU_0)\,/\,(j_\beta)^*c_{-t}(\calU_1).$$ 
The identification \eqref{ID0} takes the formal series $c_{-t}(\calU_0)\,/\,c_{-t}(\calU_1)$
to  $E_\beta(t),$ hence $J_\beta=I'_\beta$.
\end{proof}

\smallskip

\begin{remark}\label{rem:fl} 
The $\k$-algebra $\bfS_\Gamma^\sharp$ is finite dimensional by Lemma \ref{lem:sharp},
because $\bfS_\Gamma^{+,\sharp}$ is finite dimensional by \eqref{sharp}.
Hence the category $\calD_\Gamma^\sharp$ is Artinian.
\end{remark}

\smallskip

\begin{example}
Assume that $\beta=\beta_n$. Hence $G_\beta=G_n\times G_{n+1}$ and
the $G_\beta$-variety $Y_\beta$ contains a single orbit with stabilizer $D_\beta=\bbG_{m,F}$.
The constant sheaf is a graded-generator of $\Par(\calY_\beta)$, 
hence $\Par(\calG r_\beta^+)=\Par(\calY_\beta)=\bfS_\beta$-$\proj$ with 
$\bfS_\beta=H^\bullet(\calY_\beta,\k)=H^\bullet_{D_\beta}=\k[z]$.
Further, we have
$\Par(\Gr^+_\beta,S)=\bfS_\beta^\sharp\-\,\proj$ with $\bfS^\sharp_\beta=\k$.
\end{example}

\smallskip

\subsection{The Radon transform}\label{sec:Radon}
For each subset $\Gamma$ of $\Lambda_\beta$ or $M_\beta$ we have
$$\Gr^\pm_\Gamma=\bigsqcup_{\lambda\in \Lambda}\Gr^\pm_\lambda
\quad,\quad
\Gr^\pm_\Gamma=\bigsqcup_{w\in \Gamma}\Gr^\pm_w
\quad,\quad d^\pm_\lambda=\dim \Gr^\pm_\lambda
\quad,\quad d^\pm_w=\dim \Gr^\pm_w.$$
We equip the $F$-scheme $\Gr^\pm_\Gamma$ with the stratifications given by
$S_\Gamma^\pm=\Gr^\pm_\Gamma\cap S_\beta^\pm$ and
$T_\Gamma^\pm=\Gr^\pm_\Gamma\cap T_\beta^\pm.$
We'll abbreviate
$S=S_\Gamma^\pm$ and $T=T_\Gamma^\pm$.
The $F$-scheme $\Gr^\pm_\Gamma$ has an obvious $F_0$-structure $\Gr^\pm_{\Gamma,0}$.
The stratifications $S^\pm$ and  $T^\pm$ are defined over $F_0$.

\smallskip

First, let us recall the Radon transform, following \cite{Y09}.
Let $U_0$ be the diagonal orbit of $G_r(F_0((t)))$ acting on the origin of 
$\Gr_0\times\bfGr_0$.
The diagonal action of the $F_0$-group $K_{\beta,0}$ on  
$U_0\cap (\Gr_{\Gamma,0}\times \bfGr_{\Gamma,0})$ is free. 
Let $U_{\Gamma,0}$ be the quotient by $K_{\beta,0}$.
It is open and dense in $\Gr^-_{\Gamma,0}\times \Gr^+_{\Gamma,0}$ and
may view as a correspondence of $F_0$-schemes \begin{align*}
\xymatrix{\Gr^-_{\Gamma,0}&\ar[l]_-{f^-_\Gamma}U_{\Gamma,0}\ar[r]^-{f^+_\Gamma}&\Gr^+_{\Gamma,0}}.
\end{align*}
We  write $U_\Gamma=U_{\Gamma,0}\otimes_{F_0}F$.
From now on we'll always assume that 
$\Gamma$ is an interval in $\Lambda_\beta$ or in $M_\beta$.
Then, the scheme  $\Gr_\Gamma^\pm$ is irreducible, and we may set $d^\pm_\Gamma=\dim\Gr^\pm_\Gamma$.

\smallskip

The Radon transform denotes both a pair of adjoint functors
$(R_\Gamma^-\,,\,R_\Gamma^+)$
between the categories $\D^\b(\Gr^\pm_{\Gamma},T)$
and  a pair of adjoint functors
$(R_{\Gamma,\bbd}^-\,,\,R_{\Gamma,\bbd}^+)$ between the categories of mixed complexes
$\D^\b_\bbd(\Gr^\pm_{\Gamma,0},T)$. Both pairs of functors are given by
$$\big(\,(f^+_\Gamma)_!(f^-_\Gamma)^*\,,\,
(f^-_\Gamma)_*(f^+_\Gamma)^!\,\big).$$
Since the correspondence $U_\Gamma$ is $I_\beta$-equivariant,
the functors $R_{\Gamma,\bbd}^\pm$ preserve the full subcategories
$\D^\b_{\lozenge,\bbd}(\Gr^\pm_{\Gamma,0},T)$ of $\D^\b_\bbd(\Gr^\pm_{\Gamma,0},T).$
By \cite[cor.~4.1.5\,,\,\S 5.4]{Y09}, they yield a commutative square of functors
\begin{align}\label{DIAG7}
\begin{split}
\xymatrix{
\D^\b_{\lozenge,\bbd}(\Gr^-_{\Gamma,0}\,,\,T)\ar@/^/[r]^{R^-_{\Gamma,\bbd}}\ar[d]_\omega&\ar@/^/[l]^{R^+_{\Gamma,\bbd}}\D^\b_{\lozenge,\bbd}(\Gr^+_{\Gamma,0}\,,\,T)\ar[d]^\omega\\
\D^\b(\Gr^-_\Gamma\,,\,T)
\ar@/^/[r]^{R^-_\Gamma}&\ar@/^/[l]^{R^+_\Gamma}
\D^\b(\Gr^+_{\Gamma}\,,\,T).
}
\end{split}
\end{align}
In this diagram the two pairs of horizontal arrows are adjoint equivalences.
Since the correspondence $U_\Gamma$ is $H_\beta$-equivariant, the Radon transform preserves the 
$S$-constructible complexes.
Hence, it yields a commutative square of equivalences
\begin{align}\label{DIAG1}
\begin{split}
\xymatrix{
\D^\b_{\lozenge,\bbd}(\Gr^-_{\Gamma,0}\,,\,S)\ar@/^/[r]^{R^-_{\Gamma,\bbd}}\ar[d]_\omega&\ar@/^/[l]^{R^+_{\Gamma,\bbd}}\D^\b_{\lozenge,\bbd}(\Gr^+_{\Gamma,0}\,,\,S)\ar[d]^\omega\\
\D^\b(\Gr^-_\Gamma\,,\,S)
\ar@/^/[r]^{R^-_\Gamma}&\ar@/^/[l]^{R^+_\Gamma}
\D^\b(\Gr^+_{\Gamma}\,,\,S).
}
\end{split}
\end{align}

\smallskip

Given a nested pair of subsets $\Gamma\!_1\subset\Gamma$ in $\Lambda_\beta$,
we consider the following diagram 
\begin{align}\label{diag6}
\begin{split}
\xymatrix{
\Gr^-_ {\Gamma\!_1}\ar[d]_-{i_ {\Gamma\!_1,\Gamma}^-}&\ar[l]_-{f^-_{\Gamma\!_1}}
U_ {\Gamma\!_1}\ar[r]^-{f^+_ {\Gamma\!_1}}\ar[d]&
\Gr^+_{\Gamma\!_1}\ar[d]^-{i_{\Gamma\!_1,\Gamma}^+}\\
\Gr^-_{\Gamma}&\ar[l]_-{f^-_{\Gamma}}U_{\Gamma}\ar[r]^-{f^+_{\Gamma}}&\Gr^+_{\Gamma}
.}
\end{split}
\end{align}

\smallskip

\begin{lemma}\label{lem:CD}
\hfill
\begin{itemize}[leftmargin=8mm]
\item[$\mathrm{(a)}$] 
If ${\Gamma\!_1}=\{\geqslant\!\lambda\}\cap\Gamma$, then the right square in \eqref{diag6} is Cartesian.
\item[$\mathrm{(b)}$] 
If ${\Gamma\!_1}=\{\leqslant\!\mu\}\cap\Gamma$, then the left square in \eqref{diag6} is Cartesian.
\end{itemize}
\end{lemma}

\begin{proof}
Let $x^\pm$ be the origin in $\Gr^\pm_\beta$. We abbreviate
$$\G=G_r(K)\quad,\quad\P^+=G_r( O)\quad,\quad\P^-=G_r(O^-)\quad,\quad\B^+=I_\beta.$$
Let $\B^-\subset\P^-$ be the \emph{co-Iwahori} subgroup opposite to $\B^+$.
It is the preimage under the projection $\P^-\to G_r(F)$ of the Borel subgroup opposite to the image
of $\B^+$ by the projection $\P^+\to G_r(F)$.
Fix a nested pair of subsets  $\Gamma\!_1\subset\Gamma$ in $M_\beta$.
We consider the diagram \eqref{diag6}, with
$$\Gr^\pm_{\Gamma\!_1}=\B^+\,\Gamma\!_1\, x^\pm\quad,\quad\Gr^\pm_{\Gamma}=\B^+\,\Gamma\, x^\pm.$$
The original diagram \eqref{diag6} is a particular case of this new diagram.
Note that we have
$$\P^\pm x^\pm=\B^\pm\, x^\pm\quad,\quad\P^\pm x^\mp=x^\mp\quad,\quad
\Gr^\pm_\beta=\B^+\,M_\beta\, x^\pm\quad,\quad \Gr^\pm_w=\B^+\,w\,x^\pm\quad,\quad\forall w\in M_\beta,$$
and that
\begin{itemize}[leftmargin=8mm]
\item[$\mathrm{(a)}$] 
the right square is Cartesian $\Longleftarrow$ 
$\B^+\, {\Gamma\!_1}\, \P^-x^-\cap \B^+\,\Gamma\, x^-\subset \B^+\, {\Gamma\!_1} \,x^-$,
\item[$\mathrm{(b)}$] 
the left square is Cartesian $\Longleftarrow$ 
$\B^+\, {\Gamma\!_1}\, \P^+x^+\cap \B^+\,\Gamma\, x^+\subset \B^+\, {\Gamma\!_1} \,x^+$.
\end{itemize}
Therefore, to prove the lemma it is enough to check that
\begin{itemize}[leftmargin=8mm]
\item[$\mathrm{(c)}$] 
$\B^+\, \{\geqslant\!w\}\, \B^-x^-\subset \B^+\,  \{\geqslant\!w\}\,x^-$,
\item[$\mathrm{(d)}$] 
$\B^+\, \{\leqslant \! w\}\, \B^+\, x^+\subset \B^+\, \{\leqslant \! w\}\,x^+$.
\end{itemize}
To prove (d) it is enough to prove that 
$\B^+\, w\, \B^+\,x^+\subset \B^+\,\{\leqslant\!w\} \,x^+$. We'll argue by induction on the length of $w$. 
If $w=1$ this is obvious. Let $\U_i\subset\B^+$ be the root subspace associated with the simple root $\alpha_i$.
If $w=s_iv>v$, then $\U_i\,w\subset w\,\P^-$, hence
\begin{align}\label{form3}\B^+\, s_i\,\B^+\, w\,x^+=\B^+\, s_i\,\U_i\,w\,x^+=\B^+\, v\, x^+.\end{align}
Using \eqref{form3} and the equality $\B^+\,s_i\,\B^+\,s_i\B^+=\B^+\,s_i\B^+\cup \B^+$, we deduce that
\begin{align}\label{form4}
\begin{split}
\B^+\,s_i\,\B^+\,v\,x^+&=\B^+\,s_i\,\B^+\,s_i\,\B^+\,w\,x^+
=\B^+\,w\,x^+\cup\B^+\,s_i\,\B^+\,w\,x^+
=\B^+\,w\,x^+\cup\B^+\,v\,x^+.
\end{split}
\end{align}
Using \eqref{form3}, \eqref{form4}, we deduce that
\begin{align}\label{form5}
\begin{split}
\B^+\,w\,\B^+\,x^+&=\B^+\,s_i\,v\,\B^+\,x^+,\\
&\subset\bigcup_{u\leqslant v}\B^+\,s_i\,\B^+\,u\,x^+,\\
&\subset\bigcup_{u\leqslant v}(\B^+\,u\,x^+\cup\B^+\,s_i\,u\,x^+),\\
&\subset \B^+\,\{\leqslant\!w\} \,x^+.
\end{split}
\end{align}
To prove (c), note that
\begin{itemize}[leftmargin=8mm]
\item[$\mathrm{(c)}$] 
$\Longleftrightarrow\B^+\, w\,\B^-x^-\cap \B^+\, v\,x^-=\emptyset$, $\forall v<w$,
\item[$\mathrm{(d)}$] 
$\Longleftrightarrow\B^+\, w\,\B^+x^+\cap \B^+\, v\,x^+=\emptyset$, $\forall v>w$,
\end{itemize}
Thus, the proof of (d) above implies that 
$\B^+\, w\,\B^+\cap \B^+\, v\,\B^-=\emptyset$ in $\G$ for all $v>w$,
from which we deduce that
$\B^+\, w\,\B^-\cap \B^+\, v\,\B^+=\emptyset$ in $\G$ for all $v<w$,
and this implies (c).
\end{proof}

\smallskip

Let  $\Gamma$ be an interval in $\Lambda_\beta$ or in $M_\beta$.
For each $\lambda,$ $\mu$ we abbreviate
$i^\pm_{\leqslant\mu}=i^\pm_{\{\leqslant\mu\}\cap\Gamma\,,\,\Gamma}$ and
$i^\pm_{\geqslant\lambda}=i^\pm_{\{\geqslant\lambda\}\cap\Gamma\,,\,\Gamma}.$

\begin{lemma}\label{lem:FORM2} 
In $\D^\b_{\lozenge,\bbd}(\Gr^\pm_{\Gamma,0}\,,\,S)$ or $\D^\b_{\lozenge,\bbd}(\Gr^\pm_{\Gamma,0}\,,\,T)$
the following hold
\hfill
\begin{itemize}[leftmargin=8mm]
\item[$\mathrm{(a)}$] 
$(i_{\leqslant\mu}^-)^!\circ R^+_{\Gamma,\bbd}=R^+_{\{\leqslant\mu\}\cap\Gamma,\bbd}\circ(i_{\leqslant\mu}^+)^!$
and
$(i_{\geqslant\lambda}^+)^*\circ R^-_{\Gamma,\bbd}=
R^-_{\{\geqslant\lambda\}\cap\Gamma,\bbd}\circ(i_{\geqslant\lambda}^-)^*$,
\item[$\mathrm{(b)}$] 
$(i_{\geqslant\lambda}^-)^*\circ R^+_{\Gamma,\bbd}=
R^+_{\{\geqslant\lambda\}\cap\Gamma,\bbd}\circ(i_{\geqslant\lambda}^+)^*$
and
$(i_{\leqslant\mu}^+)^!\circ R^-_{\Gamma,\bbd}=R^-_{\{\leqslant\mu\}\cap\Gamma,\bbd}\circ(i_{\leqslant\mu}^-)^!$,
\item[$\mathrm{(c)}$] 
$R^+_{\Gamma,\bbd}\circ(i_{\geqslant\lambda}^+)_*=
(i_{\geqslant\lambda}^-)_*\circ R^+_{\{\geqslant\lambda\}\cap\Gamma,\bbd}$
and
$R^-_{\Gamma,\bbd}\circ(i_{\leqslant\mu}^-)_!=(i_{\leqslant\mu}^+)_!\circ R^-_{\{\leqslant\mu\}\cap\Gamma,\bbd},$
\item[$\mathrm{(d)}$] 
$R^+_{\Gamma,\bbd}\circ(i_{\leqslant\mu}^+)_!=(i_{\leqslant\mu}^-)_!\circ R^+_{\{\leqslant\mu\}\cap\Gamma,\bbd}$
 and
$R^-_{\Gamma,\bbd}\circ(i_{\geqslant\lambda}^-)_*=
(i_{\geqslant\lambda}^+)_*\circ R^-_{\{\geqslant\lambda\}\cap\Gamma,\bbd}.$
\end{itemize}
\end{lemma}

\begin{proof}
Part (a) follows from Lemma \ref{lem:CD} and proper base change.
Part (c) follows from (a) by adjunction.
Since $R^+_{\Gamma,\bbd}$ and $R^-_{\Gamma,\bbd}$ are quasi-inverse, part
(d) follows from (c), and (b) from (a).
\end{proof}

\smallskip

\begin{lemma}\label{lem:1S}
\hfill
\begin{itemize}[leftmargin=8mm]
\item[$\mathrm{(a)}$] 
If $\Gamma=\{\lambda\}$ in $\Lambda_\beta$, then we have
$R^-_{\lambda,\bbd}(\k_\lambda\langle d_\lambda^-\rangle)=\k_\lambda\langle-d_\lambda^-\rangle$
in $\D^\b_{\lozenge,\bbd}(\Gr^+_{\lambda,0}\,,\,S)$.
\item[$\mathrm{(b)}$] 
If $\Gamma=\{w\}$ in $M_\beta$, then we have
$R^-_{w,\bbd}(\k_w\langle d_w^-\rangle)=\k_w\langle-d_w^-\rangle$
in $\D^\b_{\lozenge,\bbd}(\Gr^+_{w,0}\,,\,S)$.
\end{itemize}
\end{lemma}

\begin{proof}
Consider the diagram
$$\xymatrix{\Gr^-_{\lambda}&\ar[l]_-{f^-_\lambda}U_{\lambda}\ar[r]^-{f^+_\lambda}&\Gr^+_{\lambda}}.$$
Since $U_\lambda$ is open dense in $\Gr^-_\lambda\times\Gr^+_\lambda$, it has dimension 
$d_\lambda^-+d^+_\lambda$.
Since the map $f^\pm_\lambda$ is $H_\beta$-equivariant over an $H_\beta$-orbit,
it is a fibration whose fiber at $t^\lambda x^\pm$ is isomorphic to
$$G_r(O)t^\lambda x^\mp\cap t^\lambda G_r(O^\mp) x^\mp,$$
which is an affine space (here we write $O^+=O$).
Thus, the map $f^\pm_\lambda$ is a (\'etale locally trivial) 
fibration whose fibers are affine spaces of dimension $d^\mp_\lambda$.
Part (a) follows, because $\Gr^+_{\lambda}$ is simply connected.
Part (b) is proved in a similar way.
\end{proof}

\smallskip

Now, we consider the mixed category $\D^\b_\mix(\Gr^\pm_\Gamma,S)$.
The functors
$(i^\pm_\Gamma)_{*}$, $(i^\pm_\Gamma)_{!}$, $(i^\pm_\Gamma)^{*}$ and $(i^\pm_\Gamma)^{!}$ between 
$\D^\b_\bbd(\Gr^\pm_{\Gamma,0},S)$, $\D^\b_\bbd(\Gr^\pm_{\beta,0},S)$ lift to the mixed categories
$\D^\b_\mix(\Gr^\pm_\Gamma,S)$, $\D^\b_\mix(\Gr^\pm_\beta,S)$ under the functor $\iota$,
and they enjoy all the usual adjunction properties, see Proposition \ref{prop:A}.
We consider the objects in $\D^\b_\mix(\Gr^\pm_\beta\,,\,S)$ given by
\begin{align}\label{DN}
\begin{split}
\Delta(\lambda)^+_\mix=(i_{\leqslant\lambda}^+)_!(i_{\leqslant\lambda}^+)^*IC(\lambda)_\mu^+\quad,\quad
\nabla(\lambda)^+_\mix=(i_{\leqslant\lambda}^+)_*(i_{\leqslant\lambda}^+)^*IC(\lambda)_\mu^+,\\
\Delta(\lambda)^-_\mix=(i_{\geqslant\lambda}^-)_!(i_{\geqslant\lambda}^-)^*IC(\lambda)_\mu^-\quad,\quad
\nabla(\lambda)^-_\mix=(i_{\geqslant\lambda}^-)_*(i_{\geqslant\lambda}^-)^*IC(\lambda)_\mu^-.
\end{split}
\end{align}
The Verdier duality yields an involution $D$ of the category  $\D^\b_\mix(\Gr^\pm_\beta\,,\,S)$
such that
$D(\Delta(\lambda)^\pm_\mix)=\nabla(\lambda)^\pm_\mix$ and
$D(IC(\lambda)^\pm_\mix)=IC(\lambda)_\mix^\pm.$
The complexes  $\Delta(\lambda)^+_\mix$, $\nabla(\lambda)^+_\mix$ 
are non-equivariant analogues of the objects 
of $\D^\b_\mix(\calG r^+_\beta)$ in \eqref{EDN+}.

\smallskip

\begin{proposition}\label{prop:equivalence6}
\hfill
\begin{itemize}[leftmargin=8mm]
\item[$\mathrm{(a)}$] 
We have an equivalence of graded triangulated categories 
$\B_\beta^\sharp:\D^\b_\mix(\Gr^+_\beta,S)\to\D^\b_\mix(\Gr^-_\beta,S)$
\item[$\mathrm{(b)}$] 
$\B_\beta^\sharp\nabla(\lambda)^+_\mix=\Delta(\lambda)_\mix^-$ 
for each $\lambda\in\Lambda_\beta$.
\end{itemize}
\end{proposition}

\begin{proof}

\medskip

\noindent{\bf Step 1 :} \emph{We construct an equivalence 
$S_{\Gamma,\mix}^-:\D^\b_\mix(\Gr^-_\Gamma\,,\,S)\to\D^\b_\mix(\Gr^+_{\Gamma}\,,\,S)$}

\medskip

Let $\P_{\lozenge,\bbd}(\Gr^+_{\Gamma,0}\,,\,T)^\proj$
be the full subcategory of $\P_{\lozenge,\bbd}(\Gr^+_{\Gamma,0}\,,\,T)$
consisting of the objects whose image by $\omega$ belongs to $\P(\Gr^+_{\Gamma}\,,\,T)^\proj$.
Let $\P_{\lozenge,\bbd}(\Gr^+_{\Gamma,0}\,,\,T)^\tilt$
be the full subcategory of $\P_{\lozenge,\bbd}(\Gr^+_{\Gamma,0}\,,\,T)$ consisting of the tilting objects.
By Proposition \ref{prop:even1}, the stratification $T$ on  $\Gr^-_{\Gamma}$ is good.
Hence, by Propositions \ref{prop:A}(a) and \ref{prop:B}(b), 
the functor $\iota$ restricts to a fully faithful functor
$\P\!_\mix(\Gr^-_\Gamma,T)\to\P\!_{\lozenge,\bbd}(\Gr^-_{\Gamma,0},T)$
and to an equivalence
\begin{align}\label{INCL}
\P\!_\mix(\Gr^-_{\Gamma}\,,\,T)^\tilt\to\P_{\lozenge,\bbd}(\Gr^-_{\Gamma,0}\,,\,T)^\tilt.
\end{align}
In particular, the category $\P_{\lozenge,\bbd}(\Gr^-_{\Gamma,0}\,,\,T)^\tilt$
is the full subcategory of $\P_{\lozenge,\bbd}(\Gr^-_{\Gamma,0}\,,\,T)$
consisting of the objects whose image by $\omega$ belongs to $\P(\Gr^-_{\Gamma}\,,\,T)^\tilt$.
We define the functors
$S^\mp_\Gamma=R^\mp_\Gamma[\pm d^+_\Gamma]$
and
$S^\mp_{\Gamma,\bbd}=R^\mp_{\Gamma,\bbd}\langle\pm d^+_\Gamma\rangle.$
By \cite{BBM04}, \cite{Y09}, the functors $S^-_\Gamma$ and $S^+_\Gamma$ 
yield inverse equivalences of additive categories
$$\xymatrix{
\P(\Gr^-_\Gamma\,,\,T)^\tilt\ar@/^/[r]^{S^-_\Gamma}&\ar@/^/[l]^{S^+_\Gamma}\P(\Gr^+_{\Gamma}\,,\,T)^\proj.
}$$
Hence, by \eqref{DIAG7}, we have the diagram of equivalences of additive categories
\begin{align}\label{DIAG9}\begin{split}
\xymatrix{
\P_{\lozenge,\bbd}(\Gr^-_{\Gamma,0}\,,\,T)^\tilt\ar@/^/[r]^{S^-_{\Gamma,\bbd}}\ar[d]_\omega&
\ar@/^/[l]^{S^+_{\Gamma,\bbd}}
\P_{\lozenge,\bbd}(\Gr^+_{\Gamma,0}\,,\,T)^\proj\ar[d]^\omega\\
\P(\Gr^-_{\Gamma}\,,\,T)^\tilt\ar@/^/[r]^{S^-_\Gamma}&
\ar@/^/[l]^{S^+_\Gamma}\P(\Gr^+_{\Gamma}\,,\,T)^\proj.
}
\end{split}
\end{align}
Thus,
the set of indecomposable objects in $\P_{\lozenge,\bbd}(\Gr^+_{\Gamma,0}\,,\,T)^\proj$ is
$$\{S^-_{\Gamma,\bbd}\iota\, T(w)_\mu(a/2)\,;\,w\in W_\Gamma\,,\,a\in\bbZ\}.$$
On the other hand, the set of indecomposable objects in $\P\!_\mix(\Gr^+_{\Gamma}\,,\,T)^\proj$ is
$$\{P(w)_\mu(a/2)\,;\,w\in W_\Gamma\,,\,a\in\bbZ\}.$$
Since the objects $\iota\, P(w)_\mu(a/2)$ and $S^-_{\Gamma,\bbd}\iota\, T(w)_\mu(a/2)$ 
have isomorphic images by $\omega$, we deduce that the functor $\iota$ restricts to an equivalence
\begin{align}\label{INCL2}
\P\!_\mix(\Gr^+_{\Gamma}\,,\,T)^\proj\to\P_{\lozenge,\bbd}(\Gr^+_{\Gamma,0}\,,\,T)^\proj.
\end{align}
From \eqref{DIAG9}, \eqref{INCL}, \eqref{INCL2}, we deduce that there are equivalences of graded additive categories
\begin{align}\label{Pmu}
\xymatrix{
\P\!_\mix(\Gr^-_\Gamma\,,\,T)^\tilt\ar@/^/[r]^{S^-_{\Gamma,\mix}}&\ar@/^/[l]^{S^+_{\Gamma,\mix}}\P\!_\mix(\Gr^+_{\Gamma}\,,\,T)^\proj.
}\end{align}
Taking the homotopy categories, by Propositions \ref{prop:AA}, \ref{prop:B} we get quasi-inverse
equivalences of graded triangulated categories $S^\pm_{\Gamma,\mix}$ yielding the following commutative diagram 
\begin{align}
\label{DIAG2}
\begin{split}
\xymatrix{
\D^\b_\mix(\Gr^-_\Gamma\,,\,T)\ar@/^/[r]^{S^-_{\Gamma,\mix}}\ar[d]_\iota&
\ar@/^/[l]^{S^+_{\Gamma,\mix}}\D^\b_\mix(\Gr^+_{\Gamma}\,,\,T)\ar[d]^\iota,\\
\D^\b_{\lozenge,\bbd}(\Gr^-_{\Gamma,0}\,,\,T)
\ar@/^/[r]^{S^-_{\Gamma,\bbd}}&\ar@/^/[l]^{S^+_{\Gamma,\bbd}}
\D^\b_{\lozenge,\bbd}(\Gr^+_{\Gamma,0}\,,\,T).
}
\end{split}
\end{align}
In other words, the functors $S^\pm_{\Gamma,\bbd}$ between the categories
$\D^\b_{\lozenge,\bbd}(\Gr^\pm_{\Gamma,0}\,,\,T)$ are genuine.
Thus, the functors $S^\pm_{\Gamma,\bbd}$ between the categories
$\D^\b_{\lozenge,\bbd}(\Gr^\pm_{\Gamma,0}\,,\,S)$ are also genuine 
by \cite[lem.~7.21]{AR13d}, i.e., 
we have quasi-inverse
equivalences of graded triangulated categories $S^\pm_{\Gamma,\mix}$ yielding the following commutative diagram 
\begin{align}
\label{DIAG7}
\begin{split}
\xymatrix{
\D^\b_\mix(\Gr^-_\Gamma\,,\,S)\ar@/^/[r]^{S^-_{\Gamma,\mix}}\ar[d]_\iota&
\ar@/^/[l]^{S^+_{\Gamma,\mix}}\D^\b_\mix(\Gr^+_{\Gamma}\,,\,S)\ar[d]^\iota\\
\D^\b_{\lozenge,\bbd}(\Gr^-_{\Gamma,0}\,,\,S)
\ar@/^/[r]^{S^-_{\Gamma,\bbd}}&\ar@/^/[l]^{S^+_{\Gamma,\bbd}}
\D^\b_{\lozenge,\bbd}(\Gr^+_{\Gamma,0}\,,\,S).
}
\end{split}
\end{align}

\smallskip

\medskip

\noindent{\bf Step 2 :} \emph{We prove that $S^-_{\beta,\mu}(\nabla(\lambda)_\mix^-)=\Delta(\lambda)_\mix^+$.}

\medskip

Since the map $i_{\leqslant\mu}^-$ is a closed embedding, the functor $(i_{\leqslant\mu}^-)_!$ 
preserves tilting mixed perverse sheaves.
Hence, by Lemma \ref{lem:FORM2} and \eqref{DIAG9}, the inclusion $\{\leqslant\!\mu\}\cap\Gamma\subset\Gamma$ 
yields the commutative square of functors
\begin{align}\label{sq1}
\begin{split}
\xymatrix{
\P\!_{\lozenge,\bbd}(\Gr^-_{\{\leqslant\mu\}\cap\Gamma,0}\,,\,T)^\tilt\ar[rr]^-{S^-_{\{\leqslant\mu\}\cap\Gamma,\bbd}}\ar[d]_-{(i_{\leqslant\mu}^-)_!}
&&\P\!_{\lozenge,\bbd}(\Gr^+_{\{\leqslant\mu\}\cap\Gamma,0}\,,\,T)^\proj\ar[d]^-{(i_{\leqslant\mu}^+)_!}\\
\P\!_{\lozenge,\bbd}(\Gr^-_{\Gamma,0}\,,\,T)^\tilt\ar[rr]^{S^-_{\Gamma,\bbd}}&&
\P\!_{\lozenge,\bbd}(\Gr^+_{\Gamma,0}\,,\,T)^\proj.
}
\end{split}\end{align}
From  \eqref{INCL}, \eqref{INCL2}, we deduce that there is a commutative square of functors
\begin{align}\label{sq5}
\begin{split}
\xymatrix{
\P\!_\mix(\Gr^-_{\{\leqslant\mu\}\cap\Gamma}\,,\,T)^\tilt\ar[rr]^-{S^-_{\{\leqslant\mu\}\cap\Gamma,\mix}}\ar[d]_-{(i_{\leqslant\mu}^-)_!}
&&\P\!_\mix(\Gr^+_{\{\leqslant\mu\}\cap\Gamma}\,,\,T)^\proj\ar[d]^-{(i_{\leqslant\mu}^+)_!}\\
\P\!_\mix(\Gr^-_{\Gamma}\,,\,T)^\tilt\ar[rr]^{S^-_{\Gamma,\mix}}&&
\P\!_\mix(\Gr^+_{\Gamma}\,,\,T)^\proj.
}
\end{split}\end{align}
Taking the homotopy categories, by Propositions \ref{prop:AA}, \ref{prop:B} we get the following
commutative square of functors
\begin{align}\label{sq9}
\begin{split}
\xymatrix{
\D^\b_\mix(\Gr^-_{\{\leqslant\mu\}\cap\Gamma}\,,\,S)\ar[rr]^-{S^-_{\{\leqslant\mu\}\cap\Gamma,\mix}}\ar[d]_-{(i_{\leqslant\mu}^-)_!}
&&\D^\b_\mix(\Gr^+_{\{\leqslant\mu\}\cap\Gamma}\,,\,S)\ar[d]^-{(i_{\leqslant\mu}^+)_!}\\
\D^\b_\mix(\Gr^-_{\Gamma}\,,\,S)\ar[rr]^{S^-_{\Gamma,\mix}}&&
\D^\b_\mix(\Gr^+_{\Gamma}\,,\,S).
}
\end{split}\end{align}

Similarly, since the map $i_{\geqslant\lambda}^-$ is an open embedding, the functor $(i_{\geqslant\lambda}^-)^*$ 
preserves tilting mixed perverse sheaves.
Hence, the inclusion $\{\geqslant\!\lambda\}\cap\Gamma\subset\Gamma$ yields the commutative square of functors
\begin{align}\label{sq2}
\begin{split}\xymatrix{
\P\!_{\lozenge,\bbd}(\Gr^-_{\{\geqslant\lambda\}\cap\Gamma,0}\,,\,T)^\tilt
\ar[rr]^-{S^-_{\{\geqslant\lambda\}\cap\Gamma,\bbd}}
&&\P\!_{\lozenge,\bbd}(\Gr^+_{\{\geqslant\lambda\}\cap\Gamma,0}\,,\,T)^\proj\\
\P\!_{\lozenge,\bbd}(\Gr^-_{\Gamma,0}\,,\,T)^\tilt\ar[rr]^{S^-_{\Gamma,\bbd}}\ar[u]^-{(i_{\geqslant\lambda}^-)^*}&&
\P\!_{\lozenge,\bbd}(\Gr^+_{\Gamma,0}\,,\,T)^\proj\ar[u]_-{(i_{\geqslant\lambda}^+)^*\langle d\rangle},
}
\end{split}\end{align}
where $d=d^+_{\{\geqslant\lambda\}\cap\Gamma}-d^+_\Gamma$.
Taking the adjoint functors and using a similar argument as above, 
we obtain the commutative square of functors
\begin{align}\label{sq8}
\begin{split}\xymatrix{
\D^\b_\mix(\Gr^-_{\{\geqslant\lambda\}\cap\Gamma}\,,\,S)
\ar[rr]^-{S^-_{\{\geqslant\lambda\}\cap\Gamma,\mix}}\ar[d]_-{(i_{\geqslant\lambda}^-)_*}
&&\D^\b_\mix(\Gr^+_{\{\geqslant\lambda\}\cap\Gamma}\,,\,S)\ar[d]^-{(i_{\geqslant\lambda}^+)_*\langle -d\rangle}\\
\D^\b_\mix(\Gr^-_{\Gamma}\,,\,S)\ar[rr]^{S^-_{\Gamma,\mix}}&&\D^\b_\mix(\Gr^+_{\Gamma}\,,\,S).
}
\end{split}\end{align}
In particular,  applying \eqref{sq9} and \eqref{sq8} to the chain of inclusions
$\{\lambda\}\subset\{\leqslant\!\lambda\}\subset\Lambda_\beta$
and using the identity $d^-_\lambda+d^+_\lambda=d^+_\beta$,
we get the commutative diagram of functors
\begin{align}\label{DIAG3}
\begin{split}
\xymatrix{
\D^\b_\mix(\Gr^-_{{\lambda}}\,,\,S)\ar[r]^-{S^-_{\lambda,\mix}}\ar[d]_-{(i_{\lambda}^-)_*}
&\D^\b_\mix(\Gr^+_{{\lambda}}\,,\,S)\ar[d]^-{(i_{\lambda}^+)_!\langle d^-_\lambda\rangle}\\
\D^\b_\mix(\Gr^-_{{\leqslant\lambda}}\,,\,S)
\ar[r]^-{S^-_{\leqslant\lambda,\mix}}\ar[d]_-{(i_{\leqslant\lambda}^-)_*}
&\D^\b_\mix(\Gr^+_{{\leqslant\lambda}}\,,\,S)\ar[d]^-{(i_{\leqslant\lambda}^+)_!}\\
\D^\b_\mix(\Gr^-_{\beta}\,,\,S)\ar[r]^{S^-_{\beta,\mix}}&\D^\b_\mix(\Gr^+_{\beta}\,,\,S).
}
\end{split}
\end{align}
Lemma \ref{lem:1S}, \eqref{DN} and \eqref{DIAG3} yield an isomorphism
\begin{align*}
S^-_{\beta,\mix}(\nabla(\lambda)^-_\mix)
&=S^-_{\beta,\mix}(i_{\geqslant\lambda}^-)_*(i_{\geqslant\lambda}^-)^*(IC(\lambda)^-_\mix),\\
&=S^-_{\beta,\mix}(i_{\geqslant\lambda}^-)_*(i_{\lambda}^-)_*(\k\langle d_\lambda^-\rangle),\\
&=S^-_{\beta,\mix}(i_{\leqslant\lambda}^-)_*(i_{\lambda}^-)_*(\k\langle d_\lambda^-\rangle),\\
&=(i_{\leqslant\lambda}^+)_!(i_{\lambda}^+)_!S^-_{\lambda,\mix}(\k\langle 2d_\lambda^-\rangle),\\
&=(i_{\leqslant\lambda}^+)_!(i_{\lambda}^+)_!(\k\langle d_\lambda^+\rangle),\\
&=\Delta(\lambda)^+_\mix.
\end{align*}
Thus, the functor $\B_\beta^\sharp=D\circ S^+_{\beta,\mu}\circ D$ 
satisfies the conditions in the proposition.

\end{proof}

\medskip

\section{Coherent sheaves on the nilpotent cone}

\smallskip

\subsection{Perverse coherent sheaves on the nilpotent cone}\label{sec:PCoh}

From now on, we write $G_r=GL_{r,\k}$.
This notation differs from the notation used in the beginning of \S\ref{sec:Gr}.
Let $(T_r,B_r,N_r)$ be the standard Borel triple.
We may abbreviate $T=T_r$, $B=B_r$ or $N=N_r$.
The group of characters of $T$ is identified with $\bbZ^r$.
Let $\frakg_r$, $\frakb_r$, $\frakn_r$ be the Lie algebra of $G_r$, $B_r$, $N_r$ and
$i:\calN_r\to\frakg_r$ the embedding of the nilpotent cone.
We'll abbreviate $G_r^c=G_r\times\bbG_{m,\k}$. 
The group $G_r^c$ acts on $\frakg_r$, $\calN_r$ via
$$(g,z)\cdot\xi=z^{-2}\Ad(g)(\xi)\quad,\quad\forall\xi\in\frakg_r\quad,\quad\forall (g,z)\in G_r^c.$$ 
For $X=\frakg_r$ or $\calN_r$, let $\Coh^{G_r^c}(X)$ be the category of
$G_r^c$-coherent sheaves on $X$. We equip this Abelian category 
with the grading shift functors $\langle\bullet\rangle$ such that $\langle -1\rangle$ is the tensor product
with the tautological 1-dimensional $\bbG_m$-module (of weight 1).
We equip the triangulated category $\D^\b\Coh^{G_r^c}(X)$ with the grading shift triangulated functors 
$\langle\bullet\rangle$.
We consider  Bezrukavnikov's graded Abelian subcategory of perverse coherent sheaves
$$\PCoh([\calN_r/G_r^c])\subset\D^\b\Coh([\calN_r/G_r^c]).$$

\smallskip

By \cite{M13}, it is
\emph{graded properly stratified}, with standard, costandard, proper standard, proper costandard 
and simple objects given by
$\pmb\Delta(\lambda)^\sharp$,
$\pmb\nabla(\lambda)^\sharp$,
$\pmb{\bar\Delta}(\lambda)^\sharp$,
$\pmb{\bar\nabla}(\lambda)^\sharp$,
$\bfL(\lambda)^\sharp$ with $\lambda\in\Lambda_r.$
Let us recall briefly their definitions.
Consider the diagram of $G_r^c$-equivariant $\k$-schemes
\begin{align*}
\xymatrix{\calN_r&\ar[l]_-q G_r\times_{B_r}\frakn_r\ar[r]^-p&G_r/B_r.}
\end{align*}
For each weight $\lambda\in\bbZ^r$, the  \emph{Andersen-Jantzen sheaf} $A(\lambda)$ 
is the complex in $\in\D^\b\Coh([\calN_r/G_r^c])$ given by
$A(\lambda)=Rq_*p^*\calO_{G_r/B}(\lambda).$
It is perverse.
Set $\delta_{\lambda}=\text{min}\{\ell(w)\,;\,ww_0\lambda\in\Lambda_r\}.$
If $\lambda=\lambda_\pi$ as in \eqref{bij}, then we have
\begin{align*}
\delta_{\lambda_\pi}=r(r-1)/2-\sum_{k=0}^lp_k(p_k-1)/2.
\end{align*}
The proper standard and costandard objects are given by
\begin{align}\label{PS}
\pmb{\bar\Delta}(\lambda)^\sharp=A(w_0\lambda)\langle\delta_{\lambda}\rangle\quad,\quad
\pmb{\bar\nabla}(\lambda)^\sharp=A(\lambda)\langle-\delta_{\lambda}\rangle
\quad,\quad\lambda\in\Lambda_r.
\end{align}
Let $V(\lambda)$ be the simple rational $G_r$-module with highest weight $\lambda$.
We define
\begin{align*}
\bfT(\lambda)=V(\lambda)\otimes\calO_{\frakg_r}\quad,\quad
\bfT(\lambda)^\sharp=Li^*\bfT(\lambda)=V(\lambda)\otimes\calO_{\calN_r}.
\end{align*}
The set of indecomposable tilting objects is
$\{\bfT(\lambda)^\sharp\langle a\rangle\,;\,\lambda\in\Lambda_r\,,\,a\in\bbZ\}.$
For any order ideal  $\Gamma\subset\Lambda_r$ we consider the Serre subcategory 
$\PCoh([\calN_r/G_r^c])_\Gamma$ of $\PCoh([\calN_r/G_r^c])$ 
generated by the constituents of the objects in
$\{\bfT(\lambda)^\sharp\langle a\rangle\,;\,\lambda\in\Gamma\,,\,a\in\bbZ\}.$
We have
$$\D^\b\Coh([\calN_r/G_r^c])=\D^\b\PCoh([\calN_r/G_r^c]).$$ 
Setting $\Gamma=\{<\lambda\}$ for some weight $\lambda\in\Lambda_r$,  we may consider the quotient functor 
$$(\pi_{\geqslant\lambda})^*:\D^\b\Coh([\calN_r/G_r^c])\to
\D^\b(\PCoh([\calN_r/G_r^c])\,/\,\PCoh([\calN_r/G_r^c])_{<\lambda}).$$
Let $(\pi_{\geqslant\lambda})_!$ and $(\pi_{\geqslant\lambda})_*$ be the left and right adjoints,
see \cite[\S\S 1,2]{A15d} for details.
We define
$$\pmb\Delta(\lambda)^\sharp=(\pi_{\geqslant\lambda})_!(\pi_{\geqslant\lambda})^*(\bfT(\lambda)^\sharp)\langle-\delta_\lambda\rangle
\quad,\quad
\pmb\nabla(\lambda)^\sharp=(\pi_{\geqslant\lambda})_*(\pi_{\geqslant\lambda})^*(\bfT(\lambda)^\sharp)\langle\delta_\lambda\rangle.$$
Both objects belong to the category $\PCoh([\calN_r/G_r^c])_{\leqslant\!\lambda}$. 
They are called the standard and costandard objects respectively.

\smallskip

Being the Serre subcategory 
associated with the order ideal $\Gamma$ of $\Lambda_r$,  the Abelian category
$\PCoh([\calN_r/G_r^c])_\Gamma$ is graded properly stratified, with the
standard, costandard, proper standard 
and proper costandard objects given by
$\pmb\Delta(\lambda)^\sharp$, 
$\pmb\nabla(\lambda)^\sharp$, 
$\pmb{\bar\Delta}(\lambda)^\sharp$,
$\pmb{\bar\nabla}(\lambda)^\sharp$ with
$\lambda\in\Gamma.$

\smallskip

Consider the graded triangulated subcategories   
$$\D^\b\Coh([\frakg_r/G_r^c])_\Gamma\subset\D^\b\Coh([\frakg_r/G_r^c])
\quad,\quad
\D^\perf\Coh([\calN_r/G_r^c])_\Gamma\subset\D^\perf\Coh([\calN_r/G_r^c])$$
generated by the sets
$\bfT_\Gamma=\{\bfT(\lambda)\langle a\rangle\,;\,\lambda\in\Gamma\,,\,a\in\bbZ\}$
and
$\bfT_\Gamma^\sharp=\{\bfT(\lambda)^\sharp\langle a\rangle\,;\,\lambda\in\Gamma\,,\,a\in\bbZ\}$.
We have
$$\D^\perf\Coh([\calN_r/G_r^c])_\Gamma=\D^\perf(\PCoh([\calN_r/G_r^c])_\Gamma).$$

\smallskip

\subsection{The derived Satake equivalence}

Recall the graded triangulated categories 
$\D^\b_\mix(\calG r^-_\beta)=\K^\b(\Par(\calG r^-_\beta))$ and
$\D^\b_\mix(\Gr^-_\beta,S)=\K^\b(\Par(\Gr^-_\beta,S)).$
The following  is an instance of the \emph{derived Satake equivalence} of Bezrukavnikov-Finkelberg.

\smallskip

\begin{proposition}\label{prop:equivalence7}
\hfill
\begin{itemize}[leftmargin=8mm]
\item[$\mathrm{(a)}$] 
There is an equivalence of graded triangulated categories 
$\C_\beta:\D^\b_\mix(\calG r^-_\beta)\to\D^\b\Coh([\frakg_r/G_r^c])_{\Lambda_\beta}$
such that
$\C_\beta(IC(\lambda)^-_\mix)=\bfT(\lambda)$ for all $\lambda\in \Lambda_\beta.$
\item[$\mathrm{(b)}$] 
There is an equivalence of graded triangulated categories 
$\C_\beta^\sharp:\D^\b_\mix(\Gr^-_\beta,S)\to\D^\perf\Coh([\calN_r/G_r^c])_{\Lambda_\beta}$
such that the following diagram commutes
\begin{align}\label{diag5}
\begin{split}
\xymatrix{
\D^\b_\mix(\Gr^-_\beta,S)\ar[d]_-{\C_\beta^\sharp}
\ar@/^/[r]^-{\xi_*}&\ar@/^/[l]^-{L\xi^*}
\D^\b_\mix(\calG r^-_\beta)\ar[d]^-{\C_\beta}\\
\D^\perf\Coh([\calN_r/G_r^c])_{\Lambda_\beta}
\ar@/^/[r]^-{i_*}&\ar@/^/[l]^-{Li^*}
\D^\b\Coh([\frakg_r/G_r^c])_{\Lambda_\beta}.
}
\end{split}
\end{align}
\item[$\mathrm{(c)}$]  For each $\lambda\in \Lambda_\beta$  we have
$\C_\beta^\sharp\Delta(\lambda)^-_\mix=\pmb\Delta(\lambda)^\sharp$,
$\C_\beta^\sharp\nabla(\lambda)^-_\mix=\pmb\nabla(\lambda)^\sharp$ and
$\C_\beta^\sharp IC(\lambda)^-_\mix=\bfT(\lambda)^\sharp.$
\end{itemize}
\end{proposition}

\begin{proof}
Since the group $K_\beta$ is unipotent, the obvious functor
$$\D^\b_\mix(\calG r^-_\beta)=\D^\b_\mix([\Gr^-_\beta/H_\beta])\to
\D^\b_\mix([\Gr^-_\beta/G_r(O)])$$
is a triangulated equivalence, compare \cite[thm.~3.7.3]{BL}.
Hence, part (a) follows from  \cite[thm.~5.5.1]{R17}. Part (b) is \cite[prop.~5.7]{A15d}.
More precisely, let $\Add(\bfT_\beta)$ be the additive full subcategory of $\Coh([\frakg_r/G_r^c])$
generated by the set of objects $\bfT_\beta$. 
By \cite[cor.~5.5.4]{R17}, we have an equivalence of triangulated categories
$\D^\b\Coh([\frakg_r/G_r^c])_{\Lambda_\beta}=\K^\b(\Add(\bfT_\beta)).$
The equivalence $\C_\beta$ follows from the isomorphism \cite[(5.6.2)]{R17}
$$\Hom_{\D^\b(\calG r^-_\beta)}(IC(\lambda)^-\,,\,IC(\mu)^-\langle a\rangle)=
\Hom_{\D^\b\Coh([\frakg_r/G_r^c])}(\bfT(\lambda)\,,\,\bfT(\mu)\langle a\rangle)
\quad,\quad\forall\lambda,\mu\in\Lambda_\beta$$
 which yields an equivalence of graded additive categories
$\Par(\calG r^-_\beta)=\Add(\bfT_\beta)$
which commutes with the grading shift functors $\langle\bullet\rangle$.
The equivalence $\C_\beta^\sharp$ is proved in a similar way in \cite[prop.~5.5,\,thm.~2.16]{A15d}.

\smallskip

The commutativity of the equivalences $\C_\beta$, $\C_\beta^\sharp$ with the functors $L\xi^*$, $Li^*$
is a consequence of the isomorphisms \eqref{isom33} and the isomorphism
\begin{align}\label{ISOM3}
\k\otimes_{H^\bullet_{G_r}}
\Hom_{\D^\b\Coh([\frakg_r/G_r^c])}(\bfT(\lambda)\,,\,\bfT(\mu)\langle a\rangle)
=\Hom_{\D^\perf\Coh([\calN_r/G_r^c])}(\bfT(\lambda)^\sharp\,,\,\bfT(\mu)^\sharp\langle a\rangle),
\end{align}
where the $H^\bullet_{G_r}$-module structure  comes from the identification of graded rings
$H^\bullet_{G_r}=\k[\frakg_r/\!/G_r]$.
The commutativity with $\xi_*$, $i_*$ follows by adjunction.

\smallskip

More precisely, in the derived category of $\calO_{\frakg_r/\!/G_r}$-modules, we have
$$
\RHom_{\D^\b\Coh([\frakg_r/G_r^c])}(\bfT(\lambda)\,,\,\bfT(\mu)\langle a\rangle)
=\Inv^{G_r\times(\frakg_r/\!/G_r)}
\RHom_{\D^\b(\Coh(\frakg_r))}(\bfT(\lambda)\,,\,\bfT(\mu)\langle a\rangle),
$$
where $\Inv^{G_r\times(\frakg_r/\!/G_r)}$ is the derived functor of invariants
relatively to the flat affine group scheme $G_r\times(\frakg_r/\!/G_r)$ over 
$\frakg_r/\!/G_r$.
Derived invariants commute with the derived base change functor
$M\mapsto\k\stackrel{\L}{\otimes}_{H^\bullet_{G_r}}M$,
see, e.g., \cite[App.~A]{MR16}. We deduce that
\begin{align*}
\k\stackrel{\L}{\otimes}_{H^\bullet_{G_r}}
\RHom_{\D^\b\Coh([\frakg_r/G_r^c])}(\bfT(\lambda)\,,\,\bfT(\mu)\langle a\rangle)
&=\Inv^{G_r}\Big(\k\stackrel{\L}{\otimes}_{H^\bullet_{G_r}}
\RHom_{\D^\b(\Coh(\frakg_r))}(\bfT(\lambda)\,,\,\bfT(\mu)\langle a\rangle)\Big)\\
&=\RHom_{\D^\perf\Coh([\calN_r/G_r^c])}(\bfT(\lambda)^\sharp\,,\,\bfT(\mu)^\sharp\langle a\rangle).
\end{align*}
Now, since $\bfT(\lambda)$,  $\bfT(\lambda)^\sharp$ are locally free and $G_r$ is reductive, we have
$$\RHom^{>0}_{\D^\b\Coh([\frakg_r/G_r^c])}(\bfT(\lambda)\,,\,\bfT(\mu)\langle a\rangle)=
\RHom^{>0}_{\D^\perf\Coh([\calN_r/G_r^c])}(\bfT(\lambda)^\sharp\,,\,\bfT(\mu)^\sharp\langle a\rangle)=0.$$
The isomorphism \eqref{ISOM3} follows.
Part (c) of the proposition is \cite[cor.~3.4]{AR16a}.
\end{proof}

\smallskip

We can now prove the following theorem.

\smallskip

\begin{theorem}\label{thm:final1} There is an equivalence of graded triangulated categories
$$\E^\sharp_\beta:\D^\perf(\calD_\beta^{\,\sharp})\to\D^\perf\Coh([\calN_r/G_r^c])_{\Lambda_\beta}$$
such that $\E^\sharp_\beta\Delta(\pi)^\sharp=\pmb\Delta(\lambda_\pi)^\sharp$
for all $\pi\in\Gamma\!_\beta$.

\end{theorem}

\begin{proof}
Let first observe that
\begin{align}\label{TRIA1}
\Delta(\pi)^\sharp\in\D^\perf(\calD_\beta^{\,\sharp})\quad,\quad
\pmb\Delta(\lambda_\pi)^\sharp\in\D^\perf\Coh([\calN_r/G_r^c])_{\Lambda_\beta}
\quad,\quad\forall\pi\in\Gamma\!_\beta.
\end{align}
The second identity follows from \cite[prop.~5.4]{M13}, the first one from
Lemma \ref{lem:DP}(b) below, and the fact that the category $\calD_\beta$ has finite global dimension.
Now, by Propositions \ref{prop:equivalence6}, \ref{prop:equivalence3}, 
 \ref{prop:equivalence7}  there is
a chain of equivalences of graded triangulated categories
\begin{align*}
\xymatrix{
\E^\sharp_\beta:\D^\perf(\calD_\beta^{\,\sharp})\ar[r]^-{\A^\sharp_\beta}&
\D^\b_\mix(\Gr^+_\beta,S)\ar[r]^-{\B_\beta^\sharp}&\D^\b_\mix(\Gr^-_\beta,S)\ar[r]^-{\C_\beta^\sharp}&\D^\perf\Coh([\calN_r/G_r^c])_{\Lambda_\beta}
}
\end{align*}
such that
$$\xymatrix{
\Delta(\pi)^\sharp\langle a\rangle\ar@{|->}[r]^-{\A^\sharp_\beta}&
\nabla(\lambda_\pi)^+_\mix\langle a\rangle\ar@{|->}[r]^-{\B_\beta^\sharp}&
\Delta(\lambda_\pi)^-_\mix\langle a\rangle\ar@{|->}[r]^-{\C_\beta^\sharp}&
\pmb\Delta(\lambda_\pi)^\sharp\langle a\rangle.
}$$
\smallskip
\end{proof}

\smallskip

\begin{conjecture} \label{conj:1}
There is an equivalence of graded triangulated categories
$$\E_\beta:\D^\b(\calD_\beta)\to\D^\b\Coh([\frakg_r/G_r^c])_{\Lambda_\beta}$$
such that the following square of functors commutes
$$\xymatrix{
\D^\perf(\calD^{\,\sharp}_\beta)\ar[d]_-{\E_\beta^\sharp}
\ar@/^/[r]^-{\xi_*}&\ar@/^/[l]^-{L\xi^*}
\D^\b(\calD_\beta)\ar[d]^-{\E_\beta}\\
\D^\perf\Coh([\calN_r/G_r^c])_{\Lambda_\beta}
\ar@/^/[r]^-{i_*}&\ar@/^/[l]^-{Li^*}
\D^\b\Coh([\frakg_r/G_r^c])_{\Lambda_\beta}.}$$
Further, we have
$\E_\beta\Delta(\pi)=\pmb\Delta(\lambda_\pi)$
for each $\pi\in\Gamma\!_\beta$.
\qed
\end{conjecture}

\smallskip

\begin{remark}
Given any weight dominant $\lambda$,
let $P_\lambda$ be the largest parabolic subgroup such that the line bundle
$\calO_{G_r/B}(\lambda)$ on $G_r/B$ is the 
pull-back of a line bundle on $G_r\,/\,P_\lambda$. 
Let $\calO_{G_r\,/\,P_\lambda}(\lambda)$ denote this line bundle.
Let $\frakp_\lambda$ be the Lie algebra of $P_\lambda$.
The diagram of $G_r^c$-equivariant $\k$-schemes 
\begin{align*}
\xymatrix{\frakg_r&\ar[l]_-{q_\lambda} G_r\times_{P_\lambda}
\frakp_\lambda\ar[r]^-{p_\lambda}&G_r/P_\lambda}
\end{align*}
gives rise to the following complexes
\begin{align*}
\pmb\Delta(\lambda)=(Rq_{w_0\lambda})_*(p_{w_0\lambda})^*\calO_{G_r\,/\,P_{w_0\lambda}}(w_0\lambda)
\quad,\quad
\pmb\nabla(\lambda)=(Rq_\lambda)_*(p_\lambda)^*\calO_{G_r\,/\,P_\lambda}(\lambda).
\end{align*}
The standard and costandard objects 
should be made more explicit in the following way:
$\pmb\Delta(\lambda)^\sharp=Li^*\pmb\Delta(\lambda)$ and
$\pmb\nabla(\lambda)^\sharp=Li^*\pmb\nabla(\lambda)$.
This should follows from the computation in  \cite[(29)-(30)]{B03}.

\end{remark}

\medskip

\section{The equivalence of monoidal categories}

In Theorem \ref{thm:final1} we have constructed an equivalence of graded triangulated categories
$$\E^\sharp_\beta:\D^\perf(\calD_\beta^{\,\sharp})\to\D^\perf\Coh([\calN_r/G_r^c])_{\Lambda_\beta}.$$
Our next goal is to prove that it yields indeed an equivalence of graded Abelian
categories
$$\bigoplus_{\beta\in Q_{++}}\calD_\beta^{\,\sharp}\to\bigoplus_{\beta\in Q_{++}}\PCoh([\calN_r/G_r^c])_{\Lambda_\beta}.$$

\smallskip

\subsection{The monoidal structure on  $\calD^{\,\sharp}$} \label{sec:monD}

Consider the graded Abelian categories 
$$\calD^{\,\sharp}=\bigoplus_{\beta\in Q_{++}}\calD_\beta^{\,\sharp}\quad,\quad
\calD=\bigoplus_{\beta\in Q_{++}}\calD_\beta\quad,\quad
\calC=\bigoplus_{\beta\in Q_+}\calC_\beta.$$
We have obvious fully faithful functors $\calD^{\,\sharp}\,\subset\,\calD\,\subset\,\calC.$
The category $\calD_\beta^\sharp$ is Artinian by Remark \ref{rem:fl}, for each $\beta\in Q_{++}$.
Hence we have $\calD^\sharp\subset\calD^\fl$.
For each $\alpha,\gamma\in Q_+$ with $\beta=\alpha+\gamma$, 
there is an obvious inclusion $\bfR_{\alpha}\otimes\bfR_{\gamma}\subset\bfR_\beta$.
The induction functors relative to those embeddings
equip the categories $\calC$ and $\calC^\fl$ with an exact monoidal structure $\circ$. 

\smallskip

\begin{lemma}\label{lem:Dmonoidal} 
The bifunctor $\circ$ preserves the subcategory $\calD^\fl$ of $\calC^\fl$.
\end{lemma}

\begin{proof}
For each Kostant partitions $\pi,$ $\sigma$ supported on $\Phi_{++}$, let us check that the induced
module $L(\pi)\circ L(\sigma)$ belongs to $\calD$. By the associativity of the induction, we may assume that
$\pi=(\beta_n)$ and  $\sigma=(\beta_m)$ with $n,m\in\bbN$. If $n\geqslant m$ then 
$L(\beta_n)\circ L(\beta_m)$
is the proper standard module $\bar\Delta(\beta_n,\beta_m)$ by \eqref{Delta}.
If $n\leqslant m$ then $L(\beta_n)\circ L(\beta_m)$
is the proper costandard module $\bar\nabla(\beta_n,\beta_m)$, up to some grading shift, 
by \cite[(5.2), thm.~10.1(2)]{Mc18} and the definition of the proper costandard modules in \cite[\S24]{Mc18}. 
Hence it lies again in $\calD$.
\end{proof}

\smallskip

\begin{proposition} \label{prop:monoidal}
The bifunctor $\circ$ 
preserves the subcategory $\calD^{\,\sharp}$ of $\calC^\fl$, yielding an exact graded monoidal full subcategory
$(\calD^{\,\sharp},\circ)$ of $(\calC^\fl,\circ).$
\end{proposition}

\begin{proof}
The $Z(\bfR_\beta)$-action on $\bfR_\beta$ by multiplication preserves the subalgebra
$\bfR_{\alpha}\otimes\bfR_{\gamma}$.
The action on the unit of $\bfR_{\alpha}\otimes\bfR_{\gamma}$
gives an inclusion 
\begin{align}\label{INC1}Z(\bfR_\beta)\subset Z(\bfR_{\alpha})\otimes Z(\bfR_{\gamma})\end{align} 
which fits into the following commutative diagram
\begin{align}\label{DIAG4}
\begin{split}
\xymatrix{
Z(\bfR_\beta)\ar[r]^-{\eqref{INC1}}& Z(\bfR_{\alpha})\otimes Z(\bfR_{\gamma})\\
\ar@{=}[u]H^\bullet_{G_\beta}\ar[r]^-{\Delta_{\alpha,\gamma}}&
H^\bullet_{G_\alpha}\otimes H^\bullet_{G_\gamma}\ar@{=}[u]
}
\end{split}
\end{align}
where
$\Delta_{\alpha,\gamma}$ is the diagonal map.
The formal series $E_\beta(t)$ in \eqref{Eb} is such that
\begin{align}\label{ID4}\Delta_{\alpha,\gamma}E_\beta(t)=E_{\alpha}(t)\otimes E_{\gamma}(t).\end{align}
Thus, under the inclusion \eqref{INC1}, we have
\begin{align}\label{ID3}
J_\beta\subset (J_\alpha\otimes 1)+(1\otimes J_\gamma).
\end{align}
We deduce that
$J_\beta$ annihilates the $\bfR_\beta$-module
$$M\circ N=\bfR_\beta\otimes_{\bfR_{\alpha}\otimes\bfR_{\gamma}}(M\otimes N)$$
for any modules $M\in\calD_{\alpha}$, $N\in\calD_{\gamma}$ killed by 
$J_\alpha$, $J_\gamma$ respectively.
The category $\calD_\beta^{\,\sharp}$ consists of the modules in $\calD_\beta$ which are killed by $J_\beta$.
From the discussion above, we deduce that
the monoidal structure $\circ$  preserves the subcategory $\calD^{\,\sharp}$ of $\calD$, i.e., the functor $\xi_*$ in
\eqref{xi1} extends to a monoidal functor
$(\calD^{\,\sharp},\circ)\to (\calD,\circ).$
\end{proof}

\smallskip

\subsection{The monoidal structure on perverse coherent sheaves on the nilpotent cone}

Fix $\alpha,\beta,\gamma\in Q_{++}$ with $\beta=\alpha+\gamma$.
Write $\alpha=u\alpha_0+k\delta$ and $\gamma=v\alpha_0+l\delta$.
Let $P_{u,v}^c$ be the standard parabolic subgroup of $G_r^c$ with
Levi subgroup $G_{u}\times G_{v}\times\bbG_m$.
Let $\calN_{u,v}$ be the nilpotent cone of $P_{u,v}^c$.
We consider the following diagram of Artin stacks
\begin{align}\label{DIAG-IND}
\xymatrix{
[\calN_u/G_u^c]\times[\calN_v/G_v^c]
&\ar[l]_-{q}[\calN_{u,v}/P_{u,v}^c]\ar[r]^-{p}&
[\calN_r/G_r^c].
}
\end{align}
Since the morphism $q$ is smooth, we have a triangulated bifunctor
$$\circ=Rp_*\, q^*:
\D^\b\Coh([\calN_u/G_u^c])\times\D^\b\Coh([\calN_v/G_v^c])\to
\D^\b\Coh([\calN_r/G_r^c]).$$
We consider the graded Abelian category  given by
\begin{align*}\PCoh([\calN/G^c])_{\Lambda^+}
=\bigoplus_{r>0}\PCoh([\calN_r/G_r^c])_{\Lambda^+_r}
=\bigoplus_{\beta\in Q_{++}}\PCoh([\calN_r/G_r^c])_{\Lambda_\beta}.
\end{align*}

\smallskip

\begin{proposition}\label{prop:mon4}
The functor $\circ$ restricts to an exact bifunctor of graded Abelian categories
$$\circ:\PCoh([\calN_u/G_u^c])\times\PCoh([\calN_v/G_v^c])\to
\PCoh([\calN_r/G_r^c]),$$
giving rise to an exact graded monoidal category 
$\big(\PCoh([\calN/G^c])_{\Lambda^+}\,,\,\circ\big)$.
\end{proposition}

\begin{proof} For any characters $\mu,$ $\nu$ of $T_{u}$, $T_{v}$ respectively, 
let $(\mu,\nu)$ be the character of the torus
$T_r=T_{u}\times T_{v}$ obtained by glueing $\mu$ with $\nu$.
The base change theorem yields the following isomorphism of Andersen-Jantzen sheaves
$A(\mu)\circ A(\nu)=A(\mu,\nu).$ 
Hence, for all dominant $\mu$, $\nu$ the complex
$\pmb{\bar\Delta}(\mu)^\sharp\,\circ\,\pmb{\bar\Delta}(\nu)^\sharp$
is perverse.
We deduce that the complex
$\bfL(\mu)^\sharp\,\circ\,\bfL(\nu)^\sharp$ is perverse as well, and 
the proposition follows.
\end{proof}

\smallskip

\begin{remark}
By Proposition \ref{prop:mon4}, the triangulated category
$$\D^\perf\Coh([\calN/G^c])_{\Lambda^+}=
\bigoplus_{\beta\in Q_{++}}\D^\perf\Coh([\calN_r/G_r^c])_{\Lambda_\beta}$$
is monoidal.
Considering instead of \eqref{DIAG-IND} the following diagram of Artin stacks
\begin{align*}
\xymatrix{
[\frakg_u/G_u^c]\times[\frakg_v/G_v^c]
&\ar[l][\frakp_{u,v}/P_{u,v}^c]\ar[r]&
[\frakg_r/G_r^c],
}
\end{align*}
we get a triangulated monoidal structure on the triangulated category defined by
$$\D^\b\Coh([\frakg/G^c])_{\Lambda^+}=
\bigoplus_{\beta\in Q_{++}}\D^\b\Coh([\frakg_r/G_r^c])_{\Lambda_\beta}.$$
The functor $i_*$ in \eqref{diag5}  yields a monoidal functor
$\D^\perf\Coh([\calN/G^c])_{\Lambda^+}\to\D^\b\Coh([\frakg/G^c])_{\Lambda^+}.$ 
The left adjoint functor $Li^*$ is not monoidal.
\end{remark}

\smallskip

\subsection{The monoidal equivalence}\label{sec:mon3}

The category
$\PCoh([\calN_r/G_r^c])_{\Lambda_\beta}$ is graded properly stratified.
The graded Abelian categories $\calD_\beta^{\,\sharp}$, $\calD_\beta$
are equipped with the pair of adjoint functors $(\xi^*\,,\,\xi_*)$.
By Proposition \ref{prop:monoidal} and the definition of the proper standard and costandard modules
$\bar\Delta(\pi)$, $\bar\nabla(\pi)$ in $\calD_\beta$ in  \cite[(6.5), (6.7)]{K15c}, we deduce that
$\bar\Delta(\pi), \bar\nabla(\pi)\in\calD_\beta^{\,\sharp}$ for all $\pi\in\Gamma\!_\beta$.
We may write
$\bar\Delta(\pi)^\sharp=\bar\Delta(\pi)$ and $\bar\nabla(\pi)^\sharp=\bar\nabla(\pi)$
to indicate that we view them as objects in $\calD_\beta^{\,\sharp}$.
In \eqref{Dsharp} we have defined
$\Delta(\pi)^\sharp=\xi^*\Delta(\pi)$ for all $\pi\in\Gamma\!_\beta.$
Hence, we can consider the subcategory $\calD_\beta^{\,\sharp,\,\Delta}\subset\calD_\beta^{\,\sharp}$ 
consisting of the objects with a  
finite filtration whose subquotients are isomorphic to some $\Delta(\pi)^\sharp$ with $\pi\in\Gamma\!_\beta$.

\smallskip

\begin{lemma}\label{lem:DP}
\hfill
\begin{itemize}[leftmargin=8mm]
\item[$\mathrm{(a)}$]  $\xi^*$ restricts to an exact functor
$\calD_\beta^{\,\Delta}\to\calD_\beta^{\,\sharp,\,\Delta}.$
\item[$\mathrm{(b)}$] 
$\calD_\beta^{\,\sharp,\,\proj}=
\{P\in\calD_\beta^{\,\sharp,\,\Delta}\,;\,\bbExt^1_{\calD_\beta^{\,\sharp}}(P\,,\,\Delta(\pi)^\sharp)=0
\,,\,\forall\pi\in\Gamma\!_\beta\}$.
\item[$\mathrm{(c)}$] 
$\PCoh([\calN_r/G_r^c])_{\Lambda_\beta}^\proj=
\{P\in\PCoh([\calN_r/G_r^c])_{\Lambda_\beta}^\Delta\,;\,
\bbExt^1_{\PCoh([\calN_r/G_r^c])}(P\,,\,\pmb\Delta(\lambda)^\sharp)=0\,,\,
\forall\lambda\in\Lambda_\beta\}$.
\end{itemize}
\end{lemma}

\begin{proof}
Fix  $\pi=((\beta_l)^{p_l},\dots,(\beta_0)^{p_0})$ in $\Gamma\!_\beta$.
The map \eqref{INC1} yields an inclusion
\begin{align*}
Z(\bfR_\beta)\subset Z(\bfR_{\beta_l})^{\otimes p_l}\otimes\dots\otimes Z(\bfR_{\beta_0})^{\otimes p_0}.
\end{align*}
Since $\pi$ is a Kostant partition of $\beta$, we have $r_\pi=r$, hence the central elements 
$z_{\beta_1},\dots,z_{\beta_l}$ in \eqref{CENT3} yield a $\k$-algebra embedding
\begin{align}\label{INC4}
\k[z_1,z_2,\dots,z_r]\subset Z(\bfR_{\beta_l})^{\otimes p_l}\otimes\dots\otimes Z(\bfR_{\beta_0})^{\otimes p_0}
\end{align}
which fits into the following commutative diagram of inclusions
$$\xymatrix{
Z(\bfR_\beta)\ar[r]^-{\eqref{INC1}}&Z(\bfR_{\beta_l})^{\otimes p_l}\otimes\dots\otimes(\bfR_{\beta_0})^{\otimes p_0}\\
\k[z_1,z_2,\dots,z_r]^{S_r}\ar[u]\ar[r]&\k[z_1,z_2,\dots,z_r]\ar[u]_-{\eqref{INC4}}
}$$
Now, for each $k=1,\dots,r$, the map 
$H^\bullet_{G_{\beta_k}}\to Z(\bfS_{\beta_k})$ in \eqref{CENT1} is the composition of the map 
$H^\bullet_{G_{\beta_k}}=\k[z_{\beta_k}]\to Z(\bfR_{\beta_k})$ in \eqref{CENT3} with the restriction 
$(j_{\beta_k})^*:Z(\bfR_{\beta_k})\to Z(\bfS_{\beta_k})$.
Hence, the proof of \eqref{ID3} implies that the map 
$H^\bullet_{G_{\beta}}\to Z(\bfS_{\beta})$
 in \eqref{CENT1} is the composition of
the chain of map 
$$\xymatrix{
H^\bullet_{G_{\beta}}\ar@{=}[r]&\k[z_1,z_2,\dots,z_r]^{S_r}\ar[r]&Z(\bfR_\beta)\ar[r]^{(j_\beta)^*}&Z(\bfS_\beta)
}.$$
Since $\Delta(\beta_k)$ is a free module over $\k[z_{\beta_k}]$ for each $k$,
the induced module
\begin{align}\label{IND3}
\Delta(\beta_l)^{\circ\, p_l}\circ\dots\circ
\Delta(\beta_1)^{\circ\, p_1}\circ\Delta(\beta_0)^{\circ\, p_0}
\end{align}
is free over $\k[z_1,z_2,\dots,z_r]$, hence over $H^\bullet_{G_r}$.
By \eqref{Delta0}, \eqref{Delta}, the standard module $\Delta(\pi)$ is a direct summand of 
the induced module \eqref{IND3}, 
hence it is also free over $H^\bullet_{G_r}$.
Part (a) follows, because 
$\xi^*(M)=\k\otimes_{H^\bullet_{G_r}}\!M$ for each module $M\in\calD_\beta$.

\smallskip

Now, let us concentrate on part (b).
Since $\xi^*$ is left adjoint to $\xi_*$, which is exact, we deduce that 
$\xi^*(\calD_\beta^{\,\proj})\subset\calD_\beta^{\,\sharp,\,\proj}$.
Using the fact that the categories $\calD_\beta$, $\calD^{\,\sharp}_\beta$ are both Krull-Schmidt,
we get indeed an equality
$\xi^*(\calD_\beta^{\,\proj})=\calD_\beta^{\,\sharp,\,\proj}.$
Since the category $\calD_\beta$ is affine highest weight, we have
$$\calD_\beta^{\,\proj}=
\{P\in\calD_\beta^{\,\Delta}\,;\,\bbExt^1_{\calD_\beta}(P\,,\,\Delta(\pi))=0
\,,\,\forall\pi\in\Gamma\!_\beta\},$$
hence part (a) implies that 
$$\calD_\beta^{\,\sharp,\,\proj}\subset
\{P\in\calD_\beta^{\,\sharp,\,\Delta}\,;\,
\bbExt^1_{\calD_\beta^{\,\sharp}}(P\,,\,\Delta(\pi)^\sharp)=0
\,,\,\forall\pi\in\Gamma\!_\beta\}.$$
To prove the reverse inclusion, by part (a) and adjunction we have, for all $i>0$,
\begin{align*}
\bbExt^i_{\calD_\beta^{\,\sharp}}(\Delta(\sigma)^\sharp\,,\,\bar\nabla(\pi)^\sharp)
&=\bbExt^i_{\D^\b(\calD_\beta^{\,\sharp})}(L\xi^*\Delta(\sigma)\,,\,\bar\nabla(\pi)^\sharp)\\
&=\bbExt^i_{\D^\b(\calD_\beta)}(\Delta(\sigma)\,,\,\bar\nabla(\pi))\\
&=0,
\end{align*}
from which we deduce that
$$\calD_\beta^{\,\sharp,\,\Delta}\subset
\{P\in\calD_\beta^{\,\sharp}\,;\,
\bbExt^{>0}_{\calD_\beta^{\,\sharp}}(P\,,\,\bar\nabla(\pi)^\sharp)=0
\,,\,\forall\pi\in\Gamma\!_\beta\}.$$
The reverse inclusion can be proved as in \cite[lem.~7.8]{K15b}, using the vanishing
$$\bbExt^2_{\calD_\beta^{\,\sharp}}(\Delta(\sigma)^\sharp\,,\,\bar\nabla(\pi)^\sharp)=0$$
proved above.
Now, given an object
$P\in\calD_\beta^{\,\sharp,\,\Delta}$
such that
\begin{align}\label{HYP}\bbExt^1_{\calD_\beta^{\,\sharp}}(P\,,\,\Delta(\pi)^\sharp)=0
\quad,\quad\pi\in\Gamma\!_\beta,
\quad,\quad\forall a\in\bbZ,
\end{align}
there is a projective module $Q\in\calD_\beta^{\,\sharp,\,\proj}$ mapping onto $P$, because 
$\calD_\beta^{\,\sharp}$ is a \emph{Noetherian Laurentian category} in the terminology of
\cite[\S 2]{K15b}, and we can consider the short exact sequence
\begin{align}\label{EX2}\xymatrix{0\ar[r]&K\ar[r]&Q\ar[r]&P\ar[r]&0}.\end{align}
We have $K\in\calD_\beta^{\,\sharp,\,\Delta}$, because
$$\calD_\beta^{\,\sharp,\,\Delta}=
\{M\in\calD_\beta^{\,\sharp}\,;\,
\bbExt^{>0}_{\calD_\beta^{\,\sharp}}(M\,,\,\bar\nabla(\pi)^\sharp)=0
\,,\,\forall\pi\in\Gamma\!_\beta\}.$$
Hence \eqref{HYP} implies that the short exact sequence \eqref{EX2} splits, so $P$ is projective.

\smallskip

Part (c) is a well-known property of graded properly stratified categories.
\end{proof}

\smallskip

We can now prove the following theorem, which is the main result of this paper.

\smallskip

\begin{theorem}\label{thm:2}
For each $\beta\in Q_{++}$ there is an equivalence of graded Abelian categories
$$\E^\sharp_\beta:\calD_\beta^{\,\sharp}\to\PCoh([\calN_r/G_r^c])_{\Lambda_\beta}$$
which takes  the graded $\bfS^\sharp_\beta$-modules
$\Delta(\pi)^\sharp$, 
$\bar\Delta(\pi)^\sharp$, 
$\bar\nabla(\pi)^\sharp$, 
$L(\pi)^\sharp$
to the perverse coherent sheaves
$\pmb\Delta(\lambda_\pi)^\sharp$,
$\pmb{\bar\Delta}(\lambda_\pi)^\sharp$,
$\pmb{\bar\nabla}(\lambda_\pi)^\sharp$,
$\bfL(\lambda_\pi)^\sharp$ respectively,
for all $\pi\in\Gamma\!_\beta$.
\end{theorem}

\begin{proof} Given subsets $A,B\subset\Isom(\calD)$ of the set of isomorphism classes of objects
of a triangulated category $\calD$, let $A* B\subset\Isom(\calD)$ be the set of classes of all objects $Z$ 
for which there exists an exact triangle $X\to Z\to Y\to X[1]$ with $X\in A$ and $Y\in B$.
The octahedron axiom implies that the operation $*$ is associative, hence for each $A$ as above we may consider 
the strictly full subcategory $\langle A\rangle\subset\calD$ such that
$\Isom(\langle A\rangle)=\bigcup_{n>0}A*A*\cdots *A$, where $A$ appears $n$ times.
By \eqref{TRIA1}, we may define the additive subcategories
\begin{align*}
\D^\perf(\calD_\beta^{\,\sharp})^\Delta&=
\langle\{\Delta(\pi)^\sharp\langle a\rangle\,;\,\pi\in\Gamma\!_\beta\,,\,a\in\bbZ\}\rangle,\\
\D^\perf\Coh([\calN_r/G_r^c])_{\Lambda_\beta}^\Delta&=
\langle\{\pmb\Delta(\lambda)^\sharp\langle a\rangle\,;\,\lambda\in\Lambda_\beta\,,\,a\in\bbZ\}\rangle
\end{align*}
of $\D^\perf(\calD_\beta^{\,\sharp})$ and
$\D^\perf\Coh([\calN_r/G_r^c])_{\Lambda_\beta}$ respectively.
The equivalence  of triangulated categories
$\E^\sharp_\beta$ in Theorem \ref{thm:final1} restricts to an equivalence of graded additive categories
$$\big(\D^\perf(\calD_\beta^{\,\sharp})^\Delta\,,\,\langle\bullet\rangle\big)\to
\big(\D^\perf\Coh([\calN_r/G_r^c])_{\Lambda_\beta}^\Delta\,,\,\langle\bullet\rangle\big).$$
Since we have
$\D^\perf(\calD_\beta^{\,\sharp})^\Delta=\calD_\beta^{\,\sharp,\,\Delta}$
and
$\D^\perf\Coh([\calN_r/G_r^c])_{\Lambda_\beta}^\Delta=
\PCoh([\calN_r/G_r^c])_{\Lambda_\beta}^\Delta,$
it is indeed an equivalence of graded additive categories
$$\E_\beta^\sharp:\big(\calD_\beta^{\,\sharp,\,\Delta}\,,\,\langle\bullet\rangle\big)\to
\big(\PCoh([\calN_r/G_r^c])_{\Lambda_\beta}^\Delta\,,\,\langle\bullet\rangle\big).$$
By Lemma \ref{lem:DP}, it
restricts further to an equivalence of graded additive categories
$$\big(\calD_\beta^{\,\sharp,\,\proj}\,,\,\langle\bullet\rangle\big)\to
\big(\PCoh([\calN_r/G_r^c])_{\Lambda_\beta}^\proj\,,\,\langle\bullet\rangle\big)$$
which takes the projective cover of $L(\pi)^\sharp$ to the projective cover of $\bfL(\lambda_\pi)^\sharp$.
Therefore, it yields an equivalence of graded Abelian categories
$\calD_\beta^{\,\sharp}\to\PCoh([\calN_r/G_r^c])_{\Lambda_\beta}$
as in the theorem.
\end{proof}

\smallskip

Let $\circ^\op$ denote monoidal structure opposite to $\circ$.
Set $\E^\sharp=\bigoplus_{\beta\in Q_{++}}\E^\sharp_\beta$ .

\smallskip

\begin{conjecture}\label{conj:monoidal}
The equivalence of graded Abelian categories 
$\E^\sharp:\calD^{\,\sharp}\to\PCoh([\calN/G^c])_{\Lambda^+}$
extends to an equivalence of exact graded monoidal categories
$(\calD^{\,\sharp}\,,\,\circ)\to(\PCoh([\calN/G^c])_{\Lambda^+}\,,\,\circ^\op)$.
\end{conjecture}

\smallskip

\begin{remark}\label{rem:rem-conj}
\hfill
\begin{itemize}[leftmargin=8mm]
\item[$\mathrm{(a)}$]
From \eqref{Delta}, Theorem \ref{thm:2} and the conjecture we deduce that
\begin{align*}
\pmb{\bar\Delta}(\pi)^\sharp=\E^\sharp(\bar\Delta(\pi)^\sharp)
=(\bfL(\beta_0)^\sharp)^{\circ p_0}\circ\dots\circ(\bfL(\beta_l)^\sharp)^{\circ p_l}\,
\langle\sum_{k=0}^lp_k(p_k-1)/2\rangle
\end{align*}
for any Kostant partition $\pi=((\beta_l)^{p_l},\dots,(\beta_1)^{p_1},
(\beta_0)^{p_0})$ in $\Gamma\!_\beta$.
This is precisely the formula \eqref{PS}.
So, the isomorphism $\pmb{\bar\Delta}(\pi)^\sharp=\E^\sharp(\bar\Delta(\pi)^\sharp)$
is a weak form of monoidality of the functor $\E^\sharp$.

\item[$\mathrm{(b)}$]
By \cite[\S 24]{Mc18}, \cite{VV11} the $\calA$-module isomorphism $\G_0(\calC^\fl)=\bfA_q(\frakn)$ 
maps the proper standard modules and the simple 
ones to the elements of the dual PBW basis and the dual canonical basis respectively.
More precisely, if $\pi=((\beta_l)^{p_l},\dots,(\beta_1)^{p_1},(\beta_0)^{p_0})$
it maps the graded module $\bar\Delta(\pi)$ to the element
$$E^*(\pi)=q^{\sum_{k=0}^lp_k(p_k-1)/2}E^*(\beta_l)^{p_l}\cdots E^*(\beta_2)^{p_2}E^*(\beta_1)^{p_1}.$$

\item[$\mathrm{(c)}$]
We expect that the functor 
$\E=\bigoplus_{\beta\in Q_{++}}\E_\beta$
in Conjecture \ref{conj:1}
extends also to an equivalence of exact graded monoidal categories
$(\calD\,,\,\circ)\to(\PCoh([\frakg/G^c])_{\Lambda^+}\,,\,\circ^\op).$
\end{itemize}
\end{remark}

\smallskip

\subsection{Perverse coherent sheaves on the affine Grassmannian}

Fix a positive integer $N$. Consider the element
$w=(s_0s_1)^N$ in the Weyl group of the affine root system $\Phi$, hence
$$\Phi_+\cap w(-\Phi_+)=\{\beta_0,\beta_1,\dots,\beta_{2N-1}\}.$$
Let $\calC_w^\fl$ be the graded monoidal full subcategory of $\calC^\fl$ generated by 
$L(\beta_0), L(\beta_1),\dots,L(\beta_{2N-1}).$
The \emph{quantum unipotent coordinate algebra} $\bfA_q(\frakn(w))$ 
is the $\calA$-subalgebra of $\bfA_q(\frakn)$
generated by the dual root vectors $E^*(\beta_0), E^*(\beta_1),\dots, E^*(\beta_{2N-1})$.
The graded monoidal category $\calC^\fl_w$ 
is a categorification of  $\bfA_q(\frakn(w))$. 
The \emph{quantum open unipotent cell} $\bfA_q(\frakn^w)$ is a localization of  $\bfA_q(\frakn(w))$.
In \cite{KKOP19} a localization $\widetilde\calC_w^\fl$ of the category $\calC^\fl_w$ is introduced. 
It is a graded monoidal category.
It is proved there that $\widetilde\calC_w^\fl$ categorifies $\bfA_q(\frakn^w)$.

\smallskip

From now on, let $\Gr$ denote the affine Grassmannian of $G_N$ over $\k$.
We define
$\calO=\k[[t]]$ and $G_N^c(\calO)=GL_N(\calO)\rtimes\bbG_{m,\k}.$
The multiplication in $G_N^c(\calO)$ is given by
$$(1,z)\cdot(g(t),1)=(g(z^4t),1).$$
The group $G_N^c(\calO)$ acts on $\Gr$ so that the subgroup $\bbG_{m,\k}$ is the 4-fold cover of the loop rotation.
Let $\D^\b\Coh([\Gr/G_N^c(\calO)])$ be the derived category of 
$G_N^c(\calO)$-equivariant coherent sheaves on $\Gr$. It is a ${1\over 2}\bbZ$-graded triangulated category
with the grading shift functors
such that $\langle -1/2\rangle$ is the tensor product
with the tautological 1-dimensional $\bbG_m$-module (of weight 1).
In other words 
$q=\langle 1\rangle$ corresponds to shifting the weight by -1 with respect to the double cover of loop rotation.
Let $\PCoh([\Gr/G_N^c(\calO)])$ be the heart of the perverse t-structure
in $\D^\b\Coh([\Gr/G_N^c(\calO)])$.
We equip both categories with the convolution product $\circledast$ denoted by the symbol $*$ in 
\cite{BFM05} or \cite{CW18}.
This yields an exact graded monoidal category
\begin{align}\label{mon5}\big(\PCoh([\Gr/G_N^c(\calO)])\,,\,\circledast\big).\end{align}

\smallskip

For each integer $r\geqslant 0$, we set
$\alpha_r=N\alpha_0+r\delta$ and 
$\Gr_{\alpha_r}=\overline{\Gr_{r\omega_1}}$. 
Consider the $G_r^c$-invariant open subset $\calN_{\alpha_r}\subset\calN_r$ 
consisting of all nilpotent $r$ by $r$ matrices with at most $N$ Jordan blocks. 
We define the graded categories
\begin{align*}
\PCoh([\Gr_+/G_N^c(\calO)])&=\bigoplus_{r>0}\PCoh([\Gr_{\alpha_r}/G_N^c(\calO)]),\\
\PCoh([\calN/G^c])&=\bigoplus_{r>0}\PCoh([\calN_r/G_r^c]).
\end{align*}
The graded shift functors are defined above and in \S\ref{sec:PCoh}.
The $G_N(\calO)$-orbit $\Gr_{\omega_r}$ in $\Gr$ is
$$\Gr_{\omega_r}=\{L\subset L_0\,;\,tL_0\subset L\,,\,\dim(L_0/L)=r\}.$$
Its dimension is
$d_{\omega_r}=r(N-r).$
Let $\det_{\omega_r}=\bigwedge^r(L_0/L)$ be the determinant bundle over $\Gr_{\omega_r}$.
Let $i_{\omega_r}$ be the embedding of $\Gr_{\omega_r}$ in $\Gr$.
Consider the object of  $\PCoh([\Gr/G_N^c(\calO)])$ given by
$$\calP_{r,\ell}=(i_{\omega_r})_*(\text{det}_{\omega_r})^{\otimes\ell}
[d_{\omega_r}/2]\langle -d_{\omega_r}/2-r\ell\rangle\quad,\quad
\forall r\in[1,N]\quad,\quad\forall \ell\in\bbZ.$$

\smallskip

\begin{proposition}[\cite{FF18}]
\hfill
\begin{itemize}[leftmargin=8mm]
\item[$\mathrm{(a)}$]
There is a flat stack homomorphism
$\psi_r:[\Gr_{\alpha_r}/G_N^c(\calO)]\to[\calN_{\alpha_r}/G_r^c]$.
\item[$\mathrm{(b)}$]
Set $d_r=r(N-1)/2$. The triangulated functor
$$\Psi_r=\psi_r^*[d_r]\langle d_r\rangle:\D^\b\Coh([\calN_{\alpha_r}/G_r^c])
\to\D^\b\Coh([\Gr_{\alpha_r}/G_N^c(\calO)])$$
is graded and $t$-exact with respect to the perverse $t$-structures of both sides
\item[$\mathrm{(c)}$]
Composing $\Psi_r$ with the restriction to the open subset $\calN_{\alpha_r}\subset\calN_r$ 
and taking the sum over all $r>0$, we get a graded monoidal functor
$\Psi_+$ yielding the following commutative diagram
of Abelian graded monoidal categories
\begin{align*}
\xymatrix{
\PCoh([\calN/G^c])_{\Lambda^+}\ar[d]^-{\Psi}
\ar@{^{(}->}[r]^-i&\PCoh([\calN/G^c])\ar[d]^-{\Psi_+}\\
\PCoh([\Gr/G_N^c(\calO)])&\ar@{_{(}->}[l]_-j
\PCoh([\Gr_+/G_N^c(\calO)])
}
\end{align*}
such that,  for each $k$ in $[0,2N-1]\subset\bbN=\Lambda^+_1$, we have
$\Psi\E^\sharp(L(\beta_k)^\sharp)=\Psi(\bfL(k)^\sharp)=\calP_{1,N-k}\langle-1/2\rangle.$
\end{itemize}
\end{proposition}

\begin{proof} Part (a) is \cite[lem.~4.9]{FF18}.
Part (b) is \cite[prop.~4.12]{FF18}. 
Part (c) is proved as \cite[lem.~4.15]{FF18}.
The functors $i$, $j$ are the obvious inclusions. 
The functor $i$ is a monoidal functor, relatively to the monoidal structure $\circ$.
The functor $j$ is a monoidal functor, relatively to the monoidal structure $\circledast$.
\end{proof}

\smallskip

We equip the $\calA$-modules $\G_0(\calD^\sharp)$ and $\G_0(\calC^\fl)$ with the multiplications given by 
the exact bifunctor $\circ.$
Let $\bfA_q(\frakn)_{\Lambda^+}$ be the $\calA$-subalgebra of $\bfA_q(\frakn)$ generated
by the dual root vectors $E(\beta)^*$ with $\beta\in\Phi_{++}$.
We have 
$\G_0(\calC^\fl)=\bfA_q(\frakn)$ and $\G_0(\calD^\sharp)=\bfA_q(\frakn)_{\Lambda^+}.$
Consider the graded monoidal Serre subcategory of $\calD^\sharp$ given by
$\calD^\sharp_w=\calC^\fl_w\cap\calD^\sharp$.
We have
$\G_0(\calD^\sharp_w)=\bfA_q(\frakn(w)).$
The equivalence $\E^\sharp:\calD^{\,\sharp}\to\PCoh([\calN/G^c])_{\Lambda^+}$
 gives an $\calA$-module isomorphism
$$\bfA_q(\frakn)_{\Lambda^+}\to \G_0(\PCoh([\calN/G^c])_{\Lambda^+}).$$ 
This isomorphism identifies the elements of the dual PBW basis and the dual canonical basis of 
$\bfA_q(\frakn)_{\Lambda^+}$ with the
classes of the proper standard modules and the simple ones in $\PCoh([\calN/G^c])_{\Lambda^+}$
respectively.

\smallskip

Let $\D(w\Lambda_i,\Lambda_i)$, $i=0,1$, be the \emph{unipotent quantum minors}
in $\bfA_q(\frakn)$. See, e.g., \cite[def.~1.5]{KKOP18} for the notation.
They belong to $\bfA_q(\frakn(w))$.
By \cite[\S 4]{KKOP18},  there are unique simple self-dual objects 
$\M(w\Lambda_i,\Lambda_i)$, $i=0,1$ in $\calC^\fl_w$ such that
the isomorphism $\G_0(\calC^\fl_w)=\bfA_q(\frakn(w))$ identifies the class
of $\M(w\Lambda_i,\Lambda_i)$ with $\D_i$. 
We abbreviate $\C_i=\M(w\Lambda_i,\Lambda_i)$.
With the terminology in \cite{KKOP19}, the objects $\C_0$, $\C_1$ are 
\emph{non-degenerate graded braiders}. 
Let $\widetilde\calC_w^\fl$ be  the localization
of the graded monoidal category $\calC_w^\fl$ by $\C_0$, $\C_1$. 
It is an exact graded monoidal category. We abbreviate
$\widetilde\calC_w^\fl=\calC_w^\fl\,[(\C_0)^{\circ\,-1}\,,\,(\C_1)^{\circ\,-1}].$
Since the modules $\C_0$, $\C_1$ are simple and $\G_0(\calD^\sharp_w)=\G_0(\calC^\fl_w)$, 
they belong to the subcategory $\calD^\sharp_w$.
Let $\widetilde\calD_w^\sharp$ be the localization
of the graded monoidal category $\calD_w^\sharp$ by $\C_0$, $\C_1$, i.e., we set
$$\widetilde\calD_w^\sharp=\calD_w^\sharp[(\C_0)^{\circ\,-1}\,,\,(\C_1)^{\circ\,-1}].$$ 
It is an exact graded monoidal category.

\smallskip

Composing the functors $\E^\sharp$, $\Psi$ and the inclusion
$\calD^\sharp_w\subset\calD^\sharp$,
we get a graded functor
$$\Psi_w:\calD^\sharp_w\to\PCoh([\Gr/G_N^c(\calO)]).$$
By \cite{CW18}, we have
$\Psi_w(\C_i)=\calP_{N,1-i}$ for $i=0,1.$
In particular, both objects $\Psi_w(\C_0)$ and $\Psi_w(\C_1)$ are invertible in the monoidal category
\eqref{mon5}.
Now, \emph{assume that Conjecture $\ref{conj:monoidal}$ holds}.
Then, the functor $\Psi_w$ is graded monoidal.
Thus, the universal property in \cite[thm.~2.7]{KKOP19} yields a commutative triangle of exact monoidal functors
$$\xymatrix{
\calD^\sharp_w\ar[r]^-\Upsilon\ar[dr]_-{\Psi_w}&\widetilde\calD_w^\sharp\ar[d]^-{\widetilde\Psi_w}\\
&\PCoh([\Gr/G_N^c(\calO)])
}$$

\smallskip

\begin{conjecture}\label{conj:D}
The functor $\widetilde\Psi_w$ is a graded monoidal equivalence
$$\big(\widetilde\calD^\sharp_w\,,\,\circ\big)\to\big(\PCoh([\Gr/G_N^c(\calO)])\,,\,\circledast\big).$$
\end{conjecture}

The functor $\widetilde\Psi_w$ is exact by definition.
Since $\G_0(\calD^\sharp_w)=\G_0(\calC^\fl_w)$, we have
$\G_0(\widetilde\calD^\sharp_w)=\G_0(\widetilde\calC^\fl_w)=\bfA_q(\frakn^w).$
By \cite{CW18} we have
$\bfA_q(\frakn^w)=\G_0(\PCoh([\Gr/G_N^c(\calO)])).$
Hence, taking the Grothendieck groups,  the functor $\widetilde\Psi_w$ yields a group homomorphism
$$\G_0(\widetilde\calD^\sharp_w)\to \G_0(\PCoh([\Gr/G_N^c(\calO)]))$$
which is invertible.
Now, recall that if $F:\calA\to\calB$ is an exact functor of Abelian Artinian categories such that
the induced map $\G_0(\calA)\to \G_0(\calB)$ is injective, then $F$ is faithful.
We deduce that $\widetilde\Psi_w$ is a faithful functor.
To prove the conjecture we must prove that the functor $\widetilde\Psi_w$
is full.

\medskip

\appendix

\medskip

\section{Reminders on Artin stacks and mixed geometry} \label{sec:ASMG}

\subsection{Schemes and stacks}
Let $\k=\overline\bbQ_\ell$ and $F$ be any field of characteristic
prime to $\ell$.
All $F$-schemes are assumed to be separated.
An \emph{Artin $F$-stack} is a stack over the category of affine $F$-schemes
with a smooth atlas and a representable, quasi-compact, quasi-separated diagonal.
Given any algebraic $F$-space $Z$ with an action of an affine algebraic 
$F$-group $G$, the  \emph{quotient $F$-stack} $[Z/G]$ is the Artin stack whose set of $R$-points
consists of $G$-torsors on $\Spec R$ with $G$-equivariant map to $Z$.
A \emph{locally quotient stack} is an Artin stack which is locally equivalent to a quotient stack.

\smallskip

Let $\calZ$ be an Artin $F$-stack of finite type. Let $\D^\b(\calZ)$ be
the bounded constructible derived category of \'etale sheaves of $\k$-modules
on $\calZ$, in the sense of \cite{LO}.
Let $\P(\calZ)$ be the subcategory of perverse sheaves.
For each complex $\calE\in\D^b(\calZ)$ and each integer $a$, 
let ${}^p\!H^a\calE\in\P(\calZ)$ be the $a$-th perverse cohomology complex.
Let $\C(\calZ)\subset\D^\b(\calZ)$ be the additive full subcategory of semisimple complexes.

\smallskip

Let $Z$ be an $F$-scheme of finite type with an action of an affine algebraic $F$-group $G$.
Let $\D^\b_G(Z)$ denote the $G$-equivariant derived category of complexes of \'etale sheaves
of $\k$-vector spaces on $Z$ with bounded constructible cohomology, in the sense of Bernstein-Lunts \cite{BL}.
The triangulated category $\D^\b_G(Z)$
depends only on the quotient stack $[Z\,/\,G]$, and not on the action groupoid $\{G\times Z\rightrightarrows Z\}$
representing it.
It is equivalent to the triangulated category  $\D^\b([Z\,/\,G])$.

\smallskip

By a stratification $S=\{Z_w\,;\,w\in W\}$ of  an $F$-scheme of finite type $Z$  we'll  mean  
a finite algebraic Whitney  stratification.
Let $\D^\b(Z,S)\subset\D^\b(Z)$ be the full subcategory whose objects have constructible cohomology with 
respect to $S$. We write 
$\P(Z,S)=\P(Z)\cap\D^\b(Z,S)$ and $\Par(Z,S)=\Par(Z)\cap\D^\b(Z,S).$ 

\smallskip

Unless specified otherwise, we'll assume that $\calZ$ is the quotient stack $[Z\,/\,G]$
of an $F$-scheme $Z$ of finite type by the action of an affine algebraic $F$-group $G$
with a finite number of orbits.
Then, the $G$-orbits define a stratification $S$ of $Z$ and
the categories  $\P(\calZ)$, $\C(\calZ)$ are $G$-equivariant analogues of the categories 
$\P(Z,S)$, $\C(Z,S)$.
Let 
$\For:\D^\b(\calZ)\to\D^\b(Z,S)$
be the forgetful functor.
The categories $\C(\calZ)$ and $\C(Z,S)$ are graded.
The category grading is given by the cohomological shift functor, i.e., we set
\begin{align}\label{SHIFT0}\langle\bullet\rangle=[\bullet].\end{align}

\smallskip

\subsection{Mixed complexes}\label{ss:mixcmplx}
From now on we assume that $F$ is the algebraic closure of a finite field $F_0$ of characteristic prime to $\ell$.
We'll use the following convention : objects over $F_0$ are denoted with a subscript 0, 
and suppression of the subscript means passing to $F$ by extension of scalars.
For instance, we may write $Z=F\otimes_{F_0}Z_0$ for some $F_0$-scheme $Z_0$.
Then, we'll assume that
the stratification $S$ of $Z$ is the extension of scalars of a stratification $\{Z_{w,0}\,;\,w\in W\}$ of $Z_0$.
Let $\D^\b_\bbd(Z_0)$ be the full triangulated subcategory of mixed complexes in $\D^\b(Z_0)$.
The extension of scalars yields a triangulated $t$-exact functor
$$\omega:\D^\b_\bbd(Z_0)\to\D^\b(Z).$$
We may write a mixed complex with a subscript $\bullet_\bbd$, 
and abbreviate $\calE=\omega(\calE_\bbd)$.
Let $\D^\b_\bbd(Z_0,S)$ be the full triangulated subcategory of  $\D^\b_\bbd(Z_0)$
such that $\D^\b_\bbd(Z_0,S)=\omega^{-1}\D^\b(Z_0,S).$
Let $\P\!_\bbd(Z_0,S)$ be the  category of mixed perverse sheaves in $\D^\b_\bbd(Z_0,S)$.
Recall that $\calZ$ is an $F$-stack of the form $\calZ=[Z\,/\,G]$. Assume further that 
$\calZ$ is isomorphic to $F\otimes_{F_0}\calZ_0$, with $\calZ_0=[Z_0\,/\,G_0]$
and some affine algebraic $F_0$-group $G_0$ such that $G=F\otimes_{F_0}G_0$. 
Let $\D^\b_\bbd(\calZ_0)$ be the full triangulated subcategory of mixed complexes in $\D^\b(\calZ_0)$.
See \cite{LO} and \cite{S12} for the definition and the basic properties of the category $\D^\b_\bbd(\calZ_0)$.
Let 
$\For:\D^\b_\bbd(\calZ_0)\to\D^\b_\bbd(Z_0,S)$
be the forgetful functor.
For each integer $w$, we consider the full subcategories 
$\D^\b_{\leqslant w}(\calZ_0)$, $\D^\b_{\geqslant w}(\calZ_0)$
of $\D^\b_\bbd(\calZ_0)$ 
consisting of the mixed complexes of weight $\leqslant w$ and $\geqslant w$.
The category of pure complexes of weight $w$ is
$$\D^\b_w(\calZ_0)=\D^\b_{\leqslant w}(\calZ_0)\cap\D^\b_{\geqslant w}(\calZ_0).$$
By \cite[prop.~5.1.15]{BBD}, we have
\begin{align}\label{weight}\Hom_{\D^\b(\calZ_0)}(\calE\,,\,\calF)=0
\quad,\quad
\forall\calE\in\D^\b_{<w}(\calZ_0)
\quad,\quad
\forall\calF\in\D^\b_{>w}(\calZ_0).
\end{align}
We also have the following refinement of \eqref{weight} 
\begin{align}\omega(f)=0
\quad,\quad
\forall f\in\Hom_{\D^\b(\calZ_0)}(\calE,\calF)
\quad,\quad
\forall\calE\in\D^\b_{\leqslant w}(\calZ_0)
\quad,\quad
\forall\calF\in\D^\b_{>w}(\calZ_0).
\end{align}
Let $\P\!_\bbd(\calZ_0)$ be the category of mixed perverse sheaves in $\D^\b_\bbd(\calZ_0)$.
A mixed complex $\calE$ is pure of weight $w$, $\leqslant w$ or $\geqslant w$
if and only if  the perverse sheaf ${}^p\!H^a\calE$ is pure of weight $w+a$, $\leqslant w+a$ or $\geqslant w+a$ 
for each $a\in\bbZ$, by \cite[thm.~5.4.1]{BBD}. 
A mixed perverse sheaf $\calE$ has a unique finite increasing weight filtration
$W_a\calE$, $a\in\bbZ$,
such that the subquotient
$\Gr^W_a\calE=W_a\calE/W_{a-1}\calE$ is a pure mixed perverse sheaf of weight $a$,
which may not be semisimple, see \cite[thm.~5.3.5]{BBD}.

\smallskip

\begin{proposition}\label{prop:triangle}
For each $\calE\in\D^\b_\bbd(\calZ_0)$ and $w\in\bbZ$,
there is a distinguished triangle
\begin{align}\label{DT1}
\xymatrix{\ar[r]&\calE_{\leqslant w}\ar[r]&\calE\ar[r]&\calE_{>w}}
\quad,\quad\calE_{\leqslant w}\in\D^\b_{\leqslant w}(\calZ_0)
\quad,\quad\calE_{>w}\in\D^\b_{>w}(\calZ_0)
\end{align}
such that $\calE_{\geqslant w}=0$, $\calE_{\geqslant -w}=\calE$ if $w>\!\!> 0$, and
\begin{itemize}[leftmargin=8mm]
\item[$\mathrm{(a)}$]  
there is a distinguished triangle
\begin{align*}\xymatrix{\ar[r]&\calE_w\ar[r]&\calE_{\geqslant w}\ar[r]&\calE_{>w}}
\quad,\quad
\calE_w=(\calE_{\geqslant w})_{\leqslant w}\in\D^\b_w(\calZ_0),
\end{align*}
\item[$\mathrm{(b)}$]  the long exact sequence of perverse cohomologies splits into short exact sequences
\begin{align*}
\xymatrix{0\ar[r]&{}^p\!H^a\calE_{\leqslant w}\ar[r]&{}^p\!H^a\calE\ar[r]&{}^p\!H^a\calE_{>w}\ar[r]& 0}
\quad,\quad\forall a\in\bbZ.
\end{align*}
\end{itemize}
\end{proposition}

\begin{proof}
Any perverse sheaf in $\P_\bbd(\calZ_0)$ has a finite length.
The construction of $\calE_{\leqslant w}$, $\calE_{>w}$ is by induction on the \emph{total length} of $\calE$, i.e., 
on the sum of the lengths of the perverse sheaves ${}^p\!H^a\calE$,
following the lines of \cite[lem.~6.7]{AR13d}. 
Given $w$, let $a$ be the smallest integer such that the subobject $W_{w+a}({}^p\!H^a\calE)$ 
of ${}^p\!H^a\calE$ is $\neq 0$.
Set $\calG=W_{w+a}({}^p\!H^a\calE)[-a]$.
The inclusion $\calG\subset{}^p\!H^a\calE[-a]$ factors to a distinguished triangle
\begin{align}\label{DTA}\xymatrix{\ar[r]&\calG\ar[r]^-f&\calE\ar[r]&\calF},\end{align}
see \cite[(6.8)]{AR13d} for details, such that ${}^p\!H^b\calG=0$ for all $b\neq a$ and
\begin{align}\label{SESA}
\xymatrix{0\ar[r]&{}^p\!H^b\calG\ar[r]&{}^p\!H^b\calE\ar[r]&{}^p\!H^b\calF\ar[r]& 0}
\quad,\quad\forall b\in\bbZ.\end{align}
Hence $\calF$ has a lower total length than $\calE$, and induction yields a distinguished triangle 
\begin{align}\label{DTB}
\xymatrix{\ar[r]&\calF_{\leqslant w}\ar[r]&\calF\ar[r]&\calF_{>w}}
\quad,\quad\calF_{\leqslant w}\in\D^\b_{\leqslant w}(\calZ_0)
\quad,\quad\calF_{>w}\in\D^\b_{>w}(\calZ_0).
\end{align}
From \eqref{DTA}, \eqref{DTB} and \cite[lem.~1.3.10]{BBD}, we get distinguished triangles
\begin{align}\label{DTC}
\begin{split}
\xymatrix{\ar[r]&\calH\ar[r]&\calE\ar[r]&\calF_{>w}}
\quad,\quad
\xymatrix{\ar[r]&\calG\ar[r]&\calH\ar[r]&\calF_{\leqslant w}}
\quad,\quad\calH\in\D^\b_{\leqslant w}(\calZ_0).
\end{split}
\end{align}
Set $\calE_{\leqslant w}=\calH$ and $\calE_{>w}=\calF_{>w}$.
The induction hypothesis yields short exact sequences
\begin{align}\label{SESB}
\xymatrix{0\ar[r]&{}^p\!H^b\calF_{\leqslant w}\ar[r]&{}^p\!H^b\calF\ar[r]&{}^p\!H^b\calF_{>w}\ar[r]& 0}
\quad,\quad\forall b\in\bbZ.
\end{align}
From \eqref{SESA} and \eqref{SESB}, we deduce that the long exact sequence 
$$\xymatrix{\ar[r]&{}^p\!H^b\calH\ar[r]&{}^p\!H^b\calE\ar[r]&{}^p\!H^b\calF_{>w}\ar[r]& {}^p\!H^{b+1}\calH\ar[r]&}$$
splits into short exact sequences, yielding the condition (b).
Since $\Hom_{\D^\b(\calZ_0)}(\calE_{<w}\,,\,\calE_{> w})=0$ by \eqref{weight}, the
map $\calE\to\calE_{>w}$ factors to a morphism $\calE_{\geqslant w}\to\calE_{>w}$.
Completing this morphism to a distinguished triangle yields the claim (a).
\end{proof}

\smallskip


For any mixed complexes $\calE,$ $\calF$ on $\calZ_0$, we write
\begin{align*}
\Hom^\bullet_{\D^\b(\calZ_0)}(\calE,\calF)=\bigoplus_{a\in\bbZ}\Hom^a_{\D^\b(\calZ_0)}(\calE,\calF)[-a]
\quad,\quad
\Hom^a_{\D^\b(\calZ_0)}(\calE,\calF)=\Hom_{\D^\b(\calZ_0)}(\calE,\calF[a]).
\end{align*}
Let $\text{a}: \calZ_0\to \Spec F_0$ be the structure map. We define the \emph{geometric Hom functor} by
\begin{align*}
\HHom^\bullet_{\D^\b(\calZ_0)}(\calE,\calF)=\text{a}_*R\calH om_{\D^\b(\calZ_0)}(\calE,\calF).
\end{align*}
It is a mixed complex on $\Spec F_0$.
We define
\begin{align*}
\HHom^a_{\D^\b(\calZ_0)}(\calE,\calF)=H^a(\HHom^\bullet_{\D^\b(\calZ_0)}(\calE,\calF)).
\end{align*}
It is a mixed vector space consisting of a graded $\k$-vector space
$$\omega\,\HHom^a_{\D^\b(\calZ_0)}(\calE,\calF)=\Hom^a_{\D^\b(\calZ)}(\omega\calE,\omega\calF)$$
and a Frobenius operator $\Fr$. 
We abbreviate 
$$\underline H^\bullet(\calZ_0,\calE)=\HHom^\bullet_{\D^\b(\calZ_0)}(\k_{\calZ_0},\calE).$$

\smallskip

\subsection{Even stratifications, mixed categories and parity sheaves}\label{sec:even}

Let $\k_w$ denote the mixed constant sheaf in $\D^\b(Z_{w,0})$ which is pure of weight 0.
Let $i_w$ be the locally closed embedding of the stratum $Z_w\subset Z$. Set $\dim Z_w=d_w$.
Let $\k_w=\k_{Z_w}$ be the constant sheaf on $Z_w$.
We define the following objects in  $\D^\b(Z,S)$ 
\begin{align}\label{DNIC}
\Delta(w)=(i_w)_!\k_w[d_w]\quad,\quad
\nabla(w)=(i_w)_*\k_w[d_w]\quad,\quad
IC(w)=(i_w)_{!*}\k_w[d_w].
\end{align}
Fix a square root of the Tate sheaf.
For each $a\in\bbZ$ let  $(a/2)$ be the twist by the $a$th power of this square root.
We abbreviate 
\begin{align}\label{SHIFT}\langle \bullet\rangle=(\bullet/2)[\bullet].\end{align}
We define the following mixed complexes in $\D^\b_\bbd(Z_0,S)$ 
\begin{align}\label{DNICBBD}
\Delta(w)_\bbd=(i_w)_!\k_w\langle d_w\rangle\quad,\quad
\nabla(w)_\bbd=(i_w)_*\k_w\langle d_w\rangle\quad,\quad
IC(w)_\bbd=(i_w)_{!*}\k_w\langle d_w\rangle.
\end{align}
Let $\D^\b_{\lozenge,\bbd}(Z_0,S)\subset\D^\b_\bbd(Z_0,S)$ 
and
$\P\!_{\lozenge,\bbd}(Z_0,S)\subset\P\!_\bbd(Z_0,S)$ be the full triangulated
subcategory and the Serre subcategory generated by the set of objects
$\{IC(w)_\bbd(a/2)\,;\,w\in W\,,\,a\in\bbZ\}.$
The triangulated category
$\D^\b_{\lozenge,\bbd}(Z_0,S)$ has a t-structure whose heart is $\P\!_{\lozenge,\bbd}(Z_0,S)$.

\smallskip

\begin{definition}\label{def:even}The stratification $S$ is \hfill
\begin{itemize}[leftmargin=8mm]
\item[$\mathrm{(a)}$] 
\emph{affine} if each stratum is isomorphic to an affine space,
\item[$\mathrm{(b)}$] 
\emph{even affine} if 
\begin{itemize}[leftmargin=8mm]
\item[$\mathrm{(1)}$] $S$ is affine,
\item[$\mathrm{(2)}$]
$H^a((i_u)^*IC(v)_\bbd)=0$
for all $u,v\in W$, $a\in\bbZ$ with $a+d_v$ odd, 
\item[$\mathrm{(3)}$] 
$H^a((i_u)^*IC(v)_\bbd)$ 
is a sum of copies of
$\k_u(-a/2)$
if $a+d_v$ is even,
\end{itemize}
\item[$\mathrm{(c)}$]
\emph{even}  if 
\begin{itemize}[leftmargin=8mm]
\item[$\mathrm{(4)}$] there is an even affine stratification $T$ of $Z$ which refines $S$,
\item[$\mathrm{(5)}$] the strata of $S$ are connected and simply connected.
\end{itemize}
\item[$\mathrm{(d)}$]
\emph{good} if $S$ is even affine and satisfies the conditions \emph{\cite[\S 4.1(a)-(d)]{Y09}}.
\end{itemize}
\end{definition}

\smallskip

Since $IC(v)_\bbd$ is Verdier self-dual and $Z_u$ is smooth by (1), 
the conditions (2) and (3) are equivalent to the following conditions
\begin{itemize}[leftmargin=8mm]
\item[$\mathrm{(6)}$]
$H^a((i_u)^!IC(v)_\bbd)=0$ for all $u,v\in W$, $a\in\bbZ$ with $a+d_v$ odd, 
\item[$\mathrm{(7)}$] 
$H^a((i_u)^!IC(v)_\bbd)$ is a sum of copies of $\k_u(a/2)$ if $a+d_v$ is even.
\end{itemize}
The conditions (2), (6) imply that the complex $IC(v)$ is a \emph{parity sheaf} of $\D^\b(Z,S)$ in the sense of \cite{JMW}.
The conditions (3), (7) imply that the complex $IC(v)$ is \emph{very pure} in the sense of \cite[def.~3.1.2]{BY13}.
They tell us in addition that  the mixed vector spaces $H^a((i_u)^*IC(v)_\bbd)$ and $H^a((i_u)^!IC(v)_\bbd)$
are semisimple.

\smallskip

If the stratification $S$ is even, and $T$ is an even affine stratification which refines $S$,
then there is a full embedding of triangulated categories
$$\D^\b(Z,S)\subset \D^\b(Z,T)\quad,\quad
\D^\b_{\lozenge,\bbd}(Z_0,S)\subset \D^\b_{\lozenge,\bbd}(Z_0,T).$$
Since each stratum of $S$ contains a unique dense stratum of $T$,
there is a full embedding of additive categories
$$\P(Z,S)\subset\P(Z,T)\quad,\quad\P\!_{\lozenge,\bbd}(Z_0,S)\subset\P\!_{\lozenge,\bbd}(Z_0,T).$$

\smallskip

\begin{definition}\label{def:pure}
\hfill
\begin{itemize}[leftmargin=8mm]
\item[$\mathrm{(a)}$] 
$\Pure(Z_0,S)\subset\D^\b_{\lozenge,\bbd}(Z_0,S)$ is the full subcategory
of all mixed complexes which are isomorphic to finite direct sums of objects in
$\{IC(w)_\bbd\langle a\rangle\,;\,w\in W\,,\,a\in\bbZ\}.$
It is a graded additive category with the graded shift functor $\langle\bullet\rangle$ in \eqref{SHIFT}.

\item[$\mathrm{(b)}$] 
$\D^\b_\mix(Z,S)=\K^\b(\Pure(Z_0,S))$ as a graded triangulated category with the graded shift functor 
$\langle\bullet\rangle$  in \eqref{SHIFT}.

\item[$\mathrm{(c)}$] 
$\P\!_\mix(Z,S)\subset\P\!_{\lozenge,\bbd}(Z_0,S)$ is the full subcategory 
of all mixed perverse sheaves $\calE$ such that $\Gr^W_\bullet\calE$ is semisimple.
It is a graded Abelian category for the Tate shift functor $(\bullet/2)$.
\end{itemize}
\end{definition}

\smallskip

\begin{proposition} \label{prop:A}
Assume that the stratification S is even.
\hfill
\begin{itemize}[leftmargin=8mm]

\item[$\mathrm{(a)}$] 
$\D^\b_\mix(Z,S)$ has a t-structure and a triangulated $t$-exact faithful functor 
$$\iota:\D^\b_\mix(Z,S)\to\D^\b_{\lozenge,\bbd}(Z_0,S).$$
The heart of $\D^\b_\mix(Z,S)$
is equivalent to $\P\!_\mix(Z,S)$ and the restriction of $\iota$ to this heart is the full embedding
$\P\!_\mix(Z,S)\subset\P\!_{\lozenge,\bbd}(Z_0,S).$

\item[$\mathrm{(b)}$] 

$\zeta=\omega\circ\iota$ is a $t$-exact functor $\D^\b_\mix(Z,S)\to\D^\b(Z,S)$ such that
$$\bigoplus_{a\in\bbZ}\Hom_{\D^\b_\mix(Z)}(\calE\,,\,\calF(a/2))=
\Hom_{\D^\b(Z)}(\zeta\calE\,,\,\zeta\calF)\quad,\quad\forall\calE,\calF\in\D^\b_\mix(Z,S).$$

\item[$\mathrm{(c)}$] 
For any inclusion $h:Y\to Z$ of a union of strata of $S$,
the functors $h_*$, $h_!$, $h^*$, $h^!$ between  the categories $\D^\b_\bbd(Y_0,S)$ and $\D^\b_\bbd(Z_0,S)$
 lift to triangulated 
functors between the categories $\D^\b_\mix(Y,S)$ and $\D^\b_\mix(Z,S)$ which satisfy the usual 
adjointness properties.
\end{itemize}
\end{proposition}

\begin{proof}
Part (a) is proved in  \cite[\S 7.2]{AR13d}, \cite[prop.~7.5(1),(2)]{AR13d}.
Part (b) is proved in \cite[prop.~7.5(2)]{AR13d}.
Part (c) is \cite[thm.~9.5]{AR13d}.
\end{proof}

\smallskip

Assume that the stratification $S$ is even. 
We have the graded additive category $(\Par(Z,S)\,,\,\langle \bullet\rangle\,)$ such that
$\langle \bullet\rangle$ is as in \eqref{SHIFT0}, 
and the graded additive category $(\Pure(Z_0,S)\,,\,\langle\bullet\rangle\,)$ such that
$\langle \bullet\rangle$ is as in \eqref{SHIFT}.

\smallskip

\begin{proposition}\label{prop:pure1}Assume that the stratification $S$ is even.
\hfill
\begin{itemize}[leftmargin=8mm]
\item[$\mathrm{(a)}$] 
$\Hom_{\D^\b_\mu(Z)}^\bullet(\calE\,,\,\calF)=
\HHom^\bullet_{\D^\b(Z_0)}(\iota\calE\,,\,\iota\calF)^\Fr$
for each $\calE,\calF\in\D^\b_\mix(Z,S)$.
\item[$\mathrm{(b)}$] 
$\Hom_{\D^\b_\mu(Z)}^\bullet(\calE\,,\,\calF)=
\HHom^\bullet_{\D^\b(Z_0)}(\iota\calE\,,\,\iota\calF)$
for each $\calE,\calF\in\Pure(Z_0,S)$.
\item[$\mathrm{(c)}$] 
$\zeta:(\Pure(Z_0,S)\,,\,\langle\bullet\rangle\,)\to(\Par(Z,S)\,,\,\langle\bullet\rangle\,)$ is an equivalence of
graded additive categories.
\end{itemize}
\end{proposition}

\begin{proof}
Due to the full embedding $\D^\b(Z_0,S)\subset \D^\b(Z_0,T)$ for each affine refinement $T$ of $S$, 
we can assume that $S$ is even affine. 
Part (a) follows from \cite[lem.~7.8]{AR13d}, which also  implies that
the mixed complex $\HHom^\bullet_{\D^\b(Z_0)}(\iota\calE,\iota\calF)$ is semisimple
for each objects $\calE,\calF\in\D^\b_\mix(Z,S)$. 
Hence, for each $a,b\in\bbZ$, we have
\begin{align}\label{P1}
\Hom_{\D^\b_\mu(Z)}^a(\calE\,,\,\calF(b/2))=
(\HHom^a_{\D^\b(Z_0)}(\iota\calE\,,\,\iota\calF)(b/2))^\Fr,
\end{align}
and to prove (b) it is enough to check that $\HHom^\bullet_{\D^\b(Z_0)}(\iota\calE,\iota\calF)$  is pure of weight $0$ whenever $\calE,\calF\in\Pure(Z_0,S)$.
The mixed Abelian category $\P\!_\mix(Z,S)$ is Koszul by \cite[thm.~4.4.4]{BGS96}.
Hence, if $\calE,\calF\in\P\!_\mix(Z,S)$ are pure of weight zero, we have
\begin{align}\label{P2}
\begin{split}
b\neq a&\Longrightarrow\Hom_{\D^\b_\mu(Z,S)}^a(\calE\,,\,\calF(b/2))=0.
\end{split}
\end{align}
So, the mixed vector space $\HHom^a_{\D^\b(Z_0)}(\iota\calE,\iota\calF)$ is pure of weight $a$, 
so the mixed complex
$\HHom^\bullet_{\D^\b(Z_0)}(\iota\calE,\iota\calF)$ is pure of weight $0$, proving (b) because
any object of $\Pure(Z_0,S)$ is a sum of $IC(w)_\bbd\langle a\rangle$'s.
Part (c) follows from \eqref{P2} and Proposition \ref{prop:A}, 
since for any $\calE,\calF\in\Pure(Z_0,S)$ we have
\begin{align*}
\begin{split}
\Hom_{\D^\b_\mix(Z)}(\calE\,,\,\calF)
=\bigoplus_{a\in\bbZ}\Hom_{\D^\b_\mix(Z)}(\calE\,,\,\calF(a/2))
=\Hom_{\D^\b(Z)}(\zeta\calE,\zeta\calF).
\end{split}
\end{align*}
\end{proof}

\smallskip

For each $w\in W$, let $IC(w)_\mix$ be $IC(w)_\bbd$ viewed as an object of $\D^\b_\mix(Z,S)$.
Assume that the stratification $S$ is even.
By Proposition \ref{prop:pure1}, we have
$\D^\b_\mix(Z,S)=\K^\b(\Par(Z,S)).$ 
This identification takes $IC(w)_\mix$ to $IC(w)$.
The grading  $\K^\b(\Par(Z,S))$ is given by the shift functor $\langle \bullet\rangle$
on $\Par(Z,S)$ in \eqref{SHIFT0}.

\smallskip

We define the \emph{equivariant mixed category} of the stack $\calZ=[Z\,/\,G]$ by
$\D^\b_\mix(\calZ)=\K^\b(\Par(\calZ)).$
Let $S$ be the stratification by the $G$-orbits. We have the forgetful functor
$\For:\D^\b_\mix(\calZ)\to\D^\b_\mix(Z,S).$
We do not know any equivariant analogue of Proposition \ref{prop:pure1}.
However, the following holds, see e.g. \cite[lem.~3.1.5]{BY13}.

\smallskip

\begin{proposition} \label{prop:BY} Assume that the $G_0$-orbits in $Z_0$ are affine.
If $\calE,\calF\in\D^\b_\bbd(\calZ_0)$ are very pure of weight $0$, then
the mixed complex $\HHom^\bullet_{\D^\b(\calZ_0)}(\calE,\calF)$ in $\D^+(\Spec F_0)$ is pure of weight $0$
and it is free of finite rank as an $H^\bullet_G$-module.
\qed
\end{proposition}

\smallskip

\begin{remark}\label{rem:mixte2}
\hfill
\begin{itemize}[leftmargin=8mm]
\item[$\mathrm{(a)}$] 
Let $S$ be any stratification of $Z$.
The category $\Par(Z,S)$ has split idempotents, and
the Verdier duality $D$ yields an equivalence 
$\Par(Z,S)\simeq \Par(Z,S)^\op$. 
Let $\calL\in\Par(Z,S)$ be a graded-generator.
Set $\bfR=\End^\bullet_{\D^\b(Z,S)}(\calL)^\op.$
The functor
$\calE\mapsto\Hom^\bullet_{\D^\b(Z,S)}(\calL\,,\,\calE)$ gives an equivalence of graded additive categories
$\Par(Z,S)\to \bfR\text{-}\proj$.
Taking the homotopy categories, we get a graded triangulated equivalence
$\K^\b(\C(Z,S))\to\D^\perf(\bfR).$

\item[$\mathrm{(b)}$] 
If $h:Y\to Z$ is a closed embedding then the functor 
$h_!=h_*:\D^\b_\mix(Y)\to\D^\b_\mix(Z)$  
in Proposition \ref{prop:A} is given by restricting the functor
$h_!=h_*:\D^\b(Y)\to\D^\b(Z)$ to $\C(Y)\to\C(Z)$ and taking the homotopy categories.
If $h:Y\to Z$ is an open embedding then the functor $h^!=h^*$ is defined in a similar way.
We do not know any equivariant analogue of Proposition \ref{prop:A} which would yield
functors $h_*$, $h_!$, $h^*$, $h^!$ between the categories $\D^\b_\mix(\calY)$ and $\D^\b_\mix(\calZ)$ 
for any inclusion $h:\calY\to \calZ$ of a union of strata. However the functors
$h_*=h_!$ are well defined in the equivariant case if $h$ is a closed embedding,
so are $h^*=h^!$ if $h$ is an open embedding.

\item[$\mathrm{(c)}$] 
If the stratification $S$ is even then the set
$\{IC(w)[a]\,;\,w\in W\,,\,a\in \bbZ\}$ is a complete and irredundant set of indecomposable 
objects of $\Par(Z,S)$. It is also a complete and irredundant set of  parity sheaves of $\D^\b(Z,S)$
 in the sense of \cite{JMW}.
 
 \item[$\mathrm{(d)}$] 
A triangulated functor $\phi_\bbd:\D^\b_{\lozenge,\bbd}(Y_0)\to\D^\b_{\lozenge,\bbd}(Z_0)$  is \emph{geometric}
if there is a triangulated functor $\phi:\D^\b(Y)\to\D^\b(Z)$ with a natural isomorphism
$\phi\,\omega\Rightarrow \omega\,\phi_\bbd$. It is \emph{genuine} if it is geometric and there is a triangulated functor
$\phi_\mu:\D^\b_\mix(Y)\to\D^\b_\mix(Z)$ with a natural isomorphism
$\phi\,\iota\Rightarrow \iota\,\phi_\mix$.

\item[$\mathrm{(e)}$] 
Let $T$, $V$ be even affine refinements of even stratifications $S$, $U$ of $F_0$-schemes $Y_0$, $Z_0$.
By \cite[lem.~7.12]{AR13d}, there are full embedding of triangulated categories
$\D^\b_{\lozenge,\bbd}(Y_0,S)\subset\D^\b_{\lozenge,\bbd}(Y_0,T)$ and 
$\D^\b_{\lozenge,\bbd}(Z_0,U)\subset\D^\b_{\lozenge,\bbd}(Z_0,V)$.
By \cite[lem.~7.21]{AR13d}, the restriction of a genuine functor 
$\D^\b_{\lozenge,\bbd}(Y_0,T)\to\D^\b_{\lozenge,\bbd}(Z_0,V)$ that takes 
$\D^\b_{\lozenge,\bbd}(Y_0,S)$ into $\D^\b_{\lozenge,\bbd}(Z_0,U)$, is a genuine functor 
$\D^\b_{\lozenge,\bbd}(Y_0,S)\to\D^\b_{\lozenge,\bbd}(Z_0,U)$.

\item[$\mathrm{(f)}$] 
By Proposition \ref{prop:A}, if the stratification $S$ is even we may view $\D^\b_\mix(Z,S)$ as
a (non full) subcategory of $\D^\b_{\lozenge,\bbd}(Z_0,S)$ consisting of objects 
whose stalks carry a semisimple action of the Frobenius.

\item[$\mathrm{(g)}$] 
Any object $\calE$ of  $\Pure(Z_0,S)$ is semisimple and pure of weight 0, that is
$\calE\simeq\bigoplus_a{}^p\!H^a(\calE)[-a]$ where each mixed perverse sheaf ${}^p\!H^a(\calE)$
is pure of weight $a$.

\item[$\mathrm{(h)}$] 
An even stratification $S$ is called \emph{affable} in \cite[def.~7.2]{AR13d}. Then,
the category $\D^\b_{\lozenge,\bbd}(Z_0,S)$ is the same as
$\D^\text{Weil}_S(Z_0)$ in \cite[\S 6.1]{AR13d}. If $S$ is even affine, then $\D^\b_{\lozenge,\bbd}(Z_0,S)$ is the category 
$D_{\lozenge,m}(Z_0)$ in \cite[\S 2.1]{Y09}.

\end{itemize}

\end{remark}

\smallskip

\subsection{Even affine stratifications, projective and tilting objects}

Assume that the stratification $S$ is even affine.
The objects $\Delta(w)_\bbd$ and $\nabla(w)_\bbd$ have canonical lifts 
$\Delta(w)_\mu$ and $\nabla(w)_\mu$ in $\P\!_\mix(Z,S)$
by Proposition \ref{prop:A}.
We have 
$$\iota IC(w)_\mix=IC(w)_\bbd\quad,\quad
\iota \Delta(w)_\mix=\Delta(w)_\bbd\quad,\quad
\iota \nabla(w)_\mix=\nabla(w)_\bbd\quad,\quad
\forall w\in W.$$
We equip the triangulated category $\D^\b(\P\!_\mix(Z,S))$ with the grading shift functors
\begin{align}\label{SHIFT3}\langle \bullet\rangle=(\bullet/2)[\bullet],\end{align}
where $(\bullet/2)$ is the Tate shift functor on $\P\!_\mix(Z,S)$ and $[\bullet]$ is the cohomological shift.
By \cite[cor.~7.10]{AR13d}, there is an equivalence of graded triangulated categories
$\D^\b(\P\!_\mix(Z,S))\to\D^\b_\mix(Z,S)$
which  identifies the grading shift functors \eqref{SHIFT} and \eqref{SHIFT3}.
We'll use two refinements of this equivalence which involve projective and tilting objects of $\P\!_\mix(Z,S)$.

\smallskip

\begin{proposition}\label{prop:AA}
Assume that the stratification S is even affine.
\hfill
\begin{itemize}[leftmargin=8mm]
\item[$\mathrm{(a)}$] 
$\P\!_\mix(Z,S)$, $\P(Z,S)$ have enough projectives and finite cohomological dimension.
The sets of indecomposable objects in $\P(Z,S)^\proj$ and $\P\!_\mix(Z,S)^\proj$ are
$\{P(w)\,;\,w\in W\}$ and $\{P(w)_\mix(a/2)\,;\,w\in W\,,\,a\in\bbZ\},$
where  $P(w)$, $P(w)_\mu$ are the projective covers of $\Delta(w)$, $\Delta(w)_\mix$ in 
$\P(Z,S)$, $\P\!_\mix(Z,S)$.
\item[$\mathrm{(b)}$] 
$\P\!_\mix(Z,S)^\proj=\zeta^{-1}(\P(Z,S)^\proj)$.
\item[$\mathrm{(c)}$] 
$\P\!_\mix(Z,S)^\proj\subset\D^\b_\mix(Z,S)$
extends to a graded triangulated equivalence 
$\K^\b(\P\!_\mix(Z,S)^\proj)\to\D^\b_\mix(Z,S).$
\end{itemize}
\end{proposition}

\begin{proof}
Part (a) is \cite[prop.~7.7(1),(2)]{AR13d},
part (b) is \cite[prop.~7.7(2)]{AR13d}, and
part (c) is proved \cite[cor.~7.10, prop.~7.11]{AR13d}.
\end{proof}

\smallskip

\smallskip

\begin{definition}[\cite{BBM04}, \cite{Y09}]
Assume that the stratification $S$ is even affine. 
A mixed perverse sheaf $\calE\in\P\!_{\lozenge,\bbd}(Z_0,S)$
is \emph{tilting} if either of the following equivalent conditions hold 
\hfill
\begin{itemize}[leftmargin=8mm]
\item[$\mathrm{(a)}$] 
$(i_w)^*\calE$ and $(i_w)^!\calE$ are perverse for each $w\in W$.
\item[$\mathrm{(b)}$] 
$\calE$ has both a filtration by $\Delta(w)_\bbd(a/2)$'s and by $\nabla(w)_\bbd(a/2)$'s,
with $w\in W$ and $a\in\bbZ$. 
\end{itemize}
We define a tilting object in $\P(Z,S)$ in a similar way.
\end{definition}

\smallskip

Let $\P(Z,S)^\tilt\subset\P(Z,S)$ and
$\P\!_{\lozenge,\bbd}(Z_0,S)^\tilt\subset\P\!_{\lozenge,\bbd}(Z_0,S)$
be the full additive subcategories of tilting objects. Let  
$\P\!_\mix(Z,S)^\tilt\subset\P\!_\mix(Z,S)$ be the full subcategory whose objects are the complexes which map
to $\P\!_{\lozenge,\bbd}(Z_0,S)^\tilt$ by the functor $\iota$.

\smallskip

\begin{proposition}\label{prop:B} Assume that the stratification $S$ is  even affine.
\hfill
\begin{itemize}[leftmargin=8mm]
\item[$\mathrm{(a)}$] 
For each  $w\in W,$ there are unique indecomposable objects
$T(w)$, $T(w)_\mix$ in $\P(Z,S)^\tilt$, $\P\!_\mix(Z,S)^\tilt$ supported on $\overline X_w$ 
whose restriction to $X_w$ are $\k_w[d_w]$, $\k_w\langle d_w\rangle$ respectively. 
The sets of indecomposable objects in $\P(Z,S)^\tilt$, $\P\!_\mix(Z,S)^\tilt$ are
$\{T(w)\,;\,w\in W\}$ and $\{T(w)_\mix(a/2)\,;\,w\in W\,,\,a\in\bbZ\}.$

\item[$\mathrm{(b)}$] 
Assume that the stratification $S$ is good.
Then, we have $\P\!_\mix(Z,S)^\tilt=\P\!_{\lozenge,\bbd}(Z_0,S)^\tilt$. 

\item[$\mathrm{(c)}$] 
$\P\!_\mix(Z,S)^\tilt=\zeta^{-1}(\P(Z,S)^\tilt)$.

\item[$\mathrm{(d)}$] 
$\P\!_\mix(Z,S)^\tilt\subset\D^\b_\mix(Z,S)$
extends to a graded triangulated equivalence
$\K^\b(\P\!_\mix(Z,S)^\tilt)\to\D^\b_\mix(Z,S).$
\end{itemize}
\end{proposition}

\begin{proof}
Part (a) is  \cite[prop.~10.3]{AR13d}, \cite{BBM04}.
Part (b) is \cite[\S 4.2]{Y09}.
Part (c) is obvious.
Part (d) is \cite[prop.~10.5]{AR13d}.
\end{proof}


\bigskip

\bigskip




\begin{thebibliography}{20}
\bibitem{A14} Achar, P. N., Perverse coherent sheaves on the nilpotent cone in good characteristic,
in Recent Developments in Lie Algebras, Groups and Representation Theory, Proc.
Sympos. Pure Math., 86, 1-23. Amer. Math. Soc., Providence, RI, 2012.
\bibitem{AR13d} Achar, P.N., Riche, S., Koszul duality and semisimplicity of Frobenius, Ann. Inst. Foutier, Grenoble 63 
(2013), 1511-1612.
\bibitem{AR16} Achar, P.N., Riche, S., Modular perverse sheaves on flag varieties II: Koszul duality and formality, 
Duke Math. J. 165 (2016), 161-215.
\bibitem{A15d} Achar, P.N., Rider, L., 
Parity sheaves on the affine Grassmannian and the Mirkovi\'c-Vilonen conjecture, Acta Math. 215 (2015), 183-216.
\bibitem{AR16a} Achar, P.N., Rider, L., The affine Grassmannian and the Springer resolution in positive characteristic, 
Compositio Math. 152 (2016), 2627-2677.
\bibitem{AJS} Andersen, H.H., Jantzen, J.C., Soergel, W., Representations of Quantum Groups at a $p$th 
Root of Unity and of Semisimple Groups in Characteristic $p$: Independence of $p$. Ast\'erisque,  220 (1994), pp. 321.
\bibitem{BBD}  Beilinson, A.A., Bernstein, J., Deligne, P., Faisceaux pervers, Analysis and topology on singular spaces, I 
(Luminy, 1981), 5-171, Ast\'erisque, 100, Soc. Math. France, Paris, 1982.
\bibitem{BGS96} Beilinson, A.A., Ginzburg, V., Soergel, W.,
Koszul duality patterns in representation theory,
J. Amer. Math. Soc. 9 (1996), 473-527. 
\bibitem{BBM04} Beilinson, A., Bezrukavnikov, R.,  Mirkovic, I.,Tilting exercises,
Mosc. Math. J. 4 (2004), 547-557.
\bibitem{BL} Bernstein, J., Lunts, V., Equivariant sheaves and functors, Lecture Notes in Mathematics 1578, 
Springer-Verlag, 1994.
\bibitem{B03} Bezrukavnikov, R., Quasi-exceptional sets and equivariant coherent sheaves on the nilpotent cone, Represent. Theory, 7 (2003), 1-18.
\bibitem{B16}  Bezrukavnikov, R., On two geometric realizations of an affine Hecke algebra, 
Publ. Math. Inst. Hautes \'Etudes Sci. 123 (2016), 1-67
\bibitem{BFM05} Bezrukavnikov, R., Finkelberg, M., Mirkovic, I., Equivariant homology and K-theory of affine Grassmannians and Toda lattices, Compos. Math.141 (2005), 746-768.
\bibitem{BY13}  Bezrukavnikov, R., Yun, Z., 
On Koszul duality for Kac-Moody groups, Represent. Theory 17 (2013), 1-98.
\bibitem{B12} Bridgeland, T.,  An introduction to motivic Hall algebras, Adv. Math. 229 (2012), 102-138.
\bibitem{CW18} Cautis, S., Williams, H., Cluster theory of the coherent Satake category, arXiv:1801.08111.
\bibitem{EEK99} Edelman, A., Elmroth, E., Kagstr\"om, B.,
A geometric appraoch to perturbation theory of matrices and matrix pencils, part II:
a stratification-enhanced staricase algorithme, SIAM J. Matrix Anal. Appl. 20 (1999), 667-699.
\bibitem{FF18} Finkelberg, M., Fujita, R, Coherent IC-sheaves on type $A_n$ affine Grassmannians and dual canonical 
basis of affine type $A_1$, arXiv:1901.05994.
\bibitem{G91} Ginzburg, V., Perverse sheaves and $\bbC^*$-actions, J. Amer. Math. Soc. 4 (1991), 483-490.
\bibitem{JMW} Juteau, D., Mautner, C., Williamson, G., Parity sheaves. J. Amer. Math. Soc. 27, 1169-
1212 (2014).
\bibitem{KKK18} Kang, S.-J., Kashiwara, M., Kim, M., Symmetric quiver Hecke algebras and $R$-matrices of 
quantum affine algebras, Invent. Math. 211 (2018), 591-685.
\bibitem{KKKO18} Kang, S.-J., Kashiwara, M. , Kim, M., Oh, S.-J.,
Monoidal  categorification  of  cluster  algebras, J. Amer. Math. Soc. 31(2018), 349-426.
\bibitem{KKOP18} Kashiwara, M., Kim, M., Oh, S.-J., Park, E., 
Monoidal categories associated with strata of flag manifolds, Adv. Math. 328 (2018), 959-1009.
\bibitem{KKOP19} Kashiwara, M., Kim, M., Oh, S.-J., Park, E., 
Localizations for quiver Hecke algebras, arXiv:1901.09319.
\bibitem{K89} Kashiwara, M., The flag manifold of Kac-Moody Lie algebra, in: Algebraic Analysis,
Geometry, and Number Theory (Baltimore, MD, 1988), Johns Hopkins Univ. Press,
Baltimore, MD, 1989, 161-190.
\bibitem{KS09} Kashiwara, M., Shimozono, M., Equivariant K-theory of affine flag manifolds and
affine Grothendieck polynomials. Duke Math. J. 148, 501-538 (2009)
\bibitem{KT02} Kashiwara, M., Tanisaki, T., 
Parabolic Kazhdan-Lusztig polynomials and Schubert varieties, J. Algebra 249 (2002), 306-325. 
\bibitem{K12} Kato, S., An algebraic study of extension algebras,  Amer. J. Math. 139 (2017), 567-615.
\bibitem{K15b} Kleshchev, A.,  Affine highest weight categories and affine quasihereditary algebras, 
Proc. Lond. Math. Soc. 110 (2015), 841-882.
\bibitem{K15c} Kleshchev, A., Muth, R., Stratifying KLR algebras of affine ADE types, J. Algebra 475 (2017), 133-170.
\bibitem{K17} Kumar, S., Positivity in T-equivariant K-theory of
flag varieties associated to Kac-Moody groups,
J. Eur. Math. Soc. 19 (2017), 2469-2519.
\bibitem{KW} Kiehl, R., Weissauer, R.,
Weil conjectures, perverse sheaves and l'adic Fourier transform.
Ergebnisse der Mathematik und ihrer Grenzgebiete. 3. Folge. A Series of Modern Surveys in Mathematics, 42. Springer-Verlag, Berlin, 2001.
\bibitem{L09} Lafforgue, V., Quelques calculs reli\'es \`a la correspondance de Langlands g\'eom\'etrique pour $\bbP^1$,
unpublished.
\bibitem{LO} Laszlo, Y., Olsson, M., Perverse t-structure on Artin stacks, Math. Z. 261 (2009), 737-748.
\bibitem{La} Laumon, G., Faisceaux automorphes li\'es aux s\'eries d'Eisenstein. (French) 
[Automorphic sheaves associated with Eisenstein series] 
Automorphic forms, Shimura varieties, and L-functions, Vol. I (Ann Arbor, MI, 1988), 227-281, 
Perspect. Math., 10, Academic Press, Boston, MA, 1990.
\bibitem{L92}  Lusztig, G., Affine quivers and canonical bases, Publ. Math. Inst. Hautes \'Etudes  Sci. 76 (1992), 111-162.
\bibitem{L98} Lusztig, G., Canonical bases and Hall algebras, A. Broer and A. Daigneault(eds.), Representation Theories
and Algebraic Geometry, 365-399, 1998 Kluwer Academic Publishers.
\bibitem {L} Lusztig, G., Introduction to quantum groups. Reprint of the 1994 edition. Modern Birkh\"auser Classics. 
Birkh\"auser/Springer, New York, 2010.
\bibitem{MR16} Mautner, C., Riche, S., On the exotic t-structure in positive characteristic,
Int. Math. Res. Notices 18 (2016), 5727-5774.
\bibitem{Mc18} McNamara, P., Representations of Khovanov-Lauda-Rouquier algebras III: symmetric affine type,
Math. Z. (2017), 243-286.
\bibitem{M13} Minn-Thu-Aye, M., Multipolicity formulas for Perverse Coherent Sheaves on the Nilpotent Cone, Ph. D. 
Thesis, Louisiana State University, Baton Rouge, LA, 2013.
\bibitem{R17} Riche, S., Kostant section, universal centralizer, and a modular derived Satake equivalence, 
Math. Z. 286 (2017), 223-261.
\bibitem{RSVV} Rouquier, R., Shan, P., Varagnolo, M., Vasserot, E., 
Coherent categorification of quantum loop algebras : the monoidality (in preparation).
\bibitem{SV12} Schiffmann, O., Vasserot, E., Hall algebras of curves, commuting varieties and Langlands duality, Math.Ann. 353 (2012), 1399-1451.
\bibitem{SV13} Schiffmann, O., Vasserot, E., The elliptic Hall algebra and the K-theory of the Hilbert scheme of $\mathbb{A}^2$, DukeMath. J.162 (2013), 279-366.
\bibitem{SV} Schiffmann, O., Vasserot, E., On cohomological Hall algebras of quivers : Generators, J. reine angew. Math.
(to appear), arXiv:12705.07488.
\bibitem{S12}  Sun, S., Decomposition theorem for perverse sheaves on Artin stacks over finite fields,
Duke Math. J. 161 (2012), 2297-2310.
\bibitem{VV09} Varagnolo, V., Vasserot, E., Finite-dimensional representations of DAHA and affine Springer 
fibers : the spherical case, Duke Math. J. 147 (2009), 439-540.
 \bibitem{VV11} Varagnolo, V., Vasserot, E., Canonical bases and KLR-algebras. J. Reine Angew. Math. 659 (2011), 67-100. 
 \bibitem {Y09} Yun, Z., Weights of mixed tilting sheaves and geometric Ringel duality. Selecta Math. (N.S.) 14 (2009), 299-320.
 \bibitem{Z16}  Zhu, X., An introduction to affine Grassmannians and the geometric Satake equivalence. Geometry of moduli spaces and representation theory, 59-154, IAS/Park City Math. Ser., 24, Amer. Math. Soc., Providence, RI, 2017.

\end{thebibliography}
\end{document}